\documentclass{article}

\usepackage{arxiv}
\usepackage{amsthm}

\usepackage{array}
\usepackage[numbers]{natbib}
\usepackage{graphicx}
\usepackage{epstopdf}
\usepackage{multirow}
\usepackage{amsmath,amsfonts,amssymb}
\usepackage{bbm}
\usepackage{subfig}
\usepackage[misc]{ifsym}
\usepackage{color}
\usepackage{xcolor}
\usepackage{comment}
\usepackage{tablefootnote}
\usepackage{float}
\usepackage{enumitem}
\usepackage{mycolor}
\usepackage[normalem]{ulem}
\usepackage{listings}
\usepackage[overload]{empheq}
\usepackage{pgfplots}

\usepackage{mathrsfs}
\usepackage{dsfont}
\usepackage{bm}
\usepackage{xfrac}

\usepackage[colorlinks=true,breaklinks=true,bookmarks=true,urlcolor=links,citecolor=links,linkcolor=links,bookmarksopen=false,draft=false]{hyperref}

\newtheorem{theorem}{Theorem}[section]
\newtheorem{corollary}{Corollary}[section]
\newtheorem{lemma}{Lemma}[section]
\newtheorem{proposition}{Proposition}[section]
\newtheorem{remark}{Remark}[section]
\newtheorem{example}{Example}[section]
\newtheorem{definition}{Definition}[section]

\newskip\ginf
\newskip\gzro
\RequirePackage{array}
\newcolumntype{f}{@{\extracolsep{\ginf}}}
\newcolumntype{z}{@{\extracolsep{\gzro}}}
\newcolumntype{n}{@{}}
\newcolumntype{E}{nROLn}
\newcolumntype{L}{>{$\displaystyle\bgroup}l<{\egroup$}}
\newcolumntype{R}{>{$\displaystyle\bgroup}r<{\egroup$}}
\newcolumntype{C}{>{$\displaystyle\bgroup}c<{\egroup$}}
\newcolumntype{I}{>{$\displaystyle\bgroup}r<{{}\egroup$}}
\newcolumntype{o}{@{}>{{}}c<{{}}@{}}
\newcolumntype{O}{@{}>{$\displaystyle\bgroup{}}c<{{}\egroup$}@{}}
\newcolumntype{x}[1]{nlfRz #1 fln}
\newcolumntype{w}{>{\llap\bgroup}r<{\egroup}}
\newcolumntype{K}{LzR}

\def\conv{\mathop{\textup{conv}}}
\def\epi{\mathop{\textup{epi}}}
\def\cl{\mathop{\textup{cl}}}
\def\card{\mathop{\textup{card}}}
\def\diag{\mathop{\textup{diag}}}
\def\trc{\mathop{\textup{trace}}}
\def\tra{^{\intercal}}
\def\st{\mathop{\textup{s.t.}}}
\def\proj{\mathop{\textup{proj}}}
\def\one{\mathbbm{1}}
\def\cnorm{\mathop{N^K_{\|\cdot\|_s}}}
\def\ext{\mathop{\textup{vert}}}
\def\interior{\mathop{\textup{int}}}
\def\Ball#1#2{B_{#1}(#2)}
\def\Halmos{}

\def\ie{{\it i.e.}\/}
\def\tr{^{\intercal}}

\title{Convexification of Permutation-Invariant Sets and an Application to Sparse PCA}
\author{Jinhak Kim\\
College of Business\\ Northern Illinois University\\
\texttt{jkim23@niu.edu}
\And 
Mohit Tawarmalani\\
Krannert School of Management\\ Purdue University\\ \texttt{mtawarma@purdue.edu}
\And
Jean-Philippe P. Richard\\
Industrial and Systems Engineering\\
University of Minnesota\\ \texttt{jrichar@umn.edu}}
\date{August 6, 2021}

\begin{document}
\maketitle





\begin{abstract}
We develop techniques to convexify a set that is invariant under permutation and/or change of sign of variables and discuss applications of these results. First, we convexify the intersection of the unit ball of a permutation and sign-invariant norm with a cardinality constraint.  This gives a nonlinear formulation for the feasible set of sparse principal component analysis (sparse PCA) and an alternative proof of the $K$-support norm.
Second, we characterize the convex hull of sets of matrices defined by constraining their singular values. As a consequence, we generalize an earlier result that characterizes the convex hull of rank-constrained matrices whose spectral norm is below a given threshold. 
Third, we derive convex and concave envelopes of various permutation-invariant nonlinear functions and their level-sets over hypercubes, with congruent bounds on all variables. Finally, we develop new relaxations for the exterior product of sparse vectors. 
Using these relaxations for sparse PCA, we show that our relaxation closes $98\%$ of the gap left by a classical SDP relaxation for instances where the covariance matrices are of dimension up to $50\times 50$.

\end{abstract}%


\keywords{Convexification, permutation-invariant sets, majorization, sparse PCA}

\maketitle

\section{Introduction}\label{sec:intro}

This paper is concerned with convexification of permutation-invariant sets. 
A set $S\subseteq \mathbb{R}^n$ is \emph{permutation-invariant} if $x\in S$ implies that $Px\in S$ for all $n$-dimensional permutation matrices $P$.
Permutation-invariant sets appear in a variety of optimization problems. To support this claim and highlight the relevance of our construction, we next provide a variety of example applications where exploiting permutation invariance proves fruitful. 
Consider first 
sparse principal component analysis, a problem first introduced in \cite{zou2006sparse}, that consists in finding sparse vectors that explain the most variance in a data set. 
Specifically, the problem of finding the first sparse principal component can be formulated as $\max\{x^\intercal\Sigma x\mid x \in S\}$  where $S=\{x \in \Re^n \mid \card(x)\le K, \|x\|\le 1\}$, $\card(x)$ is the number of nonzero components of $x$, and $\Sigma$ is the covariance matrix of the given data; see, for example, \cite{d2007direct}.
The feasible set of this model is permutation-invariant because 
$\card(Px)=\card(x)$ and $\|Px\|=\|x\|$ for any vector $x \in \Re^n$ and  permutation matrix $P$. 
The convex hull of $S$ is the unit ball associated with the $K$-support norm \cite{argyriou2012sparse}, a result that yields a tighter approximation of $S$ compared to the elastic net $\{x\in\mathbb{R}^n\mid \|x\|_1\le \sqrt{K}, \|x\|\le 1\}$ where $\|\cdot\|_1$ is the $l_1$-norm.
Recognizing the set as permutation-invariant allows for a streamlined derivation of these results.
Second, observe that permutation-invariance also arises when studying sets of matrices as the rank of a matrix can be thought of as a permutation-invariant cardinality constraint on the singular values. 
This equivalent formulation suggests that the applicability of this concept extends to sets of matrices whose singular values belong to a permutation-invariant set. 
For example, \cite{hiriart2012}, considers the set $\{M\in\mathcal{M}^{m,n}(\mathbb{R})\mid\textup{rank}(M)\le K, \|M\|_{sp}\le r\}$ where $\mathcal{M}^{m,n}(\mathbb{R})$ is the set of $m$-by-$n$ real matrices and $\|M\|_{sp}$ is the spectral norm of $M$. Since the spectral norm is the largest singular value, it follows easily that the singular values of these matrices belong to a permutation-invariant set. 
Using a permutation-invariant perspective to study such sets helps to generalize the results of \cite{hiriart2012} in various ways. 
Finally, observe that a variety of sets with specific structures have been studied in nonconvex optimization because of their uses in creating convex relaxations of more general problems. 
For example, specialized techniques have been used to identify the convex hull of the multilinear monomial $\prod_{i=1}^n x_i$ over $[0,1]^n$ and $[-1,1]^n$ \cite{rikun1997,lnl2012}. 
We will show that these results can be obtained as special cases of our general approach to convexifying permutation-invariant sets. 
Besides, we will show that we can obtain hitherto unknown polynomial-sized convex hull descriptions for various commonly occurring sets in global optimization.


\def\B{\mathbb{B}}
We now describe why permutation-invariance is a useful property to consider while constructing convex hulls of sets.
To see this, let 
$\Delta_\pi := \{x\mid x_{\pi(1)}\ge\dots \ge x_{\pi(n)}\}$, where $\pi$ is an $n$-dimensional permutation and denote by $\conv(S)$, the convex hull of a given set $S$.
Since $S$ can be expressed as a union of $n!$ sets of the form $S\cap \Delta_\pi$, if there is a (possibly lifted) polynomial-sized representation of $\conv(S\cap \Delta_\pi)$ for each $\pi$, disjunctive programming~\cite{balas2018disjunctive} can be used to represent $\conv(S)$ in a higher-dimensional space. Unfortunately, this representation is exponentially-sized in $n$, because the construction creates copies, one for each $\pi$, for each of the variables $(x_1,\ldots,x_n)$. What is remarkable about the permutation-invariance is that, by exploiting this property, we will show that a much more compact formulation can be derived. To appreciate the significance of this construction, we first argue that, if $S$ is not permutation-invariant, such a compact formulation for $\conv(S)$ may not even exist.
To illustrate this point, we consider the set, $\B$, which is a face of the boolean-quadric polytope, and is defined as $\B :=\{xx^\intercal\mid x\in \{0,1\}^n, x_1=1\}$. If we order all the $\frac{n(n+1)}{2}$ products, these imply an order for all $x_j$ variables since $x_1x_j = x_j$ for $j=2,\ldots,n$. Moreover, if $x_j\le x_i$, it follows that $x_jx_i = x_j$. Therefore, the ordering of the $\frac{n(n+1)}{2}$ bilinear terms reduces to that of all $x_j$ variables, possibly with some equalities. In other words, it suffices to consider $\B\cap \Delta'_\pi$, where $\Delta'_\pi = \{xx^\intercal\mid x_{\pi(r)}x_{\pi(i)}\ge x_{\pi(k)}x_{\pi(j)} \mbox{ whenever } \max\{r,i\} \le \max\{k,j\}\}$ is a lifting of $\Delta_\pi$ and $\pi(1)=1$. Since $x\in\{0,1\}^n \cap \Delta_{\pi}$ 
implies that $x_{\pi(r)}x_{\pi(i)} = x_{\pi(\max\{r,i\})}$, we can write $\conv(\B\cap \Delta'_\pi) = \bigl\{X\in [0,1]^{n\times n}\mid X_{ij} = X_{1,\max\{i,j\}}, X_{1\cdot}\in \Delta_\pi, X_{11}=1\bigr\}$.
However, it is known that $\conv(\B)$ does not have a polynomial-sized formulation~\cite{fmptw15}. 
In contrast, treating permutation-invariant sets $S$ in this way provides a significant advantage since the sets $S\cap \Delta_\pi$ are congruent to one another. 
Exploiting this fact, we show that it is possible to construct a polynomial-sized extended formulation for $S$ whenever a polynomial-sized formulation exists for $\conv(S\cap \Delta_\pi)$. 
Our construction makes use of well-known extended formulations of a permutahedron along-side the convex hull of $S\cap \Delta_\pi$.
The outline of the construction is as follows: first, we consider a permutation-invariant set $S$ and assume that its convex hull over 
\begin{equation}
\Delta^n:=\{x\in\mathbb{R}^n\mid x_1\ge \dots\ge x_n\},
\label{eq:simplex_n}
\end{equation}
has a polynomial description.
Then, the convex hull is simply the union of permutahedra where each permutahedron is generated by a point in $\conv(S\cap \Delta^n)$.
Each permutahedron can then be modeled using a polynomial number of linear equalities and inequalities to provide an extended formulation for $\conv(S)$.
The techniques involved apply to other settings. 
For example, they can be used to obtain convex hulls of sign-invariant sets using convex hull representations of $S\cap \{x\mid x\ge 0\}$.


The remainder of the paper is organized as follows.
We present basic convexification results for permutation- and/or sign-invariant sets in Section~\ref{sec:mainresults}. 
We then explore various applications of the results in the ensuing sections.
In Section~\ref{sec:sparsity}, we derive the convex hull of the intersection of a unit ball associated with a permutation-invariant norm and a cardinality constraint.
The resulting convex hull defines another norm for which we give an explicit formula.
As a result, we show that it is simple to determine whether an arbitrary point belongs to the convex hull, and  to construct a separating hyperplane when it does not.
We study the connection between permutation-invariant sets and sets of matrices characterized by their singular values.
Furthermore, we investigate the semidefinite-representability of rank-constrained sets of matrices. 

In Section~\ref{sec:nonlinear}, we develop convex and concave envelope characterizations of various permutation-invariant functions and sets described using such functions. 
For example, we derive the convex hull of the lower level-set of a Schur-concave function, which is convex when all but one variable are fixed, a lifted representation of the convex hull of the graph of $\prod_{i=1}^n x_i$ over $[a,b]^n$, where $a,b\in \Re$,
and the convex hull of $\prod_{i=1}^m y_i^\alpha \ge \prod_{i=1}^n x_j^\beta$ over $[c,d]^m\times [a,b]^n$ with $a,c\ge 0$ and $\alpha, \beta > 0$. 
We show numerically that, for general $a$ and $b$, the convex hull of $\prod_{i=1}^n x_i$ over $[a,b]^n$ is much tighter than the widely-used recursive McCormick relaxation, in contrast to the well-known fact that the recursive McCormick procedure yields the convex hull of $\prod_{i=1}^n x_i$ when $a=-1$ and $b=1$ \cite{lnl2012}.

In Section~\ref{sec:rankone}, we study the set of rank-one matrices whose generating vectors lie in a permutation-invariant set. 
We construct semidefinite programming relaxations of the convex hull by proposing various valid inequalities derived from the rank-one condition of the matrix and the fact that every extreme point of a permutahedron generated by a vector is a permutation of the generating vector.
Finally,  we perform computational experiments with our relaxation for sparse PCA on several instances taken from the literature and other randomly generated instances. 
We compare our results to the relaxations proposed by \cite{d2007direct} and \cite{bertsimas2020solving}.

To increase the readability of the paper, we include a table of some frequently used notations in the appendix.

\section{Convex hull of permutation-invariant and sign-invariant sets}\label{sec:mainresults}

\def\R{\mathbb{R}}
In this section, we show that the convex hulls of permutation-invariant and sign-invariant sets can be readily constructed if their convex hulls over a fundamental sub-domain are known. 
Given a set $S$, we use the notation $\text{int}(S)$ to represent its interior, $\ext(S)$ to denote the set of its extreme points, and $\conv(S)$ to represent its convex hull. 

For a positive integer $k$, we denote the set of $k$-by-$k$ permutation matrices by $\mathcal{P}_k$.
Given a positive integer $n$ and a nonnegative integer $p$, a set $S\subseteq \{(x,z)\in\mathbb{R}^n\times\mathbb{R}^p\}$ is called \emph{permutation-invariant with respect to $x$ if $(x,z)\in S$} implies that $(Px,z)\in S$ for all permutation matrices $P \in \mathcal{P}_n$.
When $S\subseteq\{x\in\R^n\}$ is permutation-invariant with respect to $x$, we simply say that $S$ is \emph{permutation-invariant}.
A real-valued function $(x,z)\mapsto f(x,z)$ with $(x,z)\in\Re^n\times\Re^p$ is called \emph{permutation-invariant with respect to $x$} if $f(x,z)=f(Px,z)$ for all permutation matrices $P \in \mathcal{P}_n$.
When $f:x\mapsto f(x)$ is permutation-invariant with respect to $x$, we say that $f$ is \emph{permutation-invariant}.
Moreover, any permutation-invariant set $S$ can be written as the lower-level set $S=\{(x,z)\mid f(x,z)\le 0\}$ of a permutation-invariant function $f$ by choosing this function to be the indicator function of the $S$, which takes value $0$ on the set and $\infty$ otherwise.

A set $S\subseteq \{(x,z)\in\mathbb{R}^n\times\mathbb{R}^p\}$ where $n$ is a positive integer and $p$ is a nonnegative integer is called \emph{sign-invariant} with respect to $x$ if $(x,z)\in S$ implies that $(\bar{x},z)\in S$ for all $\bar{x}$ that satisfy $|\bar{x}|=|x|$.

Lemma~\ref{res:lin_invariant} describes an important property of the convex hull of sets that are closed under certain linear transformations of the coordinates of their elements. 

\begin{lemma}\label{res:lin_invariant}
Let $T\in \mathcal{M}^{n,n}(\mathbb{R})$ and let $S\subseteq\R^n$ be such that for each $x\in S$, $Tx\in S$ as well. Then, if $x\in \conv(S)$, $Tx\in \conv(S)$.
\end{lemma}
\begin{proof}{Proof.}
The result follows because $TS\subseteq S$ implies that $T\conv(S) = \conv(TS) \subseteq \conv(S)$.\hfill \Halmos
\end{proof}

It follows from Lemma~\ref{res:lin_invariant} that if $S$ is permutation-invariant (\emph{resp.} sign-invariant) then $\conv(S)$ is also permutation-invariant (\emph{resp.} sign-invariant).

For each $x\in\mathbb{R}^n$, we denote the $i^{\small \textrm{th}}$ largest component of $x$ by $x_{[i]}$ for $i=1,\dots,n$.

\begin{definition}
\label{def:maj}
Given two vectors $x,y\in\mathbb{R}^n$, we say that $x$ \emph{majorizes} $y$, a property we denote by $x\ge_m y$, if 
$\sum_{i=1}^j x_{[i]}\ge \sum_{i=1}^j y_{[i]}$ for $j=1,\dots,n-1$, and 
$\sum_{i=1}^n x_{[i]}= \sum_{i=1}^n y_{[i]}$.
We say that $y$ \emph{is weakly submajorized by $x$}, and denote this relation as  $x\ge_{wm} y$ if 
$
\sum_{i=1}^j x_{[i]}\ge \sum_{i=1}^j y_{[i]}, \forall j=1,\dots,n
$. For simplicity of notation, we shall refer to this relation as \emph{weak majorization} of $y$ by $x$.
\end{definition}

The result of Lemma~\ref{lemma:maj_conv_comb} relates majorization and permutation. 
Its proof follows from combining Hardy, Littlewood, and P\'{o}lya's theorem with Birkhoff's theorem; see 2.B.2 and 2.A.2 in \cite{marshall2010}.

\begin{lemma}[\cite{rado1952inequality}, Corollary 2.B.3 of \cite{marshall2010}]
\label{lemma:maj_conv_comb}
For $x$, $y \in \Re^n$, $x\ge_m y$ if and only if $y$ is a convex combination of $x$ and its permutations. \hfill  \Halmos
\end{lemma}

An extension of majorization, which is known as $G$-majorization, introduced by Rado \cite{rado1952inequality}, in the context of a group of transformations, is defined using the property in Lemma~\ref{lemma:maj_conv_comb} as the set of all doubly stochastic matrices from a semigroup; 
see 14.C in \cite{marshall2010} for more detail and references about $G$-majorization.


\begin{lemma}\label{res:maj_convex}
Let $K$ be a convex subset of $\R^n\times \R^p$. Then  
\[Y:=\left\{(x,u,z)\in\R^n\times\R^n\times\R^p\,\middle|\,
\begin{array}{l}
(u,z)\in K,\\
u\ge_m x,\\ 
u_1\ge\dots\ge u_n
\end{array}
\right\}
\] is convex.
\end{lemma}
\begin{proof}{Proof.}
First, observe that  $\sum_{i=1}^j u_{[i]} = \sum_{i=1}^j u_i$ since $u_1\ge \cdots\ge u_n$.
Further, $\sum_{i=1}^j x_{[i]}$ is a convex function being the maximum of all possible sums of $j$ elements chosen from $x$. Next,  $\sum_{i=1}^n x_{[i]} = \sum_{i=1}^n x_i$ and is, therefore, linear.
Therefore, $Y$ has the following convex representation:
\[
Y=\left\{(x,u,z)\in\R^n\times\R^n\times\R^p\,\middle|\,
\begin{array}{l}
(u,z)\in K,\\
\sum_{i=1}^j u_i \ge \sum_{i=1}^j x_{[i]},\mbox{ for } j=1,\ldots,n-1,\\
\sum_{i=1}^n u_i = \sum_{i=1}^n x_{i},\\
u_1\ge\dots\ge u_n
\end{array}
\right\}.
\]
\hfill \Halmos
\end{proof}

We next present Theorem~\ref{res:main_thm}, which gives an explicit description of the convex hull of a permutation-invariant set when an explicit description of the convex hull of its intersection with the cone $x_1\ge \cdots \ge x_n$ is available. 
This description only requires the introduction of a copy $u$ of the variables $x$ together with majorization constraints.

\begin{theorem}
\label{res:main_thm}
Suppose $S\subseteq\{(x,z)\mid \R^n\times\R^p\}$ is permutation-invariant with respect to $x\in\R^n$.
Then,
\begin{equation}
\label{permconvhull}
\conv(S)=X:=\left\{(x,z)\,\middle|\,\begin{array}{l}
(u,z)\in\conv(S_0),\\ 
u\ge_m x
\end{array}
\right\},
\end{equation}
where $S_0=S\cap\{(u,z)\mid u_1\ge\dots\ge u_n\}$.
\end{theorem}
\begin{proof}{Proof.}
To prove that $X$ is convex using Lemma 3, it suffices to show that $(u,z)\in\conv(S_0)$ implies $u\in\Delta^n$.
This is clear because $\conv(S_0)\subseteq \conv(S)\cap\{(u,z)\mid u_1\ge\dots\ge u_n\}$.

We now show that $S\subseteq X$.
As $X$ is convex, this will also show that $\conv(S)\subseteq X$. 
Consider an arbitrary $(x,z)\in S$ and define $u$ as $u_i = x_{[i]}$ for $i=1,\dots,n$.
Then, $(u,z)\in S_0$ because $S$ is permutation-invariant and $u$ is in descending order. Since $u\ge_m x$, $(x,z)\in X$.

We next prove that $X\subseteq \conv(S)$. 
Let $(x,z)\in X$. We show that this point can be expressed as a convex combination of points in $S$.
Since $(x,z)\in X$, there exists $u$ such that $(u,z)\in\conv(S_0)\subseteq \conv(S)$ and $u\ge_m x$.
It follows from the permutation-invariance of $S$ with respect to $x$ and Lemma~\ref{res:lin_invariant} that $\conv(S)$ is permutation-invariant with respect to $x$. 
By Lemma 2, $x$ can be written as $x=\sum_{i}\lambda_i (P_iu)$, where
$P_i$ is a permutation matrix, $\lambda_i\ge 0$, 
and $\sum_i \lambda_i=1$.
Therefore, 
$(x,z)=(\sum_{i}\lambda_i(P_iu),z)=\sum_i \lambda_i (P_iu, z)$, concluding the proof.
\hfill \Halmos
\end{proof}

We next present classical results that allow for a linear formulation of the majorization constraints; see Section 3.3.4 of \cite{nemirovski2012introduction} for a more thorough discussion. 
\def\LS{\mathop{\text{LS}}}
To model $\sum_{i=1}^j x_{[i]}$, we express it as the value function of the following optimization problem, where $x_1,\ldots,x_n \in \R$ and $j\in \{1,\ldots,n-1\}$:
\begin{equation}
\begin{array}{rlll}
\max 	&\quad& \sum_{i=1}^n x_i s_i\\
\st	&& \sum_{i=1}^n s_i =j,\\
					&& 0\le s_i\le 1, & i=1,\ldots,n.
\end{array}
\label{plargest}
\end{equation}
To model majorization constraints, we need to enforce that for all feasible $s$, $\sum_{i=1}^n x_is_i$ does not exceed $\sum_{i=1}^j u_j$. 
By taking the dual of \eqref{plargest}, we exchange the quantifier ``for-all' to ``there-exists'', and obtain the following formulation which is amenable to direct inclusion in the model:
\begin{equation}
\begin{array}{lrlll}
\LS(j):&\min 	&\quad& jr+\sum_{i=1}^n t_i\\
&\st	&& x_i\le t_i+r,& i=1,\dots,n,\\
&					&& t_i\ge 0, & i=1,\dots,n
\end{array}
\label{plargestdual}
\end{equation}
where the dual variables $r$ and $t$ correspond to the primal constraints $\sum_{i=1}^n s_i=j$ and $s\le 1$, respectively.
Since \eqref{plargest} is feasible, \eqref{plargestdual} exhibits no duality gap. The constraint $\sum_{i=1}^j u_i\ge \sum_{i=1}^j x_{[i]}$ is then modeled through the existence of an $(r,t)$ that is feasible to \eqref{plargestdual} such that $\sum_{i=1}^j u_i \ge jr + \sum_{i=1}^n t_i$.


\begin{theorem}
\label{model_majorization}
Suppose $S\subseteq\{(x,z)\in\Re^n\times \Re^p\}$ is permutation-invariant with respect to $x$.
Then,
\begin{equation}
\label{eqn:model_majorization}
\conv(S)=\left\{
(x,z)\middle|
\begin{array}{ll}
(u,z)\in\conv(S_0)\\
u_1\ge\dots\ge u_n\\
\sum_{i=1}^n u_i=\sum_{i=1}^n x_i\\
\sum_{i=1}^j u_i\ge jr^j + \sum_{i=1}^n t^j_i,&j=1,\dots,n-1,\\
x_i\le t^j_i+r^j,& j=1,\dots,n-1,\:\: i=1,\dots,n,\\
t^j_i\ge 0,& j=1,\dots,n-1,\:\:i=1,\dots,n,\\
\end{array}
\right\},
\end{equation}
where $S_0=S\cap\{(u,z)\mid u_1\ge\dots\ge u_n\}.$
\textrm{ } \hfill \Halmos
\end{theorem}

\begin{remark}
In fact, Theorems~\ref{res:main_thm} and \ref{model_majorization} remain valid for any choice of $S_0$ that satisfies
\[
\conv(S)\cap \{(u,z)\mid u_1\ge\dots\ge u_n\}\supseteq S_0\supseteq S\cap \{(u,z)\mid u_1\ge\dots\ge u_n\}.\hfill\Halmos
\]
\end{remark}

The formulation given in (\ref{eqn:model_majorization}) expresses $\conv(S)$ as the projection of a convex set with $n^2+n+p$ variables. 
This formulation is much smaller than that which would have been obtained using a classical application of disjunctive programming \citep[see][Chapter 2]{balas2018disjunctive}. 
An even smaller representation is possible
using a more compact formulation of the permutahedron. Goemans \cite{goemans2015smallest} 
proposed such an extended formulation for the permutahedron, the convex hull of all possible permutations of a fixed vector $u\in\mathbb{R}^n$ using a sorting network where the numbers of variables and inequalities of the extended formulation depend on the number of comparators of the associated sorting network.
When the Ajtai–Koml\'{o}s–Szemer\'{e}di sorting network \cite{ajtai19830} is used, the extended formulation for the permutahedron has $\Theta(n\log n)$ variables and inequalities. 
As imposing  the majorization constraint $u\ge_m x$ is  equivalent to requiring that $x$ belongs to the permutahedron generated by $u$, alternative extended formulations of $\conv(S)$ can be obtained by replacing $u\ge_m x$ in \eqref{permconvhull} with such extended formulations.
This results in a formulation of $\conv(S)$ that is more compact  than \eqref{eqn:model_majorization}. 
For relaxations of sparse principal component analysis, 
we show in Section~\ref{subsec:sparsePCA} that this smaller formulation provides some computational benefit on larger instances in our test set.

The ideas underlying the proof of Theorem~\ref{model_majorization} can also be applied when sets are invariants with respect to collections of linear transformations that are not permutation matrices.
In particular, we describe next a related convexification result for sign-invariant sets.
\begin{theorem} 
Suppose $S\subseteq \{(x,z)\in\R^n\times\R^p\}$ is sign-invariant with respect to $x$.
Then, 
\begin{equation}
\label{signinvarianteq}
\conv(S)= X:= \left\{(x,z) \,\middle|\, (u,z)\in \conv(S_0), u\ge |x|\right\}
\end{equation}
where $S_0=S\cap \left(\R^n_+\times\R^p\right)$.
\label{signinvariant}
\end{theorem}
\begin{proof}{Proof.}
Set $X$ is convex because it is the projection of an intersection of two convex sets.
We now show that $S\subseteq X$.
For an arbitrary $(x,z)\in S$, define $u=|x|$.
By sign-invariance of $S$, $(u,z)\in S$ and hence $(u,z)\in S_0\subseteq\conv(S_0)$. 
By definition, $u$ satisfies $u\ge |x|$.
This shows that $\conv(S)\subseteq X$.

\def\x{\bar{x}}
We next show that $X\subseteq \conv(S)$.
Let $(x,z)\in X$. 
There exists $u\in\R^n$ such that $(u,z)\in\conv(S_0)\subseteq \conv(S)$ and $u\ge|x|$. 
Since $\conv(S)$ is sign-invariant by Lemma~\ref{res:lin_invariant}, it follows that $\{(\x, z)\mid \x_i \in \{u_i,-u_i\}\}\subseteq \conv(S)$. Therefore,
$(x,z)\in \{(\x, z)\mid |\x_i|\le u_i \}\subseteq \conv(S)$, where the containment follows from the convexity of $\conv(S)$.
\hfill \Halmos
\end{proof}

Next, we consider the case where the target set $S$ is both sign-invariant and permutation-invariant.
While two separate representations \eqref{permconvhull} and \eqref{signinvarianteq} for the convex hull already exist, 
we can derive a unified representation. 
Inequality $u\ge_m|x|$ would not be a useful description because $\sum_{i=1}^n u_i = \sum_{i=1}^n |x_i|$ is a non-convex constraint. In fact, such a set would not even be convex, as can be seen from $S=\{-1,1\}$. Instead, we prove that the convex hull can be represented using weak majorization.
More precisely, suppose that $S\subseteq \{(x,z)\mid\R^n\times\R^p\}$ is sign- and permutation-invariant with respect to $x\in\R^n$. Then, $\conv(S)=X:=\left\{(x,z)\,\middle|\,
(u,z)\in\conv(S_0), u\ge_{wm} |x|\right\}$,
where $S_0=S\cap\{(u,z)\mid u_1\ge\dots\ge u_n\ge 0\}$.
We first prove that $S\subseteq X$. Consider an arbitrary $(x,z)\in S$, and define $u$ as $u_i=|x|_{[i]}$ for $i=1,\dots,n$. Then, $(u,z)\in S_0\subseteq \conv(S_0)$ and $u\ge_{wm} |x|$, showing the inclusion.
We next prove $X\subseteq \conv(S)$. Consider an arbitrary $(x,z)\in X$. Then, there exists $u$ such that $(u,z)\in\conv(S_0)\subseteq \conv(S)$ and $u\ge_{wm}|x|$.
It follows that there exists $u'$ such that $u\ge_m u'$ and $u'\ge |x|$; see 5.A.9. of \cite{marshall2010}, for example.
Then, it can be shown that $(u',z)\in\conv(S)$ using the same arguments that were used in the proof of Theorem~\ref{res:main_thm}. Since $\conv(S)$ is sign-invariant, this implies that all sign variants of $(u',z) \in \conv(S)$. Then, by the last part of the proof of Theorem~\ref{signinvariant}, this implies that $(x,z)\in \conv(S)$. 
(This convex hull description can also be derived by first constructing $\conv\bigl(S\cap (\Re^n_+\times \Re^p)\bigr)$ using Theorem~\ref{res:main_thm}. Then, $\conv(S)$ is obtained by Theorem~\ref{signinvariant}. The variables $u'$ introduced by Theorem~\ref{signinvariant} can finally be projected using 5.A.9. from \cite{marshall2010}.) 


The above convexification results can be easily extended to the sets which are permutation-invariant or/and sign-invariant with respect to multiple subsets of independent variables.

\begin{theorem}
\label{permconvset}
Let $S\subseteq \{(x^1,\dots,x^m,z)\in\mathbb{R}^{n_1}\times\dots\times\mathbb{R}^{n_m}\times\mathbb{R}^p\}$.
\begin{enumerate}
\item Suppose $S$ is a permutation-invariant set with respect to $x^k$ for $k=1,\dots,m$.
Then,
\begin{equation}
\conv(S)
=\left\{(x^1,\dots,x^m,z)\,\middle|\,\begin{array}{l}(u^1,\dots,u^m,z)\in \conv(S_0),\\u^k\ge_m x^k, k=1,\dots,m\end{array}\right\}
\label{permconveq}
\end{equation}
where $S_0=S\cap \left(\Delta^{n_1}\times \dots\times\Delta^{n_m}\times \mathbb{R}^p\right)$.
\item Suppose $S$ is sign-invariant with respect to $x^k$ for $k=1,\dots,m$.
Then,
\begin{equation}
\conv(S)
=\left\{(x^1,\dots,x^m,z)\,\middle|\,\begin{array}{l}(u^1,\dots,u^m,z)\in \conv(S_0),\\u^k\ge |x^k|, k=1,\dots,m\end{array}\right\}
\label{signconv}
\end{equation}
where $S_0=S\cap\left(\R^{n_1}_+\times\dots\times\R^{n_m}_+\times\R^p\right)$.

\item Suppose $S$ is permutation-invariant and sign-invariant with respect to $x^k$ for $k=1,\dots,m$.
Then,
\begin{equation}
\conv(S)
=\left\{(x^1,\dots,x^m,z)\,\middle|\,\begin{array}{l}(u^1,\dots,u^m,z)\in \textup{conv}(S_0),\\u^k\ge_{wm} |x^k|, k=1,\dots,m\end{array}\right\}
\label{signpermconv}
\end{equation}
where 
\[
S_0=S\cap\left\{(u^1,\dots,u^m,z)\mid u^k_1\ge\dots\ge u^k_{n_k}\ge 0, \:\: k=1,\dots,m\right\}.
\tag*{\Halmos}\]
\end{enumerate}
\end{theorem}

We remark that our convexification result on sign-invariant sets can also be used to convexify reflections on a hyperplane \cite{kp11}. In particular, consider a set $S$ and assume that the set is closed under reflection on $\langle a, \cdot\rangle = 0$ for some $a$ with $\|a\|_2 = 1$. Assume further that $S_0 = S \cap \{x\mid \langle a, x\rangle \ge 0\}$ is available. Consider any orthogonal matrix $U$ that aligns $a$ along the first principal direction, so that $Ua=e_1$. Observe that $US$ is sign-invariant with respect to variable $x_1$. Then, $\conv(US) = \bigl\{z\bigm| (t,z_2,\ldots, z_n)\in U\conv(S_0), t\ge |z_1|\bigr\}$ and so, $\conv(S) = \bigl\{x\bigm| x + (t-z_1)a \in \conv(S_0), t \ge |z_1|\bigr\}$, where we have used the fact that $U^{\tra} e_1 = a$.



\section{Sparsity theorem}\label{sec:sparsity}

In this section, we first study the convex hull of the set
\begin{equation}
\label{eq:sparsity}
\cnorm=\{x\in\mathbb{R}^n\mid \|x\|_s\le1, \card(x)\le K\},
\end{equation}
where $\|\cdot\|_s$ is a sign- and permutation-invariant norm (also known as a {\it symmetric gauge function}) and introduce the \emph{Sparsity Theorem} (Theorem~\ref{thm:sparsity}), which shows that the optimal value of
\begin{equation}\label{eq:firstsparsitythm}
\begin{array}{rl}
\min & \|u\|_s\\
\st & u_1\ge\dots\ge u_K\ge 0,\\
&u_{K+1}=\dots=u_n=0,\\
&u\ge_{wm} |x|
\end{array}
\end{equation}
is no more than $1$ if and only if $x\in \conv(\cnorm)$ and that an optimal solution to \eqref{eq:firstsparsitythm} can be obtained in closed-form. 
In other words, given $x\in\mathbb{R}^n$, Theorem~\ref{thm:sparsity} gives a closed-form expression for a sparse vector $u$, in terms of $x$, in the representative disjunct $\Delta^n\cap\mathbb{R}^n_+$ that weakly majorizes $|x|$ and minimizes $\|u\|_s$.
In fact, this $u$ majorizes $|x|$. This is not surprising. If the optimal $u$ to \eqref{eq:firstsparsitythm} did not majorize $|x|$, there exists $u'$ such that $u\ge u'\ge_m |x|$; see 5.A.9. in \cite{marshall2010}. Then, since $u'$ is in the convex hull of $u$ and its sign-variants, it follows by sign-invariance of $\|\cdot\|_s$ that $\|u'\|_s\le \|u\|_s$, and so $u'$ is also optimal and majorizes $|x|$.

We denote the set $\{x \in \Re^n \mid \|x\|_s\le r\}$ by $\Ball{s}{r}$. 
When $K=1$, the convex hull is an $\ell_1$-norm ball.
The set is trivial when $K=n$. 
Therefore, we assume $1< K< n$. When the associated norm $\|\cdot\|_s$ is the $\ell_2$-norm, $N^K_{\|\cdot\|}$ is the feasible set of the sparse principal component analysis problem (sparse PCA); see \cite{d2007direct}.

We define 
\begin{equation}
\label{eq:delta_n_+}
\Delta^n_+:=\Delta^n\cap\Re^n_+
\end{equation}
and, for any vector $x\in\Re^n$, define 
\[
(x_{\Delta^n})_i=x_{[i]} \hbox{ and }  (x_{\Delta^n_+})_i=|x|_{[i]}, \quad \hbox{ for  } i=1,\dots,n.
\]
When the dimension $n$ of the set and the associated vector is clear in the context, we use the simpler notations $\Delta, \Delta_+, x_\Delta$, and $x_{\Delta_+}$.

By sign- and permutation-invariance of the norm $\|\cdot\|_s$ and that of the cardinality requirement, $\cnorm$ is sign- and permutation-invariant and hence we can apply Theorem~\ref{permconvset} to obtain its convex hull as a projection of a higher dimensional set
\begin{equation}
\label{cnorm_conv}
\conv \left(\cnorm\right)
=\left\{x\in\mathbb{R}^n\middle\vert
\begin{array}{l}
u\in \cnorm\cap\Delta_+,\\
u\ge_{wm} |x|
\end{array}
\right\}
=\left\{x\in\mathbb{R}^n\middle\vert
\begin{array}{l}
\|u\|_s\le 1,\\
u_1\ge\dots\ge u_K\ge 0,\\
u_{K+1}=\dots=u_n=0,\\
u\ge_{wm} |x|
\end{array}
\right\}.
\end{equation}

The extended formulation \eqref{cnorm_conv} can be written in closed-form with $O(nK)$ additional variables and constraints based on the modeling technique  described in Section~\ref{sec:mainresults}.
Other extended formulations are proposed in \cite{limnote} and \cite{lim2017k} for the case where $\|\cdot\|_s$ is an $\ell_p$-norm. 
In these papers, the formulations are obtained either through dynamic programming concepts or Goemans' extended formulation of the permutahedron using a sorting network \cite{goemans2015smallest}.

In this section, we describe the convex hull as a norm ball in the original variable space.
The induced norm is easily calculable if the associated norm $\|\cdot\|_s$ is calculable.
Moreover, given an arbitrary point in $\Re^n$ not in the convex hull, we devise an algorithm to construct a separating hyperplane.

We first present the following lemma introduced in \cite{li1995permutation}.
\begin{lemma}
\label{lemma:majnormineq}
Suppose $x\ge_m y$. Then, for any permutation-invariant norm $f(\cdot)$, $f(x)\ge f(y)$.
\hfill\Halmos
\end{lemma}

A set in $\mathbb{R}^n$ is called a \emph{convex body} if it is a compact convex set with non-empty interior.
In the next proposition, we show that $\conv(\cnorm)$ is a convex body.

\begin{proposition}
\label{prop:cnorm_compact}
The set $\conv(\cnorm)$ is a convex body.
\end{proposition}
\begin{proof}{Proof.}
Since $\cnorm$ is a compact set, it follows that $\conv(\cnorm)$ is a compact convex set \cite[Corollary I.2.4]{b02}.
To see that $\conv(\cnorm)$ has a non-empty interior, observe that there exists $\epsilon>0$ such that $\Ball{1}{\epsilon}\subseteq \Ball{s}{1}$ 
where $\Ball{1}{\epsilon}$ represents the $\ell_1$-norm ball with radius $\epsilon$.
This follows from the equivalence of norms in a finite vector space.
Notice that $\ext\left(B_1(\epsilon)\right)=\{\pm\epsilon e_i\mid i=1,\dots,n\}$ where $e_i$ is $i$th canonical vector.
Since, for any $x\in \ext\bigl(\Ball{1}{\epsilon}\bigr)$, $\card(x) = 1$, it follows that $\ext\bigl(\Ball{1}{\epsilon}\bigr)\subseteq \cnorm$, and, so
$\Ball{1}{\epsilon}\subseteq \conv(\cnorm)$. 
The result follows because $0\in \interior\bigl(\Ball{1}{\epsilon}\bigr)$.
\hfill \Halmos
\end{proof}



It is well-known that there exists a one-to-one correspondence between norms in $\mathbb{R}^n$ and convex bodies symmetric about 0 and containing 0 in their interior; see Section 14.4 of \cite{matouvsek2002lectures} for instance.
Given an arbitrary norm $\|\cdot\|$, we can construct its unit ball $\{x \mid \|x\| \le 1\}$, which is a convex body of the desired type. 
Conversely, given any compact convex body $C$ that is symmetric about 0 and contains 0 in its interior, we can define the function
\begin{equation}
    \label{eq:convbdnorm}
f_C(x):=\min\left\{t>0 \,\middle|\, \frac{x}{t}\in C\right\}
\end{equation}
for $x\in\mathbb{R}^n$.
It is known that the function $f_C$ satisfies the properties of norms; see Section 14.4 of \cite{matouvsek2002lectures}, for example.
Further, the convex body $C$ is a lower-level set of this norm, that is, $C=\{x\mid f_C(x)\le 1\}$.

Since $\conv(\cnorm)$ is a compact convex body that is symmetric about 0 and contains 0 in its interior, a norm associated with $\conv(\cnorm)$ can be defined as in \eqref{eq:convbdnorm}.
We denote the corresponding norm by $\|\cdot\|_c$.
Since $\|\cdot\|_c$ is sign- and permutation-invariant, the following result holds.

\begin{proposition}
\label{prop:c-normset}
The set $\conv(\cnorm)$ is the unit ball associated with a sign- and permutation-invariant norm, that is, $\conv(\cnorm)=\Ball{c}{1}$.
\hfill \Halmos
\end{proposition}

We next show that the values of the norms $\|\cdot\|_c$ and $\|\cdot\|_s$ are the same for vectors that satisfy the cardinality constraint.

\begin{proposition}
\label{prop:cardKcnorm}
If $\card(x)\le K$, $\|x\|_c=\|x\|_s$.
\end{proposition}
\begin{proof}{Proof.}
We first show that $\|x\|_c\ge \|x\|_s$. Since $\cnorm \subseteq \Ball{s}{1}$, it follows that $\Ball{c}{1}=\conv(\cnorm)\subseteq \Ball{s}{1}$. 
This implies that $\Ball{c}{r}\subseteq \Ball{s}{r}$ for any $r$. In particular, when $r=\|x\|_c$, we have $x\in \Ball{c}{r}\subseteq \Ball{s}{r}$, which implies that $\|x\|_c\ge \|x\|_s$. 
We now show $\|x\|_c \le \|x\|_s$ when $\card(x)\le K$. Let $r=\|x\|_s$ and observe that $\frac{x}{r} \in \cnorm \subseteq \Ball{c}{1}$. 
Therefore, $x\in \Ball{c}{r}$ or $\|x\|_c\le \|x\|_s$.
\hfill \Halmos
\end{proof}

We present an explicit formula to evaluate $\|\cdot\|_c$.
For an arbitrary $x\in\mathbb{R}^n$, define $s(x)\in\mathbb{R}^{K+2}$ as \[s(x)_i=\frac{\sum_{j=i}^n|x|_{[j]}}{K-i+1},\: i=1,\dots,K;\; s(x)_0=s(x)_{K+1} = \infty.\]
Let $i_x$ be the minimum among those indices that minimize $s(x)_i$, and let $\delta(x)=s(x)_{i_x}$.
Now, define $u(x)\in\mathbb{R}^n$ as 
\begin{equation}
\label{eq:ux}
u(x)_i=\left\{\begin{array}{ll}
|x|_{[i]}, & i\in\{1,\dots,i_x-1\}\\
\delta(x), & i\in\{i_x, \dots,K\}\\
0, & \textup{otherwise}.
\end{array}\right.
\end{equation}

In the following proposition, we show that, for arbitrary $x\in\mathbb{R}^n$, we can construct a vector $u(x)\in\Delta_+$ that satisfies the cardinality constraint and majorizes $|x|$.

\begin{proposition}
\label{prop:umajx}
Let $s(x), i_x, \delta(x)$, and $u(x)$ be defined as above. 
Then,
\begin{enumerate}
\item\label{proppart:sssx} $s(x)_{i+1}-s(x)_i=\frac{1}{K-i+1}(s(x)_{i+1}-|x|_{[i]})=\frac{1}{K-i}(s(x)_i-|x|_{[i]})$ for $i=1,\dots,K-1$
\item\label{proppart:sv} $s(x)_1\ge\dots\ge s(x)_{i_x}$ and $s(x)_{i_x}\le\dots\le s(x)_K$
\item\label{proppart:umajx} $u(x)\ge_m |x|$
\item\label{proppart:u1} $u(x)_1=\max\{|x|_{[1]},s(x)_1\}$
\end{enumerate}
\end{proposition}
\begin{proof}{Proof.}
By definition of $s(x)$, 
\[\begin{array}{rl}
(K-i)s(x)_{i+1} - (K-i+1)s(x)_i &=(K-i)\frac{\sum_{j=i+1}^n|x|_{[j]}}{K-i} - (K-i+1)\frac{\sum_{j=i}^n|x|_{[j]}}{K-i+1}\\
&=\sum_{j=i+1}^n|x|_{[j]}- \sum_{j=i}^n|x|_{[j]}=-|x|_{[i]}.
\end{array}\]
By adding $s(x)_{i+1}$ on both sides, we have $(K-i+1)(s(x)_{i+1}-s(x)_i)=s(x)_{i+1}-|x|_{[i]}$, implying the first equality of Part~\ref{proppart:sssx}.
The remainder of the part can be shown similarly by adding $s(x)_i$ on both sides.

To see Part~\ref{proppart:sv}, if there is no index $i$ in $\{1,\ldots,K-1\}$ so that $s(x)_{i} < s(x)_{i+1}$, the result holds trivially. Assume now that $i'$ is the smallest such index. Then, 
$s(x)_{i'+1} > s(x)_{i'} > |x|_{[i']} \ge |x|_{[i'+1]}$, where the second inequality follows because
$s(x)_{i'+1}-s(x)_{i'}>0$ implies $s(x)_{i'}>|x|_{[i']}>0$ from Part~\ref{proppart:sssx}. 
This in turn shows, using Part~\ref{proppart:sssx} that $s(x)_{i'+2} > s(x)_{i'+1}$ as long as $i' < K-1$. By induction, $s(x)_{i'} < \cdots < s(x)_K$ and, by the definition of $i'$, $s(x)_1 \ge \cdots \ge s(x)_{i'}$. Then, Part~\ref{proppart:sv} follows by defining $i_x$ as the first index such that $s(x)_{i_x} = s(x)_{i'}$. 

We next prove Part~\ref{proppart:umajx}.
We first show that $u(x)$ is nonincreasing.
If $i_x=1$ then all components of $u(x)$ equal $\delta(x)$ and hence it is nonincreasing.
Now, assume that $i_x\ge 2$. By definition of $u(x)$, it suffices to show that  $|x|_{[i_x-1]}\ge \delta(x)$.
Then, $s(x)_{i_x}-|x|_{[i_x-1]}=(K-i_x)(s(x)_{i_x}-s(x)_{i_x-1})\le 0$ by plugging in $i=i_x-1$ in the first equality of Part~\ref{proppart:sssx}.
Therefore, $|x|_{[i_x-1]}\ge s(x)_{i_x}=\delta(x)$, completing the proof that $u(x)$ is nonincreasing.
Thus, $u(x)_{[i]}=u(x)_i$ for all $i=1,\dots,n$. 
Next, observe that $\sum_{i=i_x}^n u(x)_i=\sum_{i=i_x}^n|x|_{[i]}$
by definition of $\delta(x)$. 
This implies in turn that $\sum_{i=1}^n u(x)_i=\sum_{i=1}^n |x|_{[i]}$.
We next show that $\sum_{i=1}^j u(x)_i\ge \sum_{i=1}^j |x|_{[i]}$ for all $j=1,\dots,n-1$.
When $j=1,\dots,i_x-1$, the inequality holds with equality by definition of $u(x)$.
We next consider the case $j\ge i_x$.
If $i_x=K$, the inequality holds because $\sum_{i=1}^j u(x)_i=\sum_{i=1}^K u(x)_i=\sum_{i=1}^n u(x)_i= \sum_{i=1}^n |x|_{[i]}\ge\sum_{i=1}^j |x|_{[i]}$ where the first and the second equalities follow because $j\ge K$ and $u(x)_{K+1}=\dots=u(x)_n=0$. 
Now assume that $i_x<K$.
Since $s(x)_{i_x+1}\ge s(x)_{i_x}$ and $s(x)_{i_x+1}-s(x)_{i_x}=\frac{1}{K-i_x}(s(x)_{i_x}-|x|_{[i_x]})$ by Part~\ref{proppart:sssx}, $s(x)_{i_x}\ge |x|_{[i_x]}$ and hence $\delta(x)=s(x)_{i_x}\ge |x|_{[i]}$ for all $i\ge i_x$.
Therefore, $\sum_{i=1}^j u(x)_i - \sum_{i=1}^j |x|_{[i]}
=\sum_{i=i_x}^j u(x)_i - \sum_{i=i_x}^j |x|_{[i]}
=\sum_{i=i_x}^j (\delta(x) - |x|_{[i]})\ge 0$.

For Part~\ref{proppart:u1}, first assume $i_x=1$. 
Then, $u(x)_1=s(x)_1$. By Part 1, $s(x)_1\ge|x|_{[1]}$ and hence 
$u(x)_1=\max\{|x|_{[1]},s(x)_1\}$.
Next, assume that $i_x\ge 2$. Then, $u(x)_1=|x|_{[1]}$. By Part~\ref{proppart:sv}, $s(x)_2\le s(x)_1$ and hence, by Part~\ref{proppart:sssx}, $s(x)_1\le |x|_{[1]}$.
Therefore, $u(x)_1=\max\{|x|_{[1]}, s(x)_1\}$.
\hfill \Halmos
\end{proof}

\def\betaopt{\hat{\beta}}
\def\bbeta{\bar{\beta}}
\def\cbeta{\check{\beta}}

We next show in the following theorem that $u(x)$ can be used to compute $\|x\|_c$ if $\|\cdot\|_s$ is calculable. 
To this end, we introduce the following notations.
Let $\betaopt$ be an optimal solution to 
\begin{equation}
\label{eq:supp_coeff}
\max\bigl\{\beta\tra u(x)\bigm|\|\beta\|_{s*}\le 1\bigr\},
\end{equation}
where $\|\cdot\|_{s*}$ is the dual norm of $\|\cdot\|_s$.
We now define $\theta$ and $\chi\in \R^n$ as follows:
\begin{equation}
\theta_i=
\left\{\begin{array}{cl}
\betaopt_i,  & i=1,\dots,i_x-1\\ 
\frac{\sum_{j=i_x}^K \betaopt_j}{K-i_x+1}, & i=i_x,\dots,K\\
0, & \textup{otherwise}
\end{array}\right.,\quad
\chi_i=\left\{\begin{array}{cl}\betaopt_i & \textup{for $i=1,\dots,i_x-1$}\\
\frac{\sum_{i=i_x}^K \betaopt_i}{K-i_x+1}&\textup{otherwise}
\end{array}\right..
\label{eq:theta_chi}
\end{equation}

\begin{theorem}
\label{thm:cnorm_snorm}
For an arbitrary $x\in\mathbb{R}^n$, let $u(x)$ be defined as in (\ref{eq:ux}). 
Then, $\|x\|_c=\|u(x)\|_s$.
\end{theorem}
\begin{proof}{Proof.}
Recall the definitions of $s(x)$, $i_x$, and $\delta(x)$.
We have $\|x\|_c\le \|u(x)\|_c = \|u(x)\|_s$, where the first inequality is because of Part~\ref{proppart:umajx} of Proposition~\ref{prop:umajx} and Lemma~\ref{lemma:majnormineq}, while the equality is due to Proposition~\ref{prop:cardKcnorm}. We next show $\|x\|_c\ge \|u(x)\|_s$. 
Let $r=\|u(x)\|_s$ so that $u(x)\in B_s(r)$.
By the definition of $\|\cdot\|_s$ and $\hat{\beta}$, we have $\betaopt\tra u(x) = r$. Moreover, $\betaopt\tra u \le \|u\|_s \le r$ for all $u\in \Ball{s}{r}$, where the first inequality is because $\betaopt\in \Ball{s^*}{1}$ and the second inequality is because $u\in \Ball{s}{r}$. Because $u(x)\in \Delta_+$, it follows from rearrangement inequality that, without loss of generality, we may assume that $\betaopt\in \Delta_+$. 

We next show that $\theta\tra u\le r$ for all $u\in \Ball{s}{r}$ where $\theta$ is as defined in \eqref{eq:theta_chi}. 
Define $\bbeta:=(\betaopt_1,\dots,\betaopt_K,0,\dots,0)$ and
$\cbeta:=(\betaopt_1,\ldots,\betaopt_K,-\betaopt_{K+1},\ldots,-\betaopt_n)$ and observe that $\bbeta = \frac{\betaopt+\cbeta}{2}$. By sign-invariance of $\|\cdot\|_s$ and thus of $\|\cdot\|_{s^*}$, $\cbeta\in \Ball{s^*}{1}$, and, so $\bbeta\in \Ball{s^*}{1}$. However, since $\bbeta\ge_m \theta$, Lemma~\ref{lemma:majnormineq} shows that $\theta\in \Ball{s^*}{1}$. This in turn shows that $\theta\tra u\le r$ is valid for $\Ball{s}{r}$.

We next claim that $\chi\tra u\le 1$ is valid for $\cnorm$ where $\chi$ is as defined in \eqref{eq:theta_chi}. Assume that, on the contrary, there exists $\hat{u}\in \cnorm$ such that $\chi\tra \hat{u} > 1$. 
Because of the rearrangement inequality and the fact that $\chi\in \Delta_+$, we may assume that $\hat{u}\in \Delta_+$. This yields the contradiction $1\ge \theta\tra\hat{u} = \chi\tra\hat{u} > 1$, where the first inequality is by the validity of $\theta\tra u\le 1$ for $\Ball{s}{1}$ which outer-approximates $\cnorm$ and the equality can be arranged by choosing $\hat{u}$ with support at most $K$. It follows that $\chi\tra u\le r$ is valid for $\Ball{c}{r}$ or, in other words, that $\chi \in \Ball{c^*}{1}$, where $\|\cdot\|_{c^*}$ is the dual norm of $\|\cdot\|_c$. Therefore, 
\begin{equation}\label{eq:xc-bound}
\|x\|_c = \|x_{\Delta_+}\|_c = \max\{\beta\tra x_{\Delta_+}\mid \|\beta\|_{c^*}\le 1\} \ge \chi\tra x_{\Delta_+},
\end{equation}
where the inequality is because $\|\chi\|_{c*}\le 1$. However, the following calculation shows that \begin{equation}
\label{eq:chi_sep}
\begin{array}{ll}
\chi^\intercal x_{\Delta_+} &= \displaystyle\sum_{i=1}^{i_x-1} \betaopt_i |x|_{[i]}+\frac{\sum_{i=i_x}^K \betaopt_i}{K-i_x+1} \sum_{i=i_x}^n |x|_{[i]}
=\displaystyle\sum_{i=1}^{i_x-1} \betaopt_i |x|_{[i]}+\sum_{i=i_x}^K \betaopt_i  \frac{\sum_{i=i_x}^n|x|_{[i]}}{K-i_x+1}=\displaystyle\betaopt\tra u(x) = r.
\end{array}
\end{equation}
Combining \eqref{eq:xc-bound} and \eqref{eq:chi_sep}, we conclude that $\|x\|_c\ge \|u(x)\|_s$, where we have used that $r$ was defined to be $\|u(x)\|_s$.
\hfill \Halmos
\end{proof}

\begin{corollary}
\label{cor:separation}
For a fixed $x\in \R^n$, let $\chi$ be as defined in \eqref{eq:theta_chi}. Then, $\chi\tra u \le 1$ is valid for $\cnorm$, $\|\chi\|_{c^*}=1$, and $\chi\tra x_{\Delta_+} = \|x\|_c$. In particular, $\chi\tra u\le 1$ separates $x_{\Delta_+}$ if $x\not\in \conv(\cnorm)$.
\end{corollary}
\begin{proof}{Proof.}
It was already shown in the proof of Theorem~\ref{thm:cnorm_snorm} that $\chi\tra u\le 1$ is valid for $\cnorm$. Let $u(x)$ be as defined in \eqref{eq:ux}. Then, $\chi\tra x_{\Delta_+} = \|u(x)\|_s =\|x\|_c$, where the first equality is from \eqref{eq:chi_sep} and the second equality is from Theorem~\ref{thm:cnorm_snorm}. Therefore, $\|x_{\Delta_+}\|_c = \chi\tra x_{\Delta_+}\le \|\chi\|_{c^*}\|x_{\Delta_+}\|_{c}\le \|x_{\Delta_+}\|_c$, where the first inequality is due to Cauchy-Schwarz inequality and the second inequality is because the proof of Theorem~\ref{thm:cnorm_snorm} shows that $\|\chi\|_{c^*}\le 1$. Therefore, equality holds throughout and, in particular, $\|\chi\|_{c^*} = 1$. If $x\not\in \conv(\cnorm)=\Ball{c}{1}$ then $\chi\tra x_{\Delta_+} = \|x\|_c > 1$,  where the inequality is because $x\not\in \Ball{c}{1}$. 
\hfill \Halmos
\end{proof}


\begin{remark}
In the proof of Corollary~\ref{cor:separation}, let $T$ be the transformation (a composition of sign-conversions and permutations) that maps $x$ to $x_{\Delta_+}$.
Then, the hyperplane that separates $x$ and $\cnorm$ is $T^{-1}(\chi)\tra u\le 1$.
\end{remark}



\begin{theorem}[Sparsity Theorem]
\label{thm:sparsity}
For an arbitrary $x\in\mathbb{R}^n$, $u(x)$ as defined in \eqref{eq:ux} is 
an optimal solution to
\[\begin{array}{rl}
\min & \|u\|_s\\
\st & u_1\ge\dots\ge u_K\ge 0,\\
&u_{K+1}=\dots=u_n=0,\\
&u\ge_{wm} |x|.
\end{array}\]
\end{theorem}
\begin{proof}{Proof.}
First, $u(x)$ is feasible because of its definition and Part~\ref{proppart:umajx} of Proposition~\ref{prop:umajx}.
Now, it suffices to show that $\|u(x)\|_s\le \|u\|_s$ for any feasible solution $u$.
Since $u\ge_{wm}|x|$, there exists $u'$ such that $u\ge_m u'$ and $u'\ge |x|$; see 5.A.9. of \cite{marshall2010}, for example.
Then, 
$\|u\|_s=\|u\|_c\ge\|u'\|_c\ge \|x\|_c=\|u(x)\|_s$ where the first equality follows from Proposition~\ref{prop:cardKcnorm}, the first inequality follows from Lemma~\ref{lemma:majnormineq}, and the last equality follows from Theorem~\ref{thm:cnorm_snorm}.
Finally, to prove the second inequality, observe that $|x|$ can be written as a convex combination of $u'$ and its sign-invariants because $u'\ge |x|$. 
By the triangle inequality, positive homogeneity, and the sign-invariance of $\|\cdot\|_c$, $\|x\|_c\le \|u'\|_c$.
\hfill \Halmos
\end{proof}

\begin{example}
Consider the set $N^3_{\|\cdot\|_2}$ where $n=6$. Let $x:=\left(\frac{27}{28}, \frac{5}{28}, \frac{4}{28}, \frac{3}{28}, \frac{2}{28},\frac{1}{28}\right)$.
Note that $\|x\|_2=1$ and $x\in{\Delta_+}$.
For illustration, we establish that $x\notin\conv(N^3_{\|\cdot\|_2})$ by constructing an explicit separating hyperplane described in \eqref{eq:supp_coeff} and \eqref{eq:theta_chi}.
First, we construct the vector $s(x)\in\mathbb{R}^3$ as follows:
\[\begin{array}{llrl}
s(x)_1&=\frac{\sum_{j=1}^6 x_{j}}{3-1+1}=&\frac{1}{3}\left(\frac{27}{28}+\frac{5}{28}+\frac{4}{28}+\frac{3}{28}+\frac{2}{28}+\frac{1}{28}\right)&=\frac{28}{56}\\
s(x)_2&=\frac{\sum_{j=2}^6 x_{j}}{3-2+1}=&\frac{1}{2}\left(\frac{5}{28}+\frac{4}{28}+\frac{3}{28}+\frac{2}{28}+\frac{1}{28}\right)&=\frac{15}{56}\\
s(x)_3&=\frac{\sum_{j=3}^6 x_{j}}{3-3+1}=&\frac{1}{1}\left(\frac{4}{28}+\frac{3}{28}+\frac{2}{28}+\frac{1}{28}\right)&=\frac{20}{56}.
\end{array}\]
Observe that $s(x)_2=\min\bigl\{s(x)_1,s(x)_2,s(x)_3\bigr\}$.
Next, we compute that
\[
u(x)_1 = x_1 = \tfrac{27}{28},\quad
u(x)_2 = u(x)_3 = s(x)_2  = \tfrac{15}{56},\quad u(x)_4 =u(x)_5=u(x)_6= 0.
\]
Since $\|u(x)\|_2=1.036\dots>1$, we conclude from Theorem~\ref{thm:cnorm_snorm} that $x\notin \conv\bigl(N^3_{\|\cdot\|_2}\bigr)$.
We now derive the separating hyperplane.
We first separate $u(x)$ from $\Ball{2}{1}$.
Since $\|\cdot\|_2$ is self-dual, $\betaopt = u(x)/\|u(x)\|_2$ is an optimal solution to \eqref{eq:supp_coeff}.
Then, the inequality $\betaopt\tra u\le 1$ (or equivalently $u(x)\tra u\le \|u(x)\|_2$) is valid for $\Ball{2}{1}$.
Furthermore, it separates $u(x)$ because $u(x)\tra u(x) =\|u(x)\|_2^2>\|u(x)\|_2$.
We next construct a hyperplane that separates $x$ from $N^3_{\|\cdot\|_2}$.
Define $\theta$ and $\chi$ as follows:
\[\begin{array}{ll}
\theta_1=\beta_1=\frac{1}{\|u(x)\|_2}u(x)_1=\frac{1}{\|u(x)\|_2}\frac{27}{28},\quad\theta_2=\theta_3=\frac{1}{3-2+1}(\beta_2+\beta_3)=\frac{1}{\|u(x)\|_2}\frac{15}{56},\quad \theta_4=\dots=\theta_6=0\\
\chi_1=\beta_1=\frac{1}{\|u(x)\|_2}u(x)_1=\frac{1}{\|u(x)\|_2}\frac{27}{28},\quad\chi_2=\dots=\chi_6=\frac{1}{3-2+1}(\beta_2+\beta_3)=\frac{1}{\|u(x)\|_2}\frac{15}{56}.\\
\end{array}\]
Observe that $\|\beta\|_{2} =\|\theta\|_{2}=1$, but $\|\chi\|_2 > 1$.
Now consider the inequality $\chi^\intercal y\le 1$. It is valid for $N^3_{\|\cdot\|_2}$ because for any $u\in N^3_{\|\cdot\|_2}$, 
$\chi\tra u\le \chi\tra u_{\Delta_+} =\theta\tra u_{\Delta_+}\le \|\theta\|_2\|u_{\Delta_+}\|_2\le 1$.
Moreover, it separates $x$ because $\chi^\intercal x =\frac{1}{\|u(x)\|_2}\left(\frac{27}{28}, \frac{15}{56}, \frac{15}{56}, \frac{15}{56}, \frac{15}{56}, \frac{15}{56}\right)^\intercal \left(\frac{27}{28}, \frac{5}{28}, \frac{4}{28}, \frac{3}{28}, \frac{2}{28},\frac{1}{28}\right)= 1.036\dots>1$.
\hfill \Halmos
\end{example}

Next, we consider some special cases of the set $\cnorm$ defined in \eqref{eq:sparsity} and provide explicit convex hull descriptions.
\begin{proposition}
\label{prop:hiriartx}
Let $S = \{x\in\mathbb{R}^q\mid \textup{card}(x)\le K, \|x\|_{\infty}\le r\}$ where $\|x\|_\infty=|x|_{[1]}$. Then, 
\begin{equation}
\label{eq:huvec}
\conv(S)=\left\{x\in\mathbb{R}^q\,\middle|\,\|x\|_1\le rK,  \|x\|_{\infty}\le r\right\}.
\end{equation}
\end{proposition}
\begin{proof}{Proof.}
Observe that $S/r=\{x\in\mathbb{R}^q\mid \textup{card}(x)\le K, \|x\|_{\infty}\le 1\}$. 
Then, 
\[
\conv(S)=r\conv(S/r)=\{y\in\mathbb{R}^q\mid \|u(y)\|_\infty\le r\}
=\{y\in\mathbb{R}^q\mid \max\{|y|_{[1]},s(y)_{1}\}\le r\},
\] 
where the second equality follows from Proposition~\ref{prop:c-normset} and Theorem~\ref{thm:cnorm_snorm} and the third equality from 
Part~\ref{proppart:u1} of Proposition~\ref{prop:umajx}.
Since $s(y)_1=\frac{1}{K}\sum_{j=1}^q |y|_{[j]}$ by definition of $s(y)$, the result follows.
\hfill  \Halmos
\end{proof}


When $\|\cdot\|_s$ is the $\ell_2$-norm, the norm $\|\cdot\|_c$ associated with $\conv(N^K_{\|\cdot\|_2})$ is known to be the \emph{$K$-support norm} (or \emph{$K$-overlap norm}).  
An explicit formula for this norm is introduced in \cite{argyriou2012sparse}.
We next provide an alternate derivation of this formula using our arguments.
For consistency with literature, we denote the $K$-support norm by $\|\cdot\|^{sp}_{K}$.

\begin{lemma}
\label{lemma:uniquer}
The unique integer $r\in \{0,\dots,K-1\}$ that satisfies
\begin{equation}
\label{randix}
|x|_{[K-r-1]}> s(x)_{K-r}\ge |x|_{[K-r]}.
\end{equation}
is $r=K-i_x$ where $|x|_{[0]}=\infty$ by convention.
\end{lemma}
\begin{proof}{Proof.}
This result follows from Proposition~\ref{prop:umajx} and we refer to its parts directly in the proof.
Let $|x|_{[i]} \le s(x)_{i} < |x|_{[i-1]}$ for some $i\in \{1,\ldots,K\}$. We show that $i=i_x$. First, we show that $|x|_{[j]} \le s(x)_j$ for all $j \ge i$. Let $j+1$ be the first index no less than $i$ so that $|x|_{[j+1]} > s(x)_{j+1}$ and observe that, in fact, $j+1 > i$. Then, we obtain the contradiction $|x|_{[j+1]} \le |x|_{[j]} \le s(x)_j \le s(x)_{j+1} < |x|_{[j+1]}$, where the third inequality is by Part~\ref{proppart:sssx}.
Therefore, $|x|_{[j]}\le s(x)_j$ for $j\ge i$ which implies by Part~\ref{proppart:sssx} that $s(x)_i\le \cdots \le s(x)_K$ and by Part~\ref{proppart:sv} that $i\ge i_x$. Since $|x|_{[i-1]} > s(x)_i$, it follows by Part~\ref{proppart:sssx} that either $i=1$ or $s(x)_{i-1} > s(x)_{i}$. In either case, it follows that $i\le i_x$. Therefore, $i=i_x=K-r$.
\hfill \Halmos
\end{proof}

\begin{proposition}[Proposition 2.1 of \cite{argyriou2012sparse}]
\label{prop:ksupport}
\begin{equation}
\label{eq:ksupport}
\|x\|^{sp}_K = \left(\sum_{i=1}^{K-r-1}x_{[i]}^2+\frac{1}{r+1}\left(\sum_{i=K-r}^n |x|_{[i]}\right)^2\right)^{\frac{1}{2}}.
\end{equation}
where $r$ is the unique integer in $\{0,\dots,K-1\}$ satisfying \eqref{randix}.
\end{proposition}
\begin{proof}{Proof.}
By Theorem~\ref{thm:cnorm_snorm},
\[\begin{array}{ll}
\|x\|^{sp}_K = \|u(x)\|_2&\displaystyle=\left(\sum_{i=1}^{i_x-1} |x|_{[i]}^2+(K-i_x+1)\delta(x)^2\right)^{\frac{1}{2}}=\left(\sum_{i=1}^{i_x-1} |x|_{[i]}^2+\frac{1}{K-i_x+1}\left(\sum_{i=i_x}^{n}|x|_{[i]}\right)^2\right)^{\frac{1}{2}}.
\end{array}\]
The result then follows since Lemma~\ref{lemma:uniquer} establishes that $r=K-i_x$. 
\hfill \Halmos
\end{proof}

\def\ranl{\textup{ran}\lambda}
\def\S{\bar{S}}

\subsection{Convexification of sets of matrices characterized by their singular values}\label{subsec:convexifysingular}

\def\sk{s^k}
Let $\mathcal{M}^{m,n}(\mathbb{R})$ be the set of $m\times n$ real-valued matrices.
For $M\in\mathcal{M}^{m,n}(\mathbb{R})$, 
let $\sigma_1(M)\ge \dots \ge \sigma_q(M)$ denote the singular values of $M$ where $q=\min\{m,n\}$
and let $\sigma:\mathcal{M}^{m,n}(\mathbb{R})\rightarrow \mathbb{R}^q$ be defined as $\sigma(M)=(\sigma_1(M),\dots,\sigma_q(M))$.
Let $\|M\|_{sp}=\sigma_1(M)$ and $\|M\|_*=\sum_{i=1}^q\sigma_i(M)$ be the \emph{spectral norm} and the \emph{nuclear norm} of $M$, respectively. In this subsection, we consider sets of matrices that are characterized by their singular values. More specifically, we are interested in sets of the form $\bar{S}=\{M\in\mathcal{M}^{m,n}(\mathbb{R})\mid f_i\bigl(\sigma(M)\bigr)\le 1, i=1,\dots,r\}$ and their convex hulls where each $f_i$ is a sign- and permutation-invariant function.
Define $S=\{x\in\mathbb{R}^q\mid f_i(x)\le 1, i=1,\dots,r\}$ where $q=\min\{m,n\}$.
It is clear that $M\in\bar{S}$ if and only if $\sigma(M)\in S$. As we show next, Theorem~\ref{permconvset} implies that convex hulls of sets of the form $\S$ can be obtained by studying $S$ instead. Similar results, although dealing with closed convex hulls, can be derived using conjugacy results of \cite{lewis1995}. We include a direct proof based on Theorem~\ref{permconvset}.

\begin{theorem}
\label{thm:vectomat}
For $p\in\mathbb{Z}_{++}$ and $q\in\mathbb{Z}_+$ and each $i\in \{1,\ldots,r\}$, let $f_i: (x,z)\mapsto \R$, where $(x,z)\in \R^p\times\R^q$, be sign- and permutation-invariant functions with respect to $x\in\R^{p}$. Let $m, n\in \mathbb{Z}$, such that $\min\{m,n\} = p$, and define   $\bar{S}=\bigl\{(M,z)\in\mathcal{M}^{m,n}(\R)\times\R^q\bigm| f_i\bigl(\sigma(M), z\bigl)\le 1, i=1,\dots,r\bigr\}$. Further, define
$S=\bigl\{(x,z)\in \R^p\times \R^q\mid f_i(x,z)\le 1, i=1,\dots,r\bigr\}$.
Then,
\[
\conv\bigl(\bar{S}\bigr)= X:=\bigl\{(M,z)\in\mathcal{M}^{m,n}(\R)\times \R^q\bigm| (\sigma(M), z)\in\conv(S)\bigr\}.
\]
\end{theorem}
\begin{proof}{Proof.}
We first show that $\conv(\S)\subseteq X$.
Consider $(M,z)\in \conv\bigl(\S\bigr)$, so that $(M,z)=\sum_{j}\gamma_j (M^j,z^j)$, where $(M^j,z^j)\in \S$ and $\gamma_j$ are convex multipliers. For $k\in \{1,\ldots,p\}$ and any $Y\in \mathcal{M}^{m,n}(\R)$, define $\sk(Y) := \sum_{i=1}^k \sigma(Y)_{[i]}$ to be the $k^{\text{th}}$ Ky Fan norm. Then, it follows by sublinearity and positive-homogeneity of norms that $\sk(M) \le \sum_j \gamma_j\sk(M^j)$. In other words, $\sum_j \gamma_j\sigma(M^j) \ge_{wm} \sigma(M)$. Let $\sigma^j = \sigma(M^j)$. Since $(M^j,z^j)\in \S$, it follows that $f_i(\sigma^j,z^j)\le 1$. Therefore, $\bigl(\sum_j\gamma_j\sigma^j, z\bigr)\in \conv(S_0)$ where $S_0 = S\cap \{(\sigma, z)\mid \sigma_1\ge \cdots \ge \sigma_p \ge 0\}$. Then, it follows by Part 3 of Theorem~\ref{permconvset} that $\bigl(\sigma(M), z\bigr)\in \conv(S)$ and $(M,z)\in X$.

We now show that $\conv(\S) \supseteq X$. Let $(M, z) \in X$ and let $U\diag(\sigma) V^\intercal$ be the singular value decomposition of $M$, where $\diag(\sigma)\in \mathcal{M}^{m,n}(\mathbb{R})$ is the diagonal matrix, whose diagonal is the vector $\sigma$ and $\sigma_1\ge \ldots\ge \sigma_p \ge 0$. Since $(\sigma,z)\in \conv(S)$, it follows by Part 3 of Theorem~\ref{permconvset} that there exist $\sigma'\in \R^p$, $(\sigma^j, z^j)\in S_0$ and convex multipliers $\gamma_j$ so that $(\sigma',z) = \sum_j \gamma_j(\sigma^j,z^j)$ and $\sigma' \ge_{wm} \sigma$. Now, if $\theta^j$ is obtained by permuting $\sigma^j$ or changing the sign of some of its entries, it follows readily that $(U\diag(\theta^j)V^T,z^j) \in \S$, because these operations do not alter the singular values of the matrix. Since $\sigma' \ge_{wm} \sigma$,  $\sigma = \sum_{k}\chi_k T_k \sigma'$, where $\chi_k$ are convex multipliers and each $T_k$ permutes the entries of $\sigma'$ and possibly changes the sign of a few of the entries. Then, it follows $(\sigma,z) = \sum_{k}\chi_k (T_k\sigma', z) = \sum_{k}\sum_j \chi_k\gamma_j (T_k\sigma^j, z^j)$. Since $U\diag(\theta)V^T$ is a linear operator of $\theta$, $\sum_{j}\sum_{k}\chi_k\gamma_j = 1$, and we have already shown that $(U\diag(T_k\sigma^j)V^T,z^j) \in \S$, it follows that $(M,z)\in \conv(\S)$.
\hfill\Halmos
\end{proof}

\def\mS{\mathcal{S}}
In the following, we denote the set of $p\times p$ real symmetric matrices as $\mS^{p}$ and, for any $M\in \mS^p$, we denote the eigenvalues as $\lambda(M)$. In this context, a similar result can be shown using eigenvalues instead of singular values.

\begin{theorem}
\label{thm:vectomateigen}
For $p\in\mathbb{Z}_{++}$ and $q\in \mathbb{Z}_+$ and each $i\in \{1,\ldots,r\}$, let $f_i: (x,z)\mapsto \R$, where $(x,z)\in \R^p\times \R^q$ be permutation-invariant functions with respect to $x\in\R^{p}$. Define   $\bar{S}=\bigl\{(M,z)\in\mS^p\times\R^q\bigm| f_i\bigl(\lambda(M), z\bigl)\le 1, i=1,\dots,r\bigr\}$. Further, define
$S=\bigl\{(x,z)\in \R^p\times \R^q\mid f_i(x,z)\le 1, i=1,\dots,r\bigr\}$.
Then,
\[
\conv\bigl(\bar{S}\bigr)= X:=\bigl\{(M,z)\in\mS^p\times \R^q\bigm| \bigl(\lambda(M), z\bigr)\in\conv(S)\bigr\}.
\]
\end{theorem}
\begin{proof}{Proof.}
We only provide a proof sketch since the proof is similar to that of Theorem~\ref{thm:vectomat}. To show that $\conv(\S)\subseteq X$, we consider $(M,z) = \sum_j \gamma_j (M^j,z^j)\in\conv(\S)$ and use the fact that $\sum_j \gamma_j \lambda(M^j)\ge_m\lambda(M)$ \cite[Theorem 4.3.27]{hj85}. Then, the result follows from Part 1 of Theorem~\ref{permconvset}. To show that $\conv(\S)\supseteq X$, we consider $(M,z)\in X$ and express $\sum_j \gamma_j (\lambda^j, z^j)\ge_m (\lambda(M),z)$ for some $(\lambda^j,z^j)\in S$. Then, we observe that this implies that for any orthogonal matrix $U$ and a permutation $\pi$, $(U\diag(\pi(\lambda^j))U^\intercal, z)\in \S$. The result is then derived in a manner similar to that in the proof of Theorem~\ref{permconvset} except that instead of using the singular value decomposition $U\diag(\sigma)V^\intercal$ of $M$, we use the eigenvalue decomposition $M=U\lambda(M) U^\intercal$. 
\hfill\Halmos
\end{proof}


The rank of a matrix can be represented as the cardinality of the vector of singular values.
Since cardinality is a sign- and permutation-invariant function, we obtain the following result as a special case of Theorem~\ref{thm:vectomat}.
\begin{corollary}
\label{cor:matconvhull}
Let $\bar{S}=\{M\in\mathcal{M}^{m,n}(\mathbb{R})\mid\textup{rank}(M)\le K, \|\sigma(M)\|_s\le r\}$.
Then, 
\begin{equation*}
\conv(\bar{S})=\left\{M\in\mathcal{M}^{m,n}(\mathbb{R})\mid \|\sigma(M)\|_c\le r\right\}.
\tag*{\Halmos}
\end{equation*}
\end{corollary}

Consider $\bar{S}$ in Corollary~\ref{cor:matconvhull}. Recall that determining if an arbitrary matrix $M\in\mathcal{M}^{m,n}(\mathbb{R})$ is in the convex hull $\conv(\bar{S})$ can be easily done when the norm $\|\cdot\|_s$ is calculable. In particular, when $\|\cdot\|_s$ is the Euclidean norm, a given matrix $M$ is in $\conv(\bar{S})$ if $\|\sigma(M)\|^{sp}_K\le r$; 
see \eqref{eq:ksupport} for an explicit formula for $\|\cdot\|^{sp}_K$.
Semidefinite representability of the convex hull will be discussed in Section~\ref{subsec:sd_rep}.

Next, we consider the special case where $\|\cdot\|_s$ is the $l_\infty$ norm. Proposition~\ref{prop:hiriartx} and Theorem~\ref{thm:vectomat} together give an alternative proof for the following result.

\begin{proposition}[Theorem 1 of \cite{hiriart2012}]
\label{hiriart}
Let $\bar{S}=\{M\in\mathcal{M}^{m,n}(\mathbb{R})\mid \textup{rank}(M)\le K, \|M\|_{sp}\le r\}$. Then, $\textup{conv}(\bar{S})
=\{M\in\mathcal{M}^{m,n}(\mathbb{R})\mid \|M\|_*\le rK, \|M\|_{sp}\le r\}$.\hfill\Halmos
\end{proposition}

\subsection{Semidefinite-representability of sets of matrices characterized by their singular values}
\label{subsec:sd_rep}
We presented in Corollary~\ref{cor:matconvhull} a convex hull result for a set of matrices $\bar{S}$ that is described using their singular values.
The resulting convex hull is written in a norm $\|\cdot\|_c$ induced by the defining norm $\|\cdot\|_s$ of $\bar{S}$.
In this subsection, we discuss the representability of this set as the feasible set of a semidefinite programming (SDP) problem.
A set is called \emph{semidefinite-representable} if it is a projection of a set expressed by a linear matrix inequality.
We remark the following well-known results about semidefinite-representability; see Section 4.2 of \cite{ben2001}. 
\begin{lemma}
\label{lemma:sdprep}
The following sets are semidefinite-representable:
\begin{enumerate}
\item The epigraph of the sum of $p$ largest singular values of a rectangular matrix.
\item The epigraph of the sum of $p$ largest eigenvalues of a symmetric matrix. 
\item The graph of the sum of all eigenvalues of a symmetric matrix. 
\item The set $A\cap B$, where $A$ and $B$ are semidefinite-representable.\hfill\Halmos
\end{enumerate}
\end{lemma}

In particular, we consider the set: $\mathcal{S}=\{M\in\mathcal{M}^{m,n}(\mathbb{R})\mid \textup{rank}(M)\le K, f(\sigma(M))\le r\}$
where $q=\min\{m,n\}$ and $f:\Delta_+\rightarrow\mathbb{R}$ is a quasiconvex function. 
A function $f$ is said to be \emph{quasiconvex} on $\Delta_+$ if, for $\lambda\in [0,1]$ and $x,y\in \Delta_+$, we have $f(\lambda x + (1-\lambda)y) \le \max\{f(x),f(y)\}$. We assume this function has semidefinite-representable lower-level sets. Then, we show that the convex hull of $\mathcal{S}$, the set of rank-constrained matrices whose singular values belong to a lower-level set of $f$ are semidefinite-representable.
\begin{theorem}
\label{thm:sdprep}
Let $q=\min\{m,n\}$ and $\bar{S}=\{M\in\mathcal{M}^{m,n}(\mathbb{R})\mid \textup{rank}(M)\le K, f(\sigma(M))\le r\}$, where $f:\Delta_+\rightarrow\mathbb{R}$ has semidefinite-representable lower-level sets.
Then, $\textup{conv}(\bar{S})$ is semidefinite-representable.
\end{theorem}
\begin{proof}{Proof.}
Define $g:\R^q\mapsto \R$ so that $g(x) = f(x_{\Delta_+})$. Observe further that $g$ is sign- and permutation-invariant. Then, let $S = \{x\in\mathbb{R}^q\mid \textup{card}(x)\le K, g(x)\le r\}$.
By sign- and permutation-invariance of $S$ and Theorem~\ref{permconvset},
\[
\textup{conv}(S)=\left\{x\in\mathbb{R}^q\,\middle|\ \begin{array}{l}
f(u)\le r,\\
u_1\ge \dots\ge u_K\ge 0,\\
u_{K+1}=\dots=u_n=0,\\
u\ge_{wm} |x|
\end{array}
\right\}.
\]
Therefore, by Theorem~\ref{thm:vectomat},
\[
\textup{conv}(\bar{S})=\left\{M\in\mathcal{M}^{m,n}(\mathbb{R})\,\middle|\ \begin{array}{l}
f(u)\le r,\\
u_1\ge \dots\ge u_K\ge 0,\\
u_{K+1}=\dots=u_n=0,\\
u\ge_{wm} \sigma(M)
\end{array}
\right\}.
\]
By the definition of weak majorization, the convex hull has the following representation:
\begin{equation}
\left\{\begin{array}{l}
f(u)\le r,\\
u_1\ge \dots \ge u_K\ge 0,\\
u_{K+1}=\dots=u_n=0,\\
\sum_{i=1}^j u_i\ge \sum_{i=1}^j \sigma_j(M),\quad j=1,\dots,K.\\
\end{array}\right.
\label{sdrresult_svd}
\end{equation}
The semidefinite-representability of \eqref{sdrresult_svd} follows from Lemma~\ref{lemma:sdprep} and the semidefinite-representability of the level set $\{u\mid f(u)\le r\}$ and the introduced linear inequalities.
\hfill\Halmos
\end{proof}

Although Theorem~\ref{thm:sdprep} is similar to Proposition 4.2.2 in \cite{ben2001}, we discuss next how these results differ. 
First, we introduce a rank constraint and thus treat a nonconvex set.
Second, we discuss the representation of the convex hull rather than the set itself. Third, we do not require monotonicity of $f(x)$ and require semidefinite representability only over $\Delta_+ (=\Delta^q_+)$ instead of $\R^q_+$. We briefly describe why the added assumptions are required in Proposition 4.2.2 in \cite{ben2001}, but not in our result. This is because, when $f(x)$ is not monotone but is quasiconvex over $\Delta_+$, its extension to $\Re^q_+$ defined using $g(x) := f(x_{\Delta_+})$ is not necessarily quasiconvex. As such, the lower-level sets of $g(x)$ are not always semidefinite representable. 
To see this, consider $f(x) = 1-x$, where $x\in \Re$. Then, $\{x\mid g(x)\le 0\}$ is not a convex set because $g(1) = g(-1) = 0$, while $g(0)=1$. 
On the contrary, consider an $f(x)$ that is monotone, permutation-invariant and quasiconvex over $\R^q_+$. Let $C=\{x\mid g(x)\le r\}$. We argue that $C$ is convex and can be expressed as $X=\{x\mid f(u)\le r, u\in \Delta_+, u\ge_{wm} |x|\}$. First, observe that Theorem~\ref{permconvset} shows that $X=\conv(C)\supseteq C$. Now, we argue that $X\subseteq C$. To see this, assume $x\in X$. Then, $x_{\Delta_+}\in X$ because $X$, being $\conv(C)$ inherits the sign and permutation-invariance of $C$. Then, there exist $u$ and $u'$ such that $u\ge_m u'\ge x_{\Delta_+}$ and $f(u)\le r$.
Therefore, we have $g(x) = f(x_{\Delta_+}) \le f(u')\le f(u)\le r$, where the first inequality is from monotonicity of $f$. The second inequality is because $u$ and $u'$ are non-negative, $u'$ is in the convex hull of $u$ and its permutation-variants, and $f$ is quasiconvex and permutation-invariant. The third inequality follows by definition of $u$. Therefore, it follows that $x\in C$.


\begin{corollary}
\label{cor:snorm_sdprep}
Let $\mathcal{S}=\{M\in\mathcal{M}^{m,n}(\mathbb{R})\mid\textup{rank}(M)\le K, \|\sigma(M)\|_{s}\le r\}$ where $\|\cdot\|_s$ is a permutation-invariant monotone norm.
Then, $\conv(\mathcal{S})$ is semidefinite-representable.
In particular, when $\|\sigma(\cdot)\|_s$ is a Ky Fan $p$-norm (the sum of $p$ largest singular values) for some $p=1,\dots,\min\{m,n\}$, the convex hull is semidefinite-representable.
\hfill\Halmos
\end{corollary}

Similarly, we can prove the following result, where $\mS^q_+\in \mathcal{M}^{q\times q}(\mathbb{R})$ is the set of positive semidefinite symmetric matrices. 
\begin{theorem}
\label{thm:sdprepa}
Let $\bar{S}=\bigl\{M\in\mS^q_+\bigm| \textup{rank}(M)\le K, f\bigl(\lambda(M)\bigr)\le r\bigr\}$, where $f:\Delta_+\rightarrow\mathbb{R}$ has semidefinite-representable lower-level sets.
Then, $\textup{conv}(\bar{S})$ is semidefinite-representable.
\end{theorem}
\begin{proof}{Proof.}
Define $g:\R^q_+\mapsto \R$ so that $g(x) = f(x_\Delta)$. Then, let $S = \{x\in\mathbb{R}^q_+\mid \textup{card}(x)\le K, g(x)\le r\}$.
By permutation-invariance of $S$ and Theorem~\ref{permconvset},
\[
\textup{conv}(S)=\left\{x\in\mathbb{R}^q_+\,\middle|\ \begin{array}{l}
f(u)\le r,\\
u_1\ge \dots\ge u_K\ge 0,\\
u_{K+1}=\dots=u_n=0,\\
u\ge_{m} x
\end{array}
\right\}.
\]
Therefore, by Theorem~\ref{thm:vectomateigen},
\[
\textup{conv}(\bar{S})=\left\{M\in\mS^q\,\middle|\ \begin{array}{l}
f(u)\le r,\\
u_1\ge \dots\ge u_K\ge 0,\\
u_{K+1}=\dots=u_n=0,\\
u\ge_{m} \lambda(M)
\end{array}
\right\}.
\]
By the definition of majorization, the convex hull has the following representation:
\begin{equation}
\left\{\begin{array}{l}
f(u)\le r,\\
u_1\ge \dots \ge u_K\ge 0,\\
u_{K+1}=\dots=u_n=0,\\
\sum_{i=1}^j u_i\ge \sum_{i=1}^j \lambda_j(M),\quad j=1,\dots,K-1,\\
\sum_{i=1}^K u_j = \sum_{i=1}^q \lambda_i(M).
\end{array}\right.
\label{sdrresult_evd}
\end{equation}
The semidefinite-representability of \eqref{sdrresult_evd} follows from Lemma~\ref{lemma:sdprep} and the semidefinite-representability of the level set $\{u\mid f(u)\le r\}$ and the introduced linear inequalities.
\hfill\Halmos
\end{proof}
The ideas in the above proof can be extended, using disjunctive programming techniques, to symmetric matrices, that are not necessarily positive semidefinite. Since the eigenvalues are no longer non-negative, we cannot impose the restriction that $u\ge 0$ and, thus, assume that $u_{K+1},\ldots,u_n=0$. Instead, we express the rank constraint as a union of sets each of which satisfies $u_{a+1}=\cdots = u_{a+n-K} = 0$ for some $a\in \{0,\ldots,K\}$. Then, we obtain $\conv(S)$ as the convex hull of a union of semidefinite representable sets. Using the disjunctive programming argument of Proposition~3.3.5 in \cite{ben2001}, this yields a lifted representation of a set that outer-approximates $\conv(S)$ and is contained in $\cl\conv(S)$.

\section{Convex envelopes of nonlinear functions}\label{sec:nonlinear}

The problem of finding convex envelopes of nonlinear functions is central to the global solution of factorable problems through branch-and-bound.
When the domain over which the envelope is constructed is a polytope $P$, it is often the case that the envelope is completely determined by the values that the function takes on a subset of the faces of $P$, or more generally, a subset of its feasible points.
If the above property holds, disjunctive programming techniques can often be employed to provide an explicit (although typically large) description of the envelope, through the introduction of new variables for each of the important subsets of $P$.
In this section, using Theorem~\ref{permconvset} as a foundation,  
we show that for certain functions defined over permutation-invariant polytopes ($i$) envelopes can be built without recourse to disjunctive programming (Proposition~\ref{prop:schurhypercube}), and ($ii$) polynomially-size disjunctive programming formulations can be constructed even when the number of faces of $P$ important in the construction of the envelope is exponential (Theorem~\ref{res:limitfaceshypercube}).
These results yield compact envelopes descriptions for specific families of functions (Propositions~\ref{prop_env_nonlin1} and \ref{prop_env_nonlin2}).
We also provide numerical evidence that the use of these techniques produces relaxations of multilinear functions over permutation-invariant hypercubes that are significantly stronger than those obtained using a recursive application of McCormick's procedure.   
The techniques can be extended to handle epigraphs of singular values/eigenvalues of matrices using the ideas presented in Sections~\ref{subsec:convexifysingular} and \ref{subsec:sd_rep}.




\begin{definition}
\label{def:shurconcave}
A function $\phi:C\mapsto \R$ is said to be \emph{Schur-concave} on $C$, if for every $x,y\in C$,  $x\ge_m y$ implies that $\phi(x)\le \phi(y)$. \hfill\Halmos
\end{definition}

Various functions have been shown to be Schur-concave including the Shannon entropy $\sum_{i=1}^n x_i \log(\frac{1}{x_i})$ and  elementary symmetric functions $\sum_{J \subseteq \{1,\ldots,n\} : |J|=k}  \prod_{i \in J} x_i$.  Symmetric concave functions are also Schur-concave. More complex Schur-concave functions can be constructed from known Schur-concave functions using some compositions or operations that preserve Schur-concavity; see Chapter 3 of \cite{marshall2010} for further detail.

\def\gphi{\underaccent{\check}{\phi}}
\def\epi{\mathop{\text{epi}}\nolimits}
\def\vertices{\mathop{\text{vert}}}
In this section, for any function $\phi:C\mapsto\R$ we denote 
the convex envelope of $\phi$ over $C$ by $\conv_C(\phi)$. 
A common tool in the construction of convex envelopes is to restrict the domain of the function to a smaller subset. 
We say that a function $\phi:C\mapsto \R$ can be \emph{restricted to $X$}, where $X\subseteq C$, for the purpose of obtaining $\conv_C(\phi)$ if $\conv_C(\phi|_X) = \conv_C(\phi)$, where $\phi|_X(x)$ is defined as $\phi(x)$ for any $x\in X$ and $+\infty$ otherwise. 

First, we establish in Lemma~\ref{lem:restricttonou} that, when deriving the envelope of a Schur-concave function over a permutation-invariant polytope, it is sufficient to restrict attention to those points in the domain that are not majorized by other feasible points. 
When coupled with a simple domain structure, this result permits a description of convexification results without the use of majorization inequalities, in a smaller dimensional space.

\begin{lemma}\label{lem:restricttonou}
Let $\phi:P\mapsto\R$ be a Schur-concave function, where $P\subseteq \R^n$ is a permutation-invariant polytope. 
Let $M:=\{x \in P \mid \not\exists u\in P  \mbox{ with } u\ge_m x \mbox{ and }  u_\Delta\ne x_\Delta\}$. 
Let $S:=\{(x,\phi)\mid \phi(x)\le \phi\le \alpha, x\in P\}$ and 
$X:=\{(x,\phi)\mid \phi(x)\le \phi\le \alpha, x\in M\}$.
 Then, $\conv(S) = \conv(X)$.
\end{lemma}

\proof{Proof.}
Since $M\subseteq P$ it follows that $X\subseteq S$ and, therefore, $\conv(X)\subseteq \conv(S)$. Now, consider $(x',\phi')\in S\backslash X$. 
Therefore, $\phi(x')\le \phi'\le \alpha$ and $x'\in P\backslash M$. 
Let $x''_i = \frac{1}{n}\sum_{i'=1}^n x'_{i'}$ for all $i\in \{1,\ldots,n\}$. That is, all components of $x''$ are identical.
Let $u':=\arg\max\bigl\{\|u-x''\|\mid u\ge_m x', u\in P\bigr\}$. 
The maximum is achieved because the feasible set is compact and because the objective is upper-semicontinuous. 
Assume by contradiction that there exists $y'\in P$ such that $y'\ge_m u'$ and $y'_{\Delta}\ne u'_{\Delta}$. 
Since $u'$ can be written as a convex combination of at least two permutations of $y'$ and the objective of the problem defining $u'$ is permutation-invariant and strictly convex, it follows that $\|y'-x''\| > \|u'-x''\|$ violating the optimality of $u'$. 
Therefore, there does not exist $y'\in P$ such that $y'\ge_m u'$
and $y'_{\Delta}\ne u'_{\Delta}$. 
In other words, $u'\in M$.
It follows that, for any permutation matrix $Q\in\mathcal{P}_n$, $Qu'\in M$. 
Since $\phi$ is Schur-concave, then $\phi(Qu')=\phi(u')\le \phi(x')\le \phi'\le \alpha$. 
Therefore, $(Qu',\phi')\in X$. 
Finally, since $x'\le_m u'$, $x'$ can be written as a convex combination of permutations of $u'$. 
Therefore, $(x',\phi')\in\conv(X)$.
We conclude that $S\subseteq \conv(X)$.
\hfill\Halmos
\endproof


Lemma~\ref{lem:restricttonou} requires the identification of the set $M$, which consists of points in $P$ which are not expressible as a convex combination of another point in $P$ and its permutations. In Lemma~\ref{lem:nou}, we characterize such points as those which have a supporting hyperplane of a specific form. Later, we discuss how these results can be combined to obtain, in closed-form, convex envelopes of various Schur-concave functions. 

\begin{lemma}\label{lem:nou}
Let $x'\in P$, where $P$ is a permutation-invariant polytope. Let $\pi$ be a permutation of $\{1,\ldots,n\}$ such that for each $i\in \{1,\ldots,n-1\}$, $x'_{\pi(i)} \ge x'_{\pi(i+1)}$. Then, there exists $u' \in P$ with $u'\ge_m x'$ and $u'_{\Delta} \neq x'_{\Delta}$ if and only if there does not exist $a\in \R^n$ such that $a_{\pi(i)} > a_{\pi(i+1)}$ for all $i\in\{1,\ldots,n-1\}$ and $\sum_{i=1}^n a_i(x_i-x'_i) \le 0$ is valid for $P$. 
\end{lemma}
\begin{proof}{Proof.}
We first show that if such $a$ exists, there does not exist $u'\in P$ such that $u'\ge_m x'$ and $u'_{\Delta} \neq x'_{\Delta}$. 
Assume by contradiction that such a $u'$ exists.
Because $P$ is permutation-invariant we may assume that $u'_{\pi(i)} \ge u'_{\pi(i+1)}$ for all $i\in \{1,\ldots,n-1\}$ by sorting $u'$ if necessary. 
Since $u'\ge_m x'$ and $u'\ne x'$, there exists $y'$, $\theta>0$, $k\in \{1,\ldots,n-1\}$, and $r\in\{1,\ldots,n-k\}$ such that $u'\ge_m y'\ge_m x'$, $y'_{\pi(k)} = x'_{\pi(k)} + \theta$, $y'_{\pi(k+r)} = x'_{\pi(k+r)} - \theta$ and $y'_i = x'_i$ otherwise; see Lemma 2.B.1 in \cite{marshall2010}.
Since $u'\in P$, $P$ is convex and permutation-invariant, and since $y'$ can be written as a convex combination of $u'$ and its permutations, it is clear that $y'\in P$. 
Therefore, $\sum_{i=1}^n a_i(y'_i - x'_i) \le 0$ or $a_{\pi(k)} - a_{\pi(k+r)} \le 0$. 
This is a contradiction to the assumed ordering of $a$.

Now, we show that if there does not exist $u'\in P$ such that $u'\ge_m x'$ and $u'_{\Delta} \neq x'_{\Delta}$, we can construct such a vector $a$. 
Define a polyhedral cone $K:= \{ \sum_{i=1}^{n-1}\alpha_{\pi(i)} \bigl(e_{\pi(i)} - e_{\pi(i+1)}\bigr)\mid \alpha\ge 0\}$ where $e_j$ represents the $j$th standard basis. Observe that $\nu\in K$  (equivalently, $\nu\ge_K 0$) implies that $\nu\ge_m 0$. More specifically, given $\nu\in \Re^n$, let $S^\pi_k[\nu]$ be the partial sums so that $S^\pi_k[\nu] = \sum_{i=1}^k \nu_{\pi(i)}$. Then, $\nu\in K$ if and only if $S^\pi_k[\nu]$ is non-negative for every $k$ and $S^\pi_n[\nu] = 0$. 
To verify that the latter set is contained in $K$, write any $\nu$ in that set as $\nu  = \sum_{k=1}^{n-1} S^\pi_k[\nu] \bigl(e_{\pi(k)} - e_{\pi(k+1)}\bigr)$ to see that $\nu\in K$. To show the reverse inclusion, verify that $S^\pi_k\bigl[e_{\pi(i)} - e_{\pi(i+1)}\bigr]\ge 0$, where the last inequality is satisfied at equality.
Since $\sum_{i=1}^k \nu_{[i]}\ge S^\pi_k[\nu]$ for all $k$ and $\sum_{i=1}^n \nu_{[i]}= S^\pi_n[\nu]$, $\nu\in K$ implies that $\nu\ge_m 0$.
We next construct the polyhedron $C:=P-K-\{x'\}$, where the difference is the Minkowski difference. 
Since $x'\in P\cap (K+\{x'\})$, it follows that $0\in C$. 
Let $\langle a',x\rangle \le 0$ define the minimal (possibly trivial) face $F$ of $C$ containing $0$. 
We show next that $a$ can be chosen to be $a'$. 
First, note that for $x\in P$, $x-x'\in C$. Therefore, as claimed, $\langle a', x-x'\rangle \le 0$.
Since $e_{\pi(i+1)} - e_{\pi(i)}\in C$ for all $i\in \{1,\ldots,n-1\}$, we have that 
$a'_{\pi(i+1)} - a'_{\pi(i)} \le 0$. 
We now show that the inequalities are strict otherwise there exists a point in $w\in C$ so that $w\ge_K 0$ such that there is a $k$ for which $S^\pi_k[w] > 0$. This suffices because, if there exists $w=u-v\in C\cap K$, where  $u\in P$ and $v-x'\in K$. Then, we have $u\ge_K v\ge_K x'$,
which in turn proves the existence of $u\in P$ that majorizes $x'$. 
Furthermore, by showing $S^\pi_k[w] > 0$ for some $k$, we have $u_\Delta\ne x'_\Delta$, which contradicts our assumption that such a $u$ does not exist.
Assume that the inequality is not strict for some $k\in \{1,\ldots,n-1\}$ so that $\langle a', e_{\pi(k+1)} - e_{\pi(k)}\rangle = 0$. 
Then, for $\epsilon > 0$, we define $w_\epsilon=\epsilon \bigl(e_{\pi(k)} - e_{\pi(k+1)}\bigr)\in K$ and show that, for a sufficiently small $\epsilon$, $w_\epsilon\in C$. Since $-w_\epsilon\in F$ and $0$ is in the relative interior of $F$, it follows that there exists an $\epsilon$ such that $w_\epsilon\in C$. Moreover, $S^\pi_k[w_\epsilon] = \epsilon > 0$.
As we argued above, the existence of such a $w_\epsilon$ contradicts our assumption and, so, we conclude that $a'_{\pi(i+1)} - a'_{\pi(i)} < 0$ for all $i\in\{1,\ldots,n-1\}$.
\hfill\Halmos
\end{proof}

Lemmas~\ref{lem:restricttonou} and \ref{lem:nou} can be combined to develop convex envelopes of Schur-concave functions. This is because, taken together, they prove that it suffices to restrict attention to a subset $M$ of $P$ in order to construct the convex envelope. To better understand the structure of $M$ and to illustrate potential applications, we derive in Proposition~\ref{prop:schurhypercube}  the closed-form description of the convex envelope of a Schur-concave function over $[a,b]^n$. In this case, $M$ is contained in the one-dimensional faces of $[a,b]^n$, the key insight that allows for the derivation of the closed-form. Although we provide a self-contained argument for this fact in the proof of Proposition~\ref{prop:schurhypercube}, this inclusion can be seen as a special case of Lemma~\ref{lem:nou}, a visualization which serves to illustrate the use of Lemma~\ref{lem:nou} in characterizing $M$. To see this special case, observe that $x'\in M\cap \Delta^n$ only if there is an inequality $\langle\beta,x-x'\rangle\le 0$ valid for $[a,b]^n$ such that $\beta_1 >\cdots > \beta_n$. Since such an inequality is tight at $x'$, it can be derived as a conic combination of facet-defining inequalities tight at $x'$. The facet-defining inequalities of $[a,b]^n$ are $-x_i\le -a$ and $x_i\le b$ for $i=1,\ldots,n$. Let $F(x')$ be the set of facet-defining inequalities tight at $x'$, and for any facet-defining inequality in this set, say $\langle \alpha,x-x'\rangle\le 0$, let $L_\alpha =\{(t,t+1)\mid \alpha_t > \alpha_{t+1}\}$. It is easy to see that, for $x'\in [a,b]^n$, $|L_\alpha| \le 1$ for each $\alpha\in F(x')$, \ie{} each $L_\alpha$ contains at most one pair. Then, for $\beta$ to be derived as a conic combination of inequalities in $F(x')$, it must be that $\{(i,i+1)\}_{i=1}^{n-1}\subseteq \bigcup_{\alpha \in F(x')} L_\alpha$. Since $|L_\alpha| \le 1$, it follows that $|F(x')| \ge n-1$. Therefore, $M$ is a subset of one-dimensional faces of the hypercube. More generally, a similar argument shows that if there exists a $k\in \{1,\ldots,n-1\}$ such that $|L_\alpha|\le k$, then $|F(x')| \ge \bigl\lceil\frac{n-1}{k}\bigr\rceil$, and, consequently, $M$ is a subset of $n-\bigl\lceil\frac{n-1}{k}\bigr\rceil$ faces of $P$.

\begin{proposition}\label{prop:schurhypercube}
Consider a  function $\phi(x):\R^n\mapsto\R$ that is Schur-concave over $[a,b]^n$ and let $S^\alpha:=\{(x,\phi)\mid \phi(x)\le\phi\le \alpha, x\in [a,b]^n\}$. 
For any $x\in [a,b]^n$, define $S(x) = \sum_{i=1}
^n (x_i-a)$. 
For any $s\in [0,n(b-a)]$, let $i^{s}=\max\bigl\{i\mid i(b-a) < s\bigr\}$ 
and
\begin{equation}\label{eq:defnux}
  u^{s}_i = \begin{cases}
b & \text{if } i \le i^s\\
a + s - (b-a)i^s & \text{if } i = i^s + 1\\
a & \text{otherwise.}
\end{cases}
\end{equation}
Let 
$\Theta^\alpha := \bigl\{(x,\phi)\mid \phi\bigl(u^{S(x)}\bigr)\le \phi\le \alpha, x\in [a,b]^n\bigr\}$.
Then, 
$\conv\bigl(S^\alpha\bigr) = \conv(\Theta^\alpha)$. Moreover, if $\phi$ is  component-wise convex then
$\Theta^\alpha$ is convex.
\end{proposition}
\proof{Proof.}
We first show that $u^{S(x')}\ge_m x'$ for each $x' \in [a,b]^n$.
This follows because $u^{S(x')}$ simultaneously maximizes the continuous knapsack problems $\max\bigl\{\sum_{i=1}^j x_i\mid \sum_{i=1}^n x_i = S(x') + na, x\in [a,b]^n\bigr\}$ for all $j$ because the ratio of objective and knapsack coefficient of $x_i$ reduces with increasing $i$, and $x'$ 
is a feasible solution to these knapsack problems.

We now show that $S^\alpha\subseteq \Theta^\alpha$. Let $(x',\phi')\in S^\alpha$. Therefore, $\phi(u^{S(x')})\le \phi(x')\le \phi\le \alpha$, where the first inequality follows from Schur-concavity of $\phi$ and $u^{S(x')}\ge_m x'$, and the remaining inequalities follow because $(x',\phi')$ is feasible to $S^{\alpha}$. Therefore, $(x',\phi')\in \Theta^\alpha$.

Now, we show that $\Theta^\alpha\subseteq \conv\bigl(S^\alpha\bigr)$. Let $(x',\phi')\in \Theta^\alpha$. Since $u^{S(x')}\in [a,b]^n$ and $\phi(u^{S(x')})\le \phi'\le \alpha$, we conclude that $(u^{S(x')},\phi')\in S^\alpha$. 
Then it follows that $(x',\phi')\in \conv\bigl(S^\alpha\bigr)$ since $u^{S(x')}\ge_m x'$ implies that $x'$ can be written as a convex combination of permutations of $u^{S(x')}$ and $S^\alpha$ is permutation-invariant in $x$. 

\def\x{\bar{x}}
\def\phibar{\bar{\phi}}
\def\s{\bar{s}}
\def\i{k}
To show that $\Theta^{\alpha}$ is convex when $\phi$ is component-wise convex, we write $\Theta^\alpha$ as $\proj_{(x,\phi)} \Xi^\alpha$, where $\Xi^\alpha = \bigl\{(x,s,\phi)\mid \varphi(s)\le \phi\le \alpha, x\in [a,b]^n, s=\sum_{i=1}^n (x_i-a)\bigr\}$ and $\varphi(s)=\phi(u^s)$. 
The result follows if $\varphi(s)$ is convex over $[0,n(b-a)]$ since $\Theta^\alpha$ is expressed as the projection of the convex set $\Xi^\alpha$. 
First, observe that, for $s\in \bigl(i(b-a), (i+1)(b-a)\bigr)$, the convexity of $\varphi(s)$ follows from the assumed convexity of $\phi(u^s)$ since, in this interval, $u^s$ varies only along the $i^{\text{th}}$ coordinate. 
Now, choose $\i\in \{0,\ldots,n-1\}$ and let $\s=\i(b-a)$. 
To prove the result, it suffices to verify that the left derivative of $\varphi(s)$ at $\s$ is no more than the corresponding right derivative. 
For any $\epsilon > 0$, observe that $u^{\s}+\epsilon e_{\i} \ge_m u^{\s}+\epsilon e_{\i+1}$. This follows because $(1-\lambda)(u^{\s}_k+\epsilon,u^{\s}_{k+1}) + \lambda(u^{\s}_{k+1},u^{\s}_{k}+\epsilon) = (u^{\s}_k,u^{\s}_{k+1}+\epsilon)$, where $0 < \lambda=\frac{\epsilon}{u^{\s}_k-u^{\s}_{k+1}+\epsilon}\le 1$ showing that $u^{\s}+\epsilon e_{\i+1}$ can be expressed as a convex combination of 
$u^{\s}+\epsilon e_{\i}$ and its permutations.
Since $\phi(\cdot)$ is Schur-concave, it follows that $\phi(u^{\s}+\epsilon e_{\i}) \le \phi(u^{\s} + \epsilon e_{\i+1}) = \varphi(\s+\epsilon)$.
Then, the following chain of inequalities follows
\begin{equation*}
  \lim_{\epsilon \downarrow 0}\frac{\varphi(\s) - \varphi(\s-\epsilon)}{\epsilon}= \lim_{\epsilon \downarrow 0}\frac{\phi(u^{\s})-\phi(u^{\s}-\epsilon e_{\i})}{\epsilon} \le \lim_{\epsilon \downarrow 0}\frac{\phi(u^{\s}+\epsilon e_{\i})-\phi(u^{\s})}{\epsilon} \le \lim_{\epsilon\downarrow 0} \frac{\varphi(\s+\epsilon) - \varphi(\s)}{\epsilon},
\end{equation*}
where the first equality is by the definition of $\varphi(\cdot)$ and $u^{\s}$, the first inequality is from the assumed convexity of $\phi(\cdot)$ when the argument is perturbated only along the ${\i}^{\text{th}}$ coordinate, and the second inequality is because $\phi(u^{\s}+\epsilon e_{\i}) \le \varphi(\s+\epsilon)$ and $\phi(u^{\s})=\varphi(\s)$.
\hfill\Halmos
\endproof

In essence, Proposition~\ref{prop:schurhypercube} shows that we can reduce our attention to the edges of the hypercube belonging to $\Delta^n$ in our construction of $\conv(S^\alpha)$;  
see Figure~\ref{fig:setofmajpoints} for an illustration when $a=2$ and $b=5$.
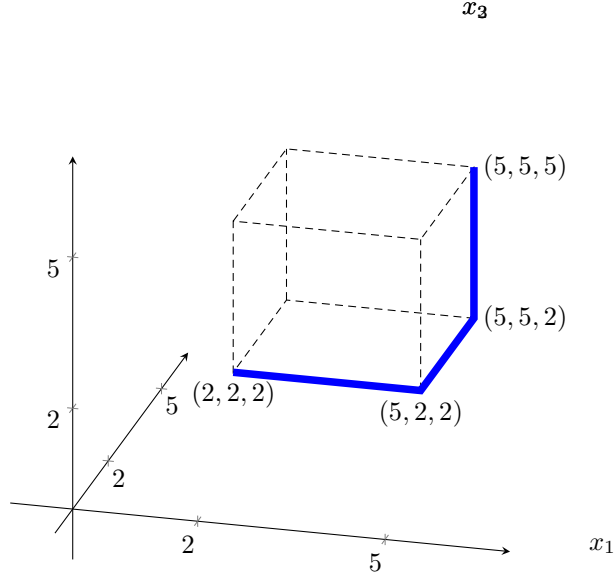
\begin{figure}
\centering
\begin{tikzpicture}
\begin{axis}[
  view={15}{25},
  axis lines=center,
  width=10cm,height=10cm,
  xtick={0,2,5},ytick={0,2,5},ztick={0,2,5},
  xmin=-1,xmax=7,ymin=-1,ymax=6.5,zmin=-1,zmax=7,
  xlabel={$x_1$},ylabel={$x_2$},zlabel={$x_3$},
]


\addplot3 [no marks, densely dashed] coordinates {(2,2,2) (5,2,2) (5,5,2) (2,5,2) (2,2,2) (2,2,5) (5,2,5) (5,5,5) (2,5,5) (2,2,5)};
\addplot3 [no marks, densely dashed] coordinates {(2,5,2) (2,5,5)};
\addplot3 [no marks, densely dashed] coordinates {(5,5,2) (5,5,5)};
\addplot3 [no marks, densely dashed] coordinates {(5,2,2) (5,2,5)};
\addplot3 [no marks, line width=0.1cm, color=blue] coordinates {(2,2,2) (5,2,2) (5,5,2) (5,5,5)};

\node [below] at (axis cs:2,2,2) {$(2,2,2)$};
\node [below] at (axis cs:5,2,2) {$(5,2,2)$};
\node [right] at (axis cs:5,5,2) {$(5,5,2)$};
\node [right] at (axis cs:5,5,5) {$(5,5,5)$};
\end{axis}
\end{tikzpicture}
\caption{The blue (thicker) segments indicate $\{u^{S(x)}\mid x\in[2,5]^3\}$, a set of points in $\Delta^3$ that are not majorized by other points in $\mathbb{R}^3$ (up to permutation).}
\label{fig:setofmajpoints}
\end{figure}
A similar result can be shown for upper level sets of quasiconcave functions over general polytopes \cite{tr17}. Symmetric quasiconcave functions are a subclass of Schur-concave functions. In other words, both the results show that for symmetric quasiconcave functions over permutation-invariant polytopes it suffices to consider the edges of the polytope to construct the convex hull. 
However, the result in \cite{tr17} applies to general quasiconcave functions over arbitrary polytopes while Proposition~\ref{prop:schurhypercube} applies to Schur-concave functions over a hypercube.
Perhaps more importantly, the result in Proposition~\ref{prop:schurhypercube} also applies to level sets of the functions while the result in \cite{tr17} only applies to convex envelope construction.

\def\F{{\cal{F}}}

Permutation-invariance also helps with constructing compact extended formulations of certain nonlinear sets. 
To motivate this statement and introduce the following result, consider the set $Z' = \bigl\{(x,z) \in [a,b]^n \times \Re \mid z=\prod_{i=1}^n x_i\bigr\}$.
It is well-known that, in order to generate $\conv(Z')$, it suffices to restrict $x$ to the vertices, $\F'=\{a,b\}^n$, of the hypercube $[a,b]^n$.
More precisely, $\conv(Z')=\conv\left(\bigcup_{x \in \F'} (x,\prod_{i=1}^n x_i)\right)$, where each disjunct in the union is a polytope with compact description (a single point.) 
A higher-dimensional description of the convex hull of this union can be obtained using classical disjunctive programming results.
Because the dimension of this formulation depends linearly on the number of disjuncts, $|\F'|$, such an approach has found limited practical use for the given example, as $|\F'|$ is exponential in $n$.
Surprisingly, taking advantage of the permutation-invariance of $Z'$ through Theorem~\ref{res:main_thm} allows for a much more economical use of disjunctive programming. 
Intuitively, this is because the number of elements of $\F'$ required to compute $\conv(S_0)$ in Theorem~\ref{res:main_thm}, is polynomial in $n$.
As a result, disjunctive programming provides a polynomial formulation for $\conv(S_0)$ which can then be integrated with the result of Theorem~\ref{res:main_thm} to obtain a polynomially-sized higher-dimensional formulation of $\conv(Z')$. 
The settings where such polynomially-sized formulations can be obtained extend far beyond the example presented above as we describe next in Theorem~\ref{res:limitfaceshypercube}.

Consider now the much more general setup of sets $S(Z,a,b):=\{(x,z)\in [a,b]^n\times\R^m\mid (x,z)\in Z\}$,
where $Z$ is compact and permutation-invariant in $x$. 
Further, for $\F=\{F_1,\ldots,F_r\}$, where $F_i$ are faces of $[a,b]^n$, define 
$X(Z,a,b,\F):=\left\{(x,z)\in [a,b]^n\times\R^m \mid (x,z)\in Z, x\in \bigcup_{i=1}^r F_i\right\}$. 
Using Theorem~\ref{res:main_thm}, we show next that a polynomial-size extended formulation can be constructed for any set $S(Z,a,b)$ for which there exists a collection of faces $\F'$ that completely determines the convex hull, \textit{i.e.},  $\conv\bigl(S(Z,a,b)\bigr) = \conv\bigl(X(Z,a,b,\F')\bigr)$ and for which a polynomial (possibly extended) formulation of the set on each of these faces $F_i \in \F'$, \textit{i.e.}, $\conv\bigl(X(Z,a,b,\{F_i\})\bigr)$, can be obtained. 
The strength of this result is that we make no assumption on $|\F'|$ and that this collection may have exponentially many faces.

\begin{theorem}\label{res:limitfaceshypercube}
Let $a,b\in \R$, $Z\subseteq \{(x,z)\mid x\in \R^n, z\in\R^m\}$ be a compact permutation-invariant set with respect to $x$,
and $\F=\{F_1,\ldots,F_r\}$ be a collection of faces of $[a,b]^n$ such that $\conv\bigl(S(Z,a,b)\bigr) = \conv\bigl(X(Z,a,b,\F)\bigr)$. 
Moreover, assume that $\conv(X(Z,a,b,\{F_i\}))$ has a polynomial-sized compact extended formulation for each $F_i \in \F$.
Then, $\conv\bigl(S(Z,a,b)\bigr)$ has a polynomial-sized extended formulation.
\end{theorem}
\begin{proof}{Proof.}
For brevity of notation, in this proof, we shall write $S(Z,a,b)$ as $S$ and $X(Z,a,b,\F)$ as $X(\F)$. 
We construct $\conv(S)$ using its equivalence to $\conv\bigl(X(\F)\bigr)$. 
We may assume for computing $\conv\bigl(X(\F)\bigr)$, by taking the union of all permutations of $X(\F)$ with respect to $x$ if necessary, that $X(\F)$ is permutation-invariant in $x$. 
This is because a permutation of $X(\F)$ with respect to $x$, say $X_\pi(\F):=\bigl\{(x,z)\mid (\pi(x),z)\in X(\F)\bigr\}$, is contained in $\conv\bigl(X(\F)\bigr)$ as is seen from $X_\pi(\F)\subseteq S\subseteq\conv\bigl(S\bigr) = \conv\bigl(X(\F)\bigr)$, where the first inclusion is by permutation-invariance of $S$ and the equality is by the assumed hypothesis. 
Since $S$ is permutation-invariant with respect to $x$, by Lemma~\ref{res:lin_invariant}, $\conv(S)$ is also permutation-invariant. 
We shall use Theorem~\ref{res:main_thm} to construct $\conv\bigl(X(\F)\bigr)$. 
We first show that we can limit the faces of $[a,b]^n$ that need to be considered in the construction of $S_0$; see Theorem~\ref{res:main_thm}.
Consider an arbitrary face $F_i$ of $[a,b]^n$, which is determined by setting a set of variables with indices in $B_i\subseteq\{1,\ldots,n\}$ to their upper bound $b$ and a disjoint set of variables $A_i\subseteq\{1,\ldots,n\}$ to their lower bound $a$. 
We will show that the only faces, $F_i$, $i=1,\ldots,r$ that need to be considered are such that $B_i$ and $A_i$ are \emph{hole-free}, {\it i.e.}\/, $B_i$ is of the form $\{1,\ldots,p\}$ and $A_i$ is of the form $\{q,\ldots,n\}$. 
To see this, let $j(i) = \max\{j\mid j\in B_i\}$ and $X_{i} = X(\{F_{i}\}) \cap \{(x,z)\mid x_1\ge \cdots\ge x_n\}$.
It follows from Theorem~\ref{res:main_thm} and permutation-invariance of $X(\F)$ that it suffices to consider points in $\bigcup_{i=1}^r X_i$ to construct $\conv(X(\F))$. We will further argue that it suffices to restrict attention to points in $\bigcup_{i:j(i) = |B_i|} X_i$.
Notice that $B_i$ is hole-free if and only if $j(i)=|B_i|$.
Assume by contradiction that $i'$ is the index of a face such that $j(i') > |B_{i'}|$ and that $X_{i'}$ contains a point which is not in $\bigcup_{i: j(i) = |B_{i}|} X_i$.
Among all such faces, we choose $i'$ to minimize $j(i) - |B_{i}|$. 
Since $B_{i'}$ is not hole-free, there exists $j\notin B_{i'}$ such that $j < j(i')$. Any point that belongs to $X_{i'}$ must satisfy $b \ge x_j \ge x_{j(i')} = b$. 
Therefore, $x_{j} = b$. Since $X_{i'}\ne\emptyset$, $j\not\in A_i$. Consider now $i''$ such that $B_{i''} = B_i\cup\{j\}\backslash \{j(i')\}$ and $A_{i''} = A_{i'}$. 
Such a face exists in $\F$ since we assumed that for every face $F_i \in \F$, $\F$ contains all faces obtained by permuting the variables, and $F_{i''}$ is obtained from $F_{i'}$ by exchanging the variables $x_j$ and $x_{j(i')}$. 
We next show that $X_{i''}\supseteq X_{i'}$. For arbitrary $(x,z)\in X_{i'}$, $x_i=b$ for all $i\in B_{i'}$. Then, it follows that $x_i=b$ for all $i\le j(i')$ as $x_1\ge\dots\ge x_n$. Therefore, $x_i=b$ for all $i\in B_{i''}$. Since $A_{i'}=A_{i''}$ and $x_j$ is not restricted for other indices $j$, $(x,z)\in X_{i''}$. Therefore, $X_{i''}$ contains a point not in $\bigcup_{i: j(i) = |B_{i}|} X_i$, establishing that $j(i'') > |B_{i''}|$.
However, since $j(i'') - |B_{i''}| < j(i') - |B_{i'}|$, this contradicts our choice of $i'$. A similar argument can be used to show that we only need to consider faces $F_i$ such that $A_i$ is hole-free.

There are at most $\binom{n+2}{2}$ such faces, one for each choice of $(p,q)$, where $0\le p\le q-1 \le n$. 
Since each $X(\{F_i\})$ is assumed to have a polynomial-sized compact extended formulation, it follows, by disjunctive programming \citep[Proposition 2.3.5 in][]{ben2001}, that $\conv(S_0)$ has a polynomial-sized compact extended formulation. 
The result then follows directly from Theorem~\ref{res:main_thm}.
\hfill\Halmos
\end{proof}

We record and summarize the extended formulation of $\conv\bigl(S(Z,a,b)\bigr)$ for later use in Corollary~\ref{res:polyfacerepresent}. 
We also observe that our construction applies even when $Z$ is not compact, as long as the convex hull of $X(Z,a,b,\{F_i\})$ for each $F_i$ of interest is available.
For a face $F$ of a polyhedron described by $a^\intercal x\le b$, we refer to another face $\bar{F}$ as a \emph{permutation} of $F$ if $\bar{F}$ is described by $\bar{a}^\intercal \le b$ where $\bar{a}$ is a permutation of $a$.
We say that a collection of faces $\mathcal{F}$ of a polyhedron is \emph{permutation-invariant}, if for a face in $\mathcal{F}$ described by an inequality, all its permutations are included in $\mathcal{F}$.
For a face $F$, we define $l(F) = \{ j \in \{1,\ldots,n\} \,|\, \hat{x}_j=a, \forall \hat{x} \in F \}$
and $u(F) = \{ j \in \{1,\ldots,n\} \,|\, \hat{x}_j=b, \forall \hat{x} \in F \}$.

\begin{corollary}\label{res:polyfacerepresent}
Let $a,b\in \R$, $Z\subseteq \{(x,z)\mid x\in \R^n, z\in\R^m\}$ be a permutation-invariant set with respect to $x$, and $\F=\{F_1,\ldots,F_r\}$ be a permutation-invariant collection of faces of $[a,b]^n$.
Let $I = \bigl\{i\in\{1,\dots,r\}\mid \exists p, q, p<q \text{ s.t. } u(F_i) = \{1,\ldots,p\} \textrm{ and } l(F_i) = \{ q,\ldots,n\} \bigr\}$. Then,

\begin{equation*}
\conv\bigl(X(Z,a,b,\F)\bigr) = \Bigl\{(x,z)\Bigm| (u,z) \in \conv\Bigl(\bigcup_{i\in I} X\bigl(Z,a,b,\{F_i\}\bigr)\Bigr), u_1\ge \cdots\ge u_n, u\ge_m x\Bigr\}.
			\tag*{\Halmos}
\end{equation*}
\end{corollary}

We remark that $\conv\bigl(\bigcup_{i\in I} X(Z,a,b,\{F_i\})\bigr)$ can be constructed using disjunctive programming techniques if the recession cone of  $X(Z,a,b,\{F_i\})$ does not depend on $i$ \cite[Corollary 9.8.1]{rf70}. Theorem~\ref{res:limitfaceshypercube} shows that even though the number of faces in $\F$ may be exponentially large, we can exploit the permutation-invariance of the set to consider only a polynomial number of faces in the construction. 
More explicitly, there are $2^{n-d}\binom{n}{d}$ $d$-dimensional faces of $[a,b]^n$ and $\binom{n+1}{d+1}$ $d$-dimensional faces of the simplex $b\ge x_1\ge\cdots\ge x_n\ge a$. 
But, there are only $n-d+1$
of the $d$-dimensional faces of the hypercube, namely, $F_l$ for $l\in \{0,\ldots,n-d\}$, that are defined by the hole-free sets $B_l = \{1,\ldots,l\}$ and $A_l=\{l+d+1,\ldots,n\}$ with the convention that $B_0=A_{n-d}=\emptyset$.

Applications of Theorem~\ref{res:limitfaceshypercube} extend beyond Schur-concave functions. 
For example, consider the convex hull of $\{(x,\alpha)\in [a,b]^n\times\R\mid \prod_{i=1}^nx_i \le \alpha\}$, where $a$ is not necessarily positive. 
The product function is not Schur-concave when some of the variables take negative values; for example consider $x_1x_2x_3$ and observe that although $(1,-1,-3)\ge_m (0,0,-3)$, the function value is higher at $(1,-1,-3)$ than at $(0,0,-3)$.

There are many functions besides the multilinear monomial that are permutation-invariant and whose envelopes are determined by their values on certain low-dimensional faces -- the postulates required for the construction in Theorem~\ref{res:limitfaceshypercube}. We discuss some examples next. Observe that all elementary symmetric polynomials satisfy these postulates because, being multilinear, their convex envelope is generated by vertices of $[a,b]^n$. Since Theorem~\ref{res:limitfaceshypercube} allows $z$ to be multi-dimensional, this result in fact yields the simultaneous convexification of the elementary symmetric monomials over $[a,b]^n$ and thus the convex envelope of positive linear combinations of elementary symmetric polynomials. 
Next, consider the function $x_1^t\cdots x_n^t$ over $[a,b]^n$, where $a,t > 0$. If $t \le 1$, the function is concave when all but one variable is fixed and therefore the convex envelope is generated by vertices of $[a,b]^n$. If $t\ge 1$, then by restricting attention to any face of the hypercube, where two or more variables are not fixed, say $x_1$ and $x_2$, we see that the determinant of the Hessian of $x_1^tx_2^t$ is $-x_1^{2t-2}x_2^{2t-2}t^2(2t-1)$, which is strictly negative. Therefore, there is a direction in which the function is concave and the convex envelope is determined by $1$-dimensional faces of the hypercube. Since the function is convex on these faces, the convex envelope can be developed using Corollary~\ref{res:polyfacerepresent}. In fact, when $t \ge 1$,  Proposition~\ref{prop:schurhypercube} provides a closed-form expression for the convex envelope. This is because the function $f(x) = x_1^t\cdots x_n^t$ is Schur-concave because $(x_i-x_j)\left(\frac{\partial f}{\partial x_j} - \frac{\partial  f}{\partial x_i}\right) = t(x_i-x_j)^2\frac{x_1^t\cdots x_n^t}{x_ix_j} > 0$.

We next expand on the characterization we gave in Corollary~\ref{res:polyfacerepresent} to obtain a more streamlined description of the result for the particular cases of permutation-invariant functions whose convex envelopes are completely determined by the extreme points of their domain, which is assumed to be a hypercube.



\begin{proposition}\label{prop_env_nonlin1}
Consider a function $\phi(x):[a,b]^n\mapsto\R$, that is permutation-invariant in $x$ and whose convex envelope remains the same even if its domain is restricted to $\{a,b\}^n$. For $i=1,\ldots,n$ and $j=0,\ldots,n$, let $p_{ij} = a$ if $i > j$ and $b$ otherwise and let $p_{\cdot j}$ denote the $j^{\text{th}}$ column of this matrix. Define $f(x) := \phi(p_{\cdot 0}) + \sum_{i=1}^n \frac{x_i-a}{b-a} \bigl(\phi(p_{\cdot i})-\phi(p_{\cdot i-1})\bigr)$. Then, the convex envelope of $\phi(x)$ over $[a,b]^n$ can be expressed as the value function of the following problem:
\begin{equation}
{\conv}_{[a,b]^n}(\phi) (x) = \min \bigl\{f(u)\,\bigm|\, u\ge_m x, b\ge u_1\ge\cdots\ge u_n\ge a\}.
\label{nonlinearabconvhull}
\end{equation}
\end{proposition}
\proof{Proof.}
Observe that the points in $\{a,b\}^n$ that intersect with $x_1\ge \cdots \ge x_n$ are precisely the columns $p_{\cdot j}$ described in the statement of the result. 
Consider the column $p_{\cdot j}$ and observe that $f\bigl(p_{\cdot j}\bigr) = \phi(p_{\cdot j})$. 
Moreover, $f$ is affine. Let $\Gamma = \{x\in \{a,b\}^n \mid x_1\ge \cdots \ge x_n\}$. 
Then, we show that $f = \conv_{\conv(\Gamma)}(\phi|_{\Gamma})$, where $\phi|_\Gamma$ denotes the restriction of $\phi$ to $\Gamma$. Clearly, $f(x)\le \conv_{\conv(\Gamma)}(\phi|_\Gamma)(x)$ because it matches $\phi|_\Gamma$ at all the points in the domain and is a convex underestimator. 
Also, $f(x) \ge \conv_{\conv(\Gamma)}(\phi|_\Gamma)(x)$ because of Jensen's inequality applied to $\conv_{\conv(\Gamma)}(\phi)$, observing that $f(x)$ is exact at the extreme points of $\Gamma$, and because $f(x)$ is affine. 
Now, consider Corollary~\ref{res:polyfacerepresent}. Let $\F$ be the set of extreme points of $[a,b]^n$ and $Z= \{(x,z)\mid z\ge \phi(x)\}$.
By assumption, $\conv\bigl(S(Z,a,b)\bigr) = \conv\bigl(X(Z,a,b,\F)\bigr)$.  
By Corollary~\ref{res:polyfacerepresent}, it follows that $\conv(X(Z,a,b,\F))=\left\{(x,z)\,\middle|\, (u,z)\in\conv\Bigl(\bigcup_{j=0}^n X(Z,a,b,\{p_{\cdot j}\})\Bigr), u\ge_m x\right\}
= \{(x,z)\mid z\ge f(u), b\ge u_1\ge \cdots \ge u_n\ge a, u\ge_m x\}$. 
\hfill\Halmos
\endproof


In the next result, we show how the convex hull of $\bigl\{(x,y)\in [a,b]^n \times [c,d]^m\bigm| \prod_{i=1}^m y_i^\alpha = \prod_{j=1}^n x_j^\beta\bigr\}$ can be constructed. To do so, it suffices to construct the convex hull of $S=\left\{(x,y)\in [a,b]^n\times [c,d]^m\;\middle|\; \prod_{i=1}^m y_i^\alpha \ge \prod_{j=1}^n x_j^\beta\right\}$, since the convex hull for the reverse inequality can be developed by switching the $x$ and $y$ variables, and the convex hull for the equality can then be obtained as the intersection of these two convex hulls; see \cite{nguyen2018deriving}. Such a set occurs in polynomial optimization problems where linearized variables are introduced to relax $uv=y$, $u^2=x_1$, and $v^2=x_2$ in the form of $y^2=x_1x_2$. Similar higher dimensional equality constraints also occur and their relaxations can be employed in polynomial optimization problems.

\begin{proposition}\label{prop_env_nonlin2}
 Consider 
$S=\left\{(x,y)\in H\;\middle|\; \prod_{i=1}^m y_i^\alpha \ge \prod_{j=1}^n x_j^\beta\right\}$,
where $H=[a,b]^n\times [c,d]^m$, with $a\ge 0$, $c\ge 0$, $\alpha > 0$, and $\beta > 0$. 
Let $k = \min\{m, \bigl\lfloor\frac{\beta}{\alpha}\bigr\rfloor\}$. Define the convex sets: 
\[
\begin{array}{ll}
S_{ij} &= S\cap \left\{(x,y)\in H\,\middle|\,
\begin{array}{ll}(y_r)_{r=1}^i=d, (y_r)_{r=i+k+1}^m=c,
&y\in\Delta^m,\\
(x_s)_{s=1}^j=b, (x_s)_{s=j+2}^n=a, 
\end{array}
\right\}\\
C_{j} &= 
S \cap \left\{
(x,y)\in H
\,\middle|\, 
(x_s)_{s=1}^j= b, (x_s)_{s=j+1}^n=a, y\in\Delta^m
\right\}
\end{array}
\]
where $S_{ij}$ is defined for $i=0,\ldots,m-k$ and $j=0,\ldots,n-1$ and $C_j$ is defined for $j=0,\ldots,n$. Let $T = \bigcup_{i,j} S_{ij} \cup \bigcup_j C_j$. Then
\[
\conv(S) = X :=\{(x,y)\mid v\ge_m y, u\ge_m x, (u,v)\in \conv(T)\}.
\]
In particular, if $m\alpha \le \beta$, then 
\begin{equation}\label{monomialhullnolift}
\conv(S) = X' := \left\{(x,y)\in H\;\middle|\; \prod_{i=1}^m y_i^\frac{1}{m} \ge \prod_{j=1}^n {u(x)_j}^\frac{\beta}{m\alpha}\right\},
\end{equation}
where $u(x):=u^{S(x)}$, defined as in Proposition~\ref{prop:schurhypercube}.
\end{proposition}

Before providing the proof, we discuss its architecture. This proof will write the convex hull of $S$ as a disjunctive hull of convex subsets of $S$. Since these convex subsets will be obtained by fixing variables at their bounds, even though they are exponentially many, Corollary~\ref{res:polyfacerepresent} will allow the construction of the convex hull. In the first part of the proof, we consider a slice of $S$ at a fixed $y$ and show that it suffices to restrict $x$ to one-dimensional faces of $[a,b]^n$ for constructing the convex hull of the slice.  Then, we show that there are two cases. If the remaining $x_j$ variable is also fixed to the bounds, the set is already convex. If not, then we show that any point in such a set cannot be extremal in $S$ unless at least a certain number of $y$ variables, specifically $\max\bigl\{0, m-\lceil\frac{\beta}{\alpha}\rceil\bigr\}$ of them, are fixed to their bounds. The sets obtained by fixing these variables are once again convex. Then, it will follow by Theorem~\ref{res:limitfaceshypercube}, that we may restrict our attention to the faces in $T$ and, as a result, $\conv(S) = X$.

\proof{Proof.}

Let $\phi(x) := \prod_{j=1}^n x_j^\beta$ and consider the set $\Upsilon(\gamma) = \{x \in [a,b]^n\mid \phi(x) \le \gamma\}$. By Theorem~ 3.A.3 in \cite{marshall2010}, $\phi(x)$ is Schur-concave over $[a,b]^n$ because it is permutation-invariant and $\frac{\partial \phi}{\partial x_1} \le\dots\le \frac{\partial \phi}{\partial x_n}$ at any point with $x_1\ge\cdots\ge x_n$. 
Let $\Upsilon_i(\gamma) = \bigl\{x\in \{b\}^{i-1}\times[a,b]\times\{a\}^{n-i}\,\bigm|\, x_i^\beta b^{(i-1)\beta}a^{(n-i)\beta} \le \gamma\bigr\}$. 
Then, it follows by Proposition~\ref{prop:schurhypercube} and Corollary~\ref{res:polyfacerepresent} that $\conv(\Upsilon(\gamma)) = \bigl\{x\,\bigm|\, u\ge_m x, u_1\ge\cdots\ge u_n, u\in \conv\bigl(\bigcup_{i=1}^n \Upsilon_i(\gamma)\bigr)\bigr\}$. 

\def\y{\bar{y}}

Since $n-1$ of the $x_j$ variables can be fixed to bounds in the construction of the convex hull of each slice, this restriction can also be imposed in the construction of $\conv(S)$. 
Now, let $\psi(y) := \prod_{i=1}^m y_i$ and consider the slightly more general set $\Theta$ which will appear when we fix some of the $y$ variables at their bounds. This set is defined as $\Theta = \bigl\{(x,y)\in [a,b]\times [c,d]^m \,\bigm|\, \zeta \psi(y) \ge \delta x^\frac{\beta}{\alpha}\bigr\}$, where $\delta,\zeta\ge 0$. 
Consider a point $(x',y')\in \Theta$. 
Then, by restricting attention to $\y=\lambda y'$, where $\lambda \in \Re$, we obtain an affine transform of a subset 
$\Lambda$ of $\Theta$ such that $\Lambda = \bigl\{(x,\lambda)\in [a,b]\times [c',d']\,\bigm|\, \lambda\theta \ge \delta' x^\frac{\beta}{m\alpha} \bigr\}$, where $\theta = \zeta^{\frac{1}{m}}\psi(y')^{\frac{1}{m}}$, $\delta' = \delta^{\frac{1}{m}}$, $c' = \max\{\lambda\mid \lambda y'_i \le c \mbox{ for some } i\}$, and $d'=\min\{\lambda \mid \lambda y'_i \ge d \mbox{ for some } i\}$. 

Assume $x'\in (a,b)$ and $y'\in (c,d)^m$. We will first show that, if $m > \frac{\beta}{\alpha}$, such a point is not an extreme point of $S$.
By definition, $c' < 1$, $d' > 1$, and $(x',1)\in \Lambda$. 
Assume $m > \frac{\beta}{\alpha}$ and $(x',y')$ is an extreme point of $S$. Define $s = \delta'\frac{\beta}{m\alpha} (x')^{\frac{\beta}{m\alpha} - 1}$. 
If $x'\in (a,b)$, then, for sufficiently small $\epsilon > 0$, 
we show that $(x',1)$ can be written as a convex combination of $(x'-\epsilon\theta, 1-s\epsilon)$ and $(x'+\epsilon\theta, 1+s\epsilon)$. 
The latter points are feasible in $\Lambda$ because
\begin{equation*}
				\delta'(x' \pm \epsilon\theta)^{\frac{\beta}{m\alpha}}\le \delta'x'^{\frac{\beta}{m\alpha}} + s(x'\pm \epsilon\theta-x')\le \theta(1\pm  s\epsilon),
\end{equation*}
where the first inequality is by concavity of $x^{\frac{\beta}{m\alpha}}$ for $m \ge \frac{\beta}{\alpha}$ and the second inequality is because $\delta' x'^{\frac{\beta}{m\alpha}} \le \theta$ as $(x',1)$ belongs to $\Lambda$. 
Since $\Lambda$ is an affine transform of a subset of $\Theta$, we have expressed $(x',y')$ as a convex combination of $\bigl(x'-\epsilon\theta,(1-s\epsilon)y'\bigr)$ and $\bigl(x'+\epsilon\theta,(1+s\epsilon)y'\bigr)$, each of which is feasible to $\Theta$. 
Since $\epsilon > 0$ and $x' > a\ge 0$ implies $s > 0$, it follows that these points are distinct thus contradicting the extremality of $(x',y')$.

Since, $S$ is compact, in order to construct $\conv(S)$, we may restrict our attention to the extreme points of $S$. Therefore, either $x'\in \{a,b\}$ or there exists an $i$ such that $y'_i \in \{c,d\}$.
If $x'\in \{a,b\}$, the point belongs to the convex subset of $\Theta$ obtained by fixing $x'$ at its current value because the defining inequality can be written as $\zeta^{\frac{1}{m}}\psi(y)^{\frac{1}{m}} \ge \delta x'^{\frac{\beta}{\alpha}}$, a convex inequality. 
A permutation of such an $(x',y')$ is included in one of the $C_j$s.
On the other hand, if $y'_i\in \{c,d\}$, we can reduce the dimension of the set by fixing $y_i$ at $y'_i$ and, thus, effectively reduce $m$. 
Therefore, we may assume without loss of generality that $m \le \frac{\beta}{\alpha}$. 
Then, we rewrite the defining inequality of $\Theta$ as $\zeta^{\frac{1}{m}}\psi(y)^{\frac{1}{m}} \ge \delta^\frac{1}{m} x^{\frac{\beta}{m\alpha}}$ and observe that this is a convex inequality since $\psi(y)^{\frac{1}{m}}$ is a concave function and $x^{\frac{\beta}{m\alpha}}$ is a convex function. 
Therefore, we need to consider faces where either all $x_j$ are fixed at their bounds or where we fix all $y_i$ except for a subset of cardinality $\min\bigl\{m,\bigl\lfloor\frac{\beta}{\alpha}\bigr\rfloor\bigr\}$ and fix all $x_j$ except for one; all such sets are, upto permutation, included in one of the $C_j$s or $S_{ij}$s.  

Since $S_{ij}\subseteq S$, $C_j\subseteq S$, and $S$ is permutation-invariant, it follows that $\conv(S)\supseteq X$. Since $X$ is convex and $S$ is compact, we only need to show that the extreme points of $S$ are contained in $X$. However, we have shown that the extreme points of $S$ belong to $T$ or a set obtained from $T$ by permuting the $x$ and/or $y$ variables. Since $X$ is permutation invariant and contains $T$, it follows that every extreme point belongs to $X$. Therefore, $X=\conv(S)$.

Now, consider the case where $m\alpha \le \beta$. Clearly, $k=m$. If we fix $y$ at $\y$, it follows from the Schur-concavity of $\prod_{j=1}^n x_i^{\frac{\beta}{m\alpha}}$ over $[a,b]^n$, the convexity of $x_i^{\frac{\beta}{m\alpha}}$ over $[a,b]$, and Proposition~\ref{prop:schurhypercube} that the convex hull of this slice is defined by $\prod_{i=1}^m \y_i^{\frac{1}{m}} \ge \prod_{j=1}^n u(x)_j^{\frac{\beta}{m\alpha}}$. This shows that $X'\subseteq \conv(S)$. By Schur-concavity of $\prod_{j=1}^n x_j^{\frac{\beta}{m\alpha}}$, it follows that $\prod_{j=1}^n x_j^{\frac{\beta}{m\alpha}} \ge \prod_{j=1}^n u(x)_j^{\frac{\beta}{m\alpha}}$, or $S\subseteq X'$. This implies that $S\subseteq X'\subseteq \conv(S)$. To show that $X'=\conv(S)$, we only need to show that $X'$ is convex.
As in the proof of Proposition~\ref{prop:schurhypercube}, we let $\varphi(s) = \prod_{j=1}^n u(x)_j^{\frac{\beta}{m\alpha}}$, where $s=\sum_{i=1}^n (x_i-a)$, and rewrite the above inequality as $\varphi(s) - \prod_{i=1}^m y_i^{\frac{1}{m}}\le 0$. Since the left-hand-side is jointly convex in $(s,y)$ and $s$ is a linear function of $x$, this proves that $X'$ is convex in $(x,y)$. 
\hfill\Halmos
\endproof

So far in this section, we have given various results where we describe the convex hull of a set in an extended space by introducing variables $u$. 
We now discuss how inequalities in the original space can be obtained by solving a separation problem. 

Usually, given a set $X$ and an extended space representation of its convex hull, $C$, we separate a given point $\x$ from $X$ by  solving the problem $\inf_{(x,u)\in C} \|x - \x\|$. 
By duality, the optimal value matches $\max_{\|a\|_* \le 1} \bigl\{\langle \x, a\rangle - h(a)\bigr\}$, where $h(\cdot)$ is the support-function of $C$ and $\|\cdot\|_*$ is the dual norm. 
Then, if the optimal value, $z^*$ is strictly larger than zero and the optimal solution to the dual problem is $a^*$, we have $\langle \x, a^*\rangle - z^*\ge \langle x, a^*\rangle$ for all $x\in \proj_x C$ and this inequality separates $\x$ from $\proj_x C$. 

However, such an inequality is typically not facet-defining for $\conv(X)$ even when the latter set is polyhedral. 
We now discuss a separation procedure that often generates facet-defining inequalities. 
The structure of permutation-invariant sets and their extended space representation allow for this alternate approach. 
Assume we are interested in developing the convex envelope of a permutation-invariant function $\phi$, such as $\prod_{i=1}^n x_i$, over $[a,b]^n$. 
As in Theorem~\ref{res:main_thm}, the convex envelope of $\phi$ at $x$ is obtained by expressing $x$ as a convex combination of $u$ and its permutations, where $u\ge_m x$. Moreover, assume that the convex envelope at $u$ is obtained as a convex combination of the extreme points of the simplex $a\le x_1\le \cdots \le x_n\le b$ with convex multipliers $\gamma$. Since $x=Su$ for some doubly stochastic matrix $S$ and $u = V\gamma$, where the columns of $V$ correspond to vertices of the simplex, it follows that $x=SV\gamma$.
Therefore, we can find a representation of $x$ as a convex combination of vertices in $V$ and their permutations.

We implement the above procedure for multilinear sets over $[a,b]^n$ to evaluate its impact on the quality of the relaxation. 
For the purpose of illustration, we consider the special case of $\prod_{i=1}^n x_i$ over $[a,b]^n$.  
In this case, (\ref{nonlinearabconvhull}) reduces to
\begin{eqnarray}
&\min & a^n + \sum_{i=1}^n b^{i-1}a^{n-i}(u_i-a) \nonumber\\
&\mbox{s.t.}&u\ge_m x \label{LPFORPRODX}\\
&&b\ge u_1\ge\dots\ge u_n\ge a. \nonumber
\end{eqnarray}
Given $x\in\R^n$ in general position inside $[a,b]^n$, assume that the optimal solution to \eqref{LPFORPRODX} is $u$. 
Then, we express $x=Su$, where $S\in \mathcal{M}^{n,n}(\mathbb{R})$ is a doubly-stochastic matrix. 
Given $x$ and $u$, this can be done through the solution of a linear program, $\min\{0\mid x=Su, S\mathbbm{1}=\mathbbm{1}, \mathbbm{1}^\intercal S=\mathbbm{1}^\intercal, S\ge 0\}$.
Although $S$ can also be derived as a product of $T$-transforms, see proof of Lemma~2 in Section 2.19 of \cite{hardy1952inequalities}, we use the LP approach in our numerical experiments, given its simplicity of implementation.
Then, we express $S$ as a convex combination of permutation matrices. 
Such a representation exists due to Birkhoff Theorem. 
We obtain it using a straightforward algorithm, which we describe next. 
Observe first that the desired representation is such that all permutation matrices with non-zero convex multipliers have a support that is contained within the support of $S$. 
This implies that the bipartite graph, we describe next, has a perfect matching. 
The bipartite graph is constructed with nodes labeled $\{1,\ldots,n\}$ in each partition and edges that connect a node $i$ in the first partition to $j$ in the second partition if and only if $S_{ij} > 0$.
Given a perfect matching, we construct a permutation matrix $P$ so that $P_{ij} = 1$ if node $i$ in the first partition is matched to node $j$ in the second partition. 
Then, we associate $P$ with a convex multiplier $\pi$ which is chosen to be $\min_{ij}\{ S_{ij} | P_{ij}=1\}$. If $\pi=1$, we have a representation of $S$ as a convex combination of permutation matrices.  Otherwise, observe that $\frac{1}{1-\pi}(S-\pi P)$ is again a doubly-stochastic matrix with one less non-zero entry. 
Therefore, by recursively using the above approach we obtain $S$ as a convex combination of at most $n^2$ permutation matrices. 
Then, we permute $u$ according to these permutation matrices. 
For each such $u$, the convex envelope is given by the optimal function value of \eqref{LPFORPRODX}. 
Each permuted $u$ can be expressed as a convex combination of the corner points of the permuted simplex $\{b\ge u_1\ge \cdots \ge u_n \ge a\}$. We claim that an affine underestimator $\breve{\phi}$ of $\conv(\phi)$ that is such that $\breve{\phi}(x) = \conv(\phi)(x)$ must also satisfy $\breve{\phi}(v^i) = \conv(\phi)(v^i)$ for every vertex $v^i$ of $[a,b]^n$ that has a non-zero multiplier in the representation of $x$ computed above. If not, we have convex multipliers $\lambda_i$, where $x=\sum_{i}\lambda_i v^i$, such that 
\[\breve{\phi}(x) = \conv(\phi)(x) =  \sum_i\lambda_i \conv(\phi)(v^i) = \sum_i \lambda_i \phi(v^i) 
 > \sum_i \lambda_i \breve{\phi}(v^i) = 
\breve{\phi}(x),  
\]
which yields a contradiction. Here, the first equality is by the definition of $\breve{\phi}$, second equality is because our construction obtains $\conv(\phi)(x)$ as a convex combination of $\conv(\phi)(v^i)$ using multipliers $\lambda^i$, the third equality is because $v^i$ are extreme points, the first inequality is because there exists an $i$ such that $\phi(v^i) > \breve{\phi}(v^i)$, and the final equality is because $\breve{\phi}$ is affine. Therefore, it follows that $\breve{\phi}(v^i) = \phi(v^i)$ for all $i$. Since $x$ was assumed to be in general position, these equality constraints uniquely identify $\breve{\phi}$.

We conclude this section by presenting the results of a numerical experiment that suggests that the bounds obtained when building convex relaxations of $\psi_n(x)=\prod_{i=1}^n x_i$ over $[a,b]^n$ using the 
procedure described above are significantly stronger than those obtained using factorable relaxations.
To this end, we consider functions $\psi_n(x)$ where $n=10$ over two permutation-invariant hyper-rectangles.
The first one, $B_1=[2,4]^{10}$, is contained in the positive orthant, while the second, $B_2=[-2,3]^{10}$, contains $0$ in its interior. 
We generate nine sample points uniformly at random inside of $B_1$ and $B_2$. 
At each point, we compare the value $z_r$ of the relaxation obtained using a recursive application of McCormick's procedure with the value $z_e$ of the convex envelope, obtained using the results described in this section. 
We then compute the existing gap (Gap) and relative gap (\%Gap), using the formulas $z_e-z_r$ and $\frac{z_e-z_r}{z_e}$, respectively. 
Results are presented in Tables~\ref{TBLMC1} and ~\ref{TBLMC2} where it can be observed that the proposed approach leads to substantial improvements in bounds, especially when variables $x$ take both positive and negative values. 

\begin{table}[htbp]
\[
\begin{tabular}{nrrrrrn}\hline
Sample & $z_e$ & $z_r$ & Gap & \%Gap\\\hline
1    & 18943.5 & 7584.8 & 11358.6 & 60.0\%\\
2    & 52904.9 & 21933.9 & 30970.9 & 58.5\%\\
3    & 22754.2 & 8622.1 & 14132.1 & 62.1\%\\
4    & 26299.0 & 8526.7 & 17772.3 & 67.6\%\\
5    & 13817.1 & 5750.7 & 8066.3 & 58.4\%\\
6    & 25028.6 & 8906.2 & 16122.4 & 64.4\%\\
7    & 13852.4 & 5694.1 & 8158.3& 58.9\%\\
8    & 16059.4 & 8069.1 & 7990.2 & 49.8\%\\
9    & 10122.1 & 4812.2 & 5309.9 & 52.5\%\\\hline
Average & & & & 59.1\%\\\hline
\end{tabular}
\]
\caption{Gap at a randomly chosen point for $\prod_{i=1}^{10}x_i$ on $[2,4]^{10}$}\label{TBLMC1}
\end{table}

\begin{table}[htbp]
\[
\begin{tabular}{nrrrrrn}\hline
Sample & $z_e$ & $z_r$ & Gap & \%Gap\\\hline
1    & -12314.6 & -25655.4 & 13340.9 & 108.3\%\\
2    & -16221.2 & -29559.4 & 13338.2 & 82.2\%\\
3    & -13247.0   &-29405.9  & 16158.9 & 122.0\%\\
4    & -14069.4 & -28248.4 & 14179.0 & 100.8\%\\
5    & -10660.9 & -23134.2 & 12463.3 & 116.9\%\\
6    & -10979.5 & -21263.1 & 10283.7 & 93.7\%\\
7    & -9367.8 &  -21327.4 &  11959.6 & 127.7\%\\
8    & -10245.9 & -24782.6 & 14536.8 & 141.9\%\\
9    & -9182.8 & -21137.0 & 11954.2 & 130.2\%\\\hline
Average & & & & 113.7\%\\\hline
\end{tabular}
\]
\caption{Gap at a randomly chosen point for $\prod_{i=1}^{10}x_i$ on $[-2,3]^{10}$
}\label{TBLMC2}
\end{table}

\section{Set of rank-one matrices associated with permutation-invariant sets}
\label{sec:rankone}

For a positive integer $n$ and a given set $S\in\mathbb{R}^n$, define $M_S:=\left\{ (x,X)\in\mathbb{R}^n\times\mathcal{M}^n\mid X=xx^\intercal,  x\in S\right\}$ where $\mathcal{M}^n=\mathcal{M}^{n,n}(\mathbb{R})$.
For each element $(x,X)\in M_S$, it is clear that $\textup{rank}(X)=1$.
Studying $M_S$ is particularly important when constructing valid inequalities for semidefinite relaxations of non-convex optimization problems.
In this section, we study the case where the base set $S$ is permutation-invariant.

For an arbitrary positive integer $n$, we denote by $\mathbbm{1}_n$ the $n$-dimensional vector of ones. When $n$ is clear in the context, we simply denote it by $\mathbbm{1}$. Similarly, we denote by $\mathbbm{1}_{n\times n}$ the $n$-by-$n$ matrix of ones.

As a motivating example, consider sparse PCA, which, for a given covariance matrix $\Sigma$, finds a sparse vector $x$ that maximizes the variance $x^\intercal\Sigma x$.
A semidefinite relaxation of sparse PCA therefore aims to approximate
\begin{equation}
\label{spca}
M:=\{(x,X)\in\mathbb{R}^n\times \mathcal{M}^n\mid X=xx^\intercal,  \|x\|\le 1, \card(x)\le K\} 
\end{equation}
for a positive integer $K\in\{1,\dots,n-1\}$.
The set $M$ can be seen as $M_S$ by choosing $S$ to be the permutation-invariant set $S=\{x\in\mathbb{R}^n\mid \|x\|\le 1, \card(x)\le K\}$. 
The separation problem associated with $M$ is known to be NP-hard \cite{tillmann2013computational}.
Hence semidefinite relaxations have been considered that relax the non-convex constraint $X=xx^\intercal$ with $X\succeq xx^\intercal$, which is equivalent to the convex constraint 
$\begin{bmatrix}X&x\\
x^\intercal &1\end{bmatrix}
\succeq 0$.
Linear valid inequalities in $(x,X)$ are then developed by exploiting the property that $X=xx^\intercal$.
For example, the authors of \cite{d2007direct} introduce the valid inequality $\mathbbm{1}^\intercal X\mathbbm{1}\le K$ 
for \eqref{spca}, which is implied by valid inequality $\sum_{i=1}^n x_i\le \sqrt{K}$ and the condition $X=xx^\intercal$.

We next show that additional valid inequalities can be constructed in a higher dimensional space by using the permutation-invariance of the base set $S$.
To this end, we prove the following result.

\begin{proposition}
\label{prop:onlypermutation}
Suppose $S\subseteq\mathbb{R}^n$ is a permutation-invariant set. Let 
\[
\mathcal{N}= \left\{(x,u,X,U)\in\mathbb{R}^n\times\mathbb{R}^n\times\mathcal{M}^n\times\mathcal{M}^n
\middle\vert
\begin{array}{l}
X=xx^\intercal, U=uu^\intercal,\\ 
\textup{$x=Pu$ for some $P\in \mathcal{P}_n$}\\
u\in S\cap\Delta
\end{array}
\right\}.
\]
Then, $M_S=\textup{proj}_{(x,X)}\mathcal{N}$.
\end{proposition}
\begin{proof}{Proof.}
To prove $M_S \subseteq \textup{proj}_{x,X} \mathcal{N}$, consider
$(x,X)\in M_S$. Let $P\in \mathcal{P}_n$ be such that  $u:=P^{-1}x\in \Delta$ and let $U=uu^\intercal$. 
Then, $(x,u,X,U)\in\mathcal{N}$. 
To prove $M_S \supseteq \textup{proj}_{x,X} \mathcal{N}$,
consider  $(x,u,X,U)\in\mathcal{N}$
and let $P\in\mathcal{P}_n$ be such that $x=Pu$.
By permutation-invariance of $S$, $x\in S$, showing  $(x,X) \in M_S$.
\Halmos
\end{proof}

Now we develop linear inequalities implied by the conditions in $\mathcal{N}$.
For any $(x,u,X,U)\in\mathcal{N}$, observe that $
X = xx^\intercal = (Pu)(Pu)^\intercal = PUP^\intercal$
for a permutation matrix $P$.
Therefore, consider linear inequalities implied by the facts that $x$ is a permutation of $u$, and $X$ is obtained by permuting some columns and rows of $U$ symmetrically.
Perhaps, the most straightforward such inequalities are
\begin{subequations}
\label{basicvalid}
\begin{align}
\mathbbm{1}^\intercal u&=\mathbbm{1}^\intercal x\\
\trc(X)&=\trc(U),\label{basicvalid_a}\\
\mathbbm{1}^\intercal X\mathbbm{1}&=\mathbbm{1}^\intercal U\mathbbm{1},\label{basicvalid_b}
\end{align}
\end{subequations}

More generally, consider any  function $\phi:\mathbb{R}^n\times\mathcal{M}^n\rightarrow\mathbb{R}$ such that $\phi(x,xx^\intercal)$ is permutation-invariant with respect to $x$.
Then, we can impose the equality $\phi(x,X)=\phi(u,U)$ if $\phi$ is linear in $(x,X)$.
In fact, if $\phi$ is linear in $(x,X)$ then we argue that this identity is implied by \eqref{basicvalid}.
To see this, observe that $\phi(x,xx^\intercal)$ is a quadratic function in $x$.
Let $\psi(x)=\phi(x,xx^\intercal)=x^\intercal C x+d^\intercal x+e$ for $C\in\mathbb{S}^n, d\in\mathbb{R}^n$, and $e\in\mathbb{R}$.
By permutation-invariance of $\psi$, 
the function $\psi(x)-\psi(Px)$ is the zero function in $x$ for every $P\in\mathcal{P}_n$.
Observe that 
\[
\psi(x)-\psi(Px)=x^\intercal(C-P^\intercal CP)x+(d-P^\intercal d)^\intercal x \equiv 0.
\]
Therefore, $d=P^\intercal d$ and $C=P^\intercal CP$.
For $i\ne j\in\{1,\dots,n\}$, consider the permutation matrix $P$ such that $(Px)_i=x_j, (Px)_j=x_i$, and $(Px)_k=x_k$ for all $k\ne i, j$.
Then, $d_i=d_j, C_{ii}=C_{jj}$, and $C_{ij}=C_{ji}$.
Since the choices for $i$ and $j$ are arbitrary, it holds that $d=\rho\mathbbm{1}$ for some $\rho\in\R^n$, the diagonal entries of $C$ are identical, and $C$ is symmetric.
We next claim that all off-diagonal entries of $C$ are identical.
We assume $n\ge 3$ since it is clear otherwise.
For any $q\in\{1,\dots,n\}\setminus \{i,j\}$,
$[C-P^\intercal CP]_{iq}=C_{iq}-C_{jq}$.
That is, all entries of $q$th column of $C$ except for $C_{qq}$ are equal. By symmetry of $C$, all entries of $q$th row of $C$ except for $C_{qq}$ are equal.
Since $q$ is arbitrary, all off-diagonal entries of $C$ are identical because for any $i<j$ and $p<q$ with $q<j$, 
$C_{ij}=C_{pj}=C_{pq}$.
Therefore, 
\[\begin{array}{ll}
\psi(x)&=x^\intercal (c_1\diag(\mathbbm{1})+c_2\mathbbm{1}_{n\times n})x+(\rho\mathbbm{1})^\intercal x+e\\
&=c_1 \trc(xx^\intercal) + c_2 \mathbbm{1}^\intercal (xx^\intercal)\mathbbm{1}+\rho(\mathbbm{1}^\intercal x)+e.
\end{array}\]
Now, the desired equality $\phi(x,X)=\phi(u,U)$ is
\[
c_1 \trc(X) + c_2 \mathbbm{1}^\intercal X\mathbbm{1}+\rho(\mathbbm{1}^\intercal x)
=c_1 \trc(U) + c_2 \mathbbm{1}^\intercal U\mathbbm{1}+\rho(\mathbbm{1}^\intercal u),
\]
which is implied by equalities in \eqref{basicvalid}.

Another type of constraints can be obtained when $S\subseteq\mathbb{R}^n_+$.
That is, $u\in S$ is chosen to be nonnegative and in descending order.
Then, entries in each row of $uu^\intercal$ are also in descending order, yielding the inequalities.
\begin{equation}
\label{basicdescending}
U_{i,j}\ge U_{i,j+1},\quad  1\le i\le n,\:\: 1\le j< n-1.
\end{equation}
Similar arguments can be made for column entries.
These inequalities, however, are redundant because of the symmetry of $U$.

We next introduce a general framework that constructs tighter linear relaxations by exploring the conceptual relationship that $x$ is a permutation of $u$.
This allows us to model or relax identities of the form $\phi(x,xx^\intercal)=\phi(u,uu^\intercal)$ where $\phi$ is a certain real-valued nonlinear permutation-invariant function.
To this end, for fixed integers $p, q\in\{1,\dots,n\}$, $r\in\{1,\dots,\min\{pn,qn\}\}$, and $W\in\mathcal{M}^n$, consider the optimization problem
\begin{subequations}
\label{eq:transpqr}
\begin{align}
\max & \textstyle\sum_{i=1}^n\sum_{j=1}^n W_{ij}  t_{ij}  \nonumber\\
\st & \textstyle\sum_{j=1}^n t_{ij}\le q,&i\in\{1,\dots,n\}\label{eq:transpqrrow}\\
	& \textstyle\sum_{i=1}^n t_{ij}\le p,&j\in\{1,\dots,n\}\label{eq:transpqrcol}\\
    & \textstyle\sum_{i=1}^n\sum_{j=1}^n t_{ij}\le r\label{eq:transwhole}\\
    &0\le  t_{ij}\le 1,& i, j\in\{1,\dots,n\}.\label{eq:transpqrbounds}
\end{align}
\end{subequations}
Its dual is
\begin{subequations}
\label{eq:dualtransportationpqr}
\begin{align}
\min \quad	& \displaystyle q\sum_{i=1}^n \alpha_i + p\sum_{j=1}^n\beta_j + r\gamma +\sum_{i=1}^n\sum_{j=1}^n \delta_{ij}\nonumber\\
\st\quad	& \displaystyle \alpha_i+\beta_j+\gamma+\delta_{ij} =  W_{ij} + \theta_{ij} & i,j\in\{1,\dots,n\}\label{dualmainconstpqr}\\
			& \alpha_i\ge 0, \beta_j\ge 0, \gamma\ge 0, \delta_{ij}\ge 0, \theta_{ij} \ge 0& i,j\in\{1,\dots,n\},\label{dualnonnegativitypqr}
\end{align}
\end{subequations}
where dual variables $\alpha$, $\beta$, $\gamma$, and $\delta$ correspond to primal constraints \eqref{eq:transpqrrow}, \eqref{eq:transpqrcol},
\eqref{eq:transwhole}, and
the upper-bound constraints of \eqref{eq:transpqrbounds}, respectively.
We denote these optimization problems by $\max\{f^W(t)\mid t\in\Phi\}$ and $\min\{g(z)\mid z\in\Omega(W)\}$. 
Strong duality holds since both \eqref{eq:transpqr} and \eqref{eq:dualtransportationpqr} are feasible.
We denote $h(W):=\max\{f^W(t)\mid t\in\Phi\}=\min\{g(z)\mid z\in \Omega(W)\}$.
Observe that $h(ww^\intercal)$, as a function of $w$ is permutation-invariant with respect to $w$. 
Therefore, if $(x,u,X,U)\in\mathcal{N}$, it holds that $h(U)=h(X)$.
Since the linearity of the identity is not guaranteed, we construct linear inequalities in $(x,u,X,U)$ by taking the identity and the conditions in the set description of $\mathcal{N}$ into account. 
In the following discussion, we assume that $(x,u,X,U)\in\mathcal{N}$.

We first consider the case where a closed-form description of the optimal value $h(ww^\intercal)$ is known, $h(U)$ is linear in $U$, and $h(X)$ is not linear in $X$.
Since $h(X)$ and $f^X(t)$ are both nonlinear in $(X,W)$ and $(t,W)$, respectively, we use the dual objective formulation to reformulate the identity $h(U)=h(X)$ because the dual objective function and the constraints are linear in $(z, W)$.
We obtain a reformulation by replacing $h(X)$ with $g(z)$ and adding the conditions in the feasible set $\Omega(X)$ into the formulation.

We next consider the case where either a closed-form of $h(ww^\intercal)$ is unknown or $h(U)$ is nonlinear in $U$.
Then, we construct the relaxation by replacing both $h(U)$ and $h(X)$ with linear functions $g(z^U)$ and $g(z^X)$ with distinct variables $z^U=(\alpha^U, \beta^U,\gamma^U,\delta^U, \theta^U)$ and $z^X=(\alpha^X,\beta^X,\gamma^X,\delta^X, \theta^X)$ and add the conditions in $\Omega(U)$ and $\Omega(X)$.
Then, we tighten the relaxation by exploring the permutation relationship between $u$ and $x$ and the rank-one conditions.
Assume that $x=Pu$ and that $z^U=(\alpha^U, \beta^U,\gamma^U,\delta^U, \theta^U)$ is an optimal solution to the dual with $W=uu^\intercal$.
Then, $(P\alpha^U, P\beta^U,\gamma^U,P\delta^UP^\intercal, P\theta^U P^\intercal)$ is an optimal solution to the dual with $W=xx^\intercal$. 
Therefore, we can tighten the relaxation by considering conditions $\alpha^X=P\alpha^U, \beta^X=P\beta^X$,  $\gamma^U=\gamma^X$, and
$\delta^X=P\delta^UP^\intercal$ for some $P\in\mathcal{P}_n$.
For example, we can add the baseline linear conditions $\mathbbm{1}^\intercal \alpha^U=\mathbbm{1}^\intercal \alpha^X$, $\mathbbm{1}^\intercal \beta^U=\mathbbm{1}^\intercal \beta^X$, $\gamma^U=\gamma^X$, $\trc(\delta^U)=\trc(\delta^X)$, $\trc(\theta^U)=\trc(\theta^X)$, $\mathbbm{1}^\intercal \delta^U\mathbbm{1}=\mathbbm{1}^\intercal \delta^X\mathbbm{1}$, and $\mathbbm{1}^\intercal \theta^U\mathbbm{1}=\mathbbm{1}^\intercal \theta^X\mathbbm{1}$. If $\alpha^U\in\Delta$ (resp. $\beta^U\in\Delta$) then we can add the linear reformulation for $\alpha^U\ge_m \alpha^X$ (resp. $\beta^U\ge_m \beta^X$).
Furthermore, we can take advantage of a ``good" feasible solution of the dual.
Let $z_0^U=(\alpha_0^U,\beta_0^U,\gamma_0^U,\delta_0^U)$ be a feasible solution to the dual with $W=uu^\intercal$.
Then, we can replace the left-hand side with $g(z_0^U)$ and the equality with the inequality $\ge$.
We may add inequalities that capture the permutation relationships.
Obviously, we can add
$\gamma_0^U=\gamma^X$.
In addition, we can impose the linear reformulations of $\alpha_0^U\ge_m \alpha^X$ and $\beta_0^U\ge_m \beta^X$
because $\alpha_0^U$ and $\beta_0^U$ are constants.
In addition, any linear inequalities implied by the relationship $\delta_0^U=P^\intercal\delta^X P$ for some $P\in\mathcal{P}_n$, such as $\trc(\delta_0^U)=\trc(\delta^X)$ and $\mathbbm{1}^\intercal\delta_0^U\mathbbm{1}=\mathbbm{1}^\intercal\delta^X\mathbbm{1}$, can be considered. Similar relations can also be introduced for $\theta$. More generally, arbitrary linear functions that are permutation-invariant in $\alpha$, $\beta$, $\gamma$, $\delta$, and $\theta$ can be considered instead of the specific one in the objective of the dual.

We next present some special but important cases where the closed-form of $h(W)$ is known.

\begin{lemma}
\label{lemma:diagpqr}
When $W=ww^\intercal$ for $w\ge 0$ and $p=q=1$, the optimal value of \eqref{eq:transpqr} is $\sum_{i=1}^r w_{[i]}^2$.
\end{lemma}

\begin{proof}{Proof.}
Without loss of generality, we assume that $w\in\Delta$ and prove that the optimal value is $\sum_{i=1}^r w_{i}^2$.
Define $t'$ as $t'_{ij}=1$ if $i=j\le r$ and 0 otherwise.
Then, $t'$ is feasible for the primal.
Its objective value is $\sum_{i=1}^r w_i^2$.
We next define $z':=(\alpha', \beta', \gamma', \delta')\in\Re^n\times\Re^n\times\Re\times \mathcal{M}^n$ as $\alpha'_i=\beta'_i =\max\left\{\frac{w_i^2-w_r^2}{2}, 0\right\}$ for $i=1,\dots,n$, $\gamma'=w_r^2$, and $\delta'_{ij}=0$ for $i,j\in\{1,\dots,n\}$ and prove that $z'$ is feasible for the dual.
The nonnegativity of $z'$ is clear. 
We next show that $z'$ satisfies \eqref{dualmainconstpqr}.
First, suppose $i\le r$ and $j\le r$.
Then, $\alpha'_i+\beta'_j+\gamma'+\delta'_{ij}
=\frac{w_i^2+w_j^2}{2}\ge w_iw_j$.
Next, consider the case where $i\le r$ and $j> r$. 
Then, 
$\alpha'_i+\beta'_j+\gamma'+\delta'_{ij}
=\frac{w_i^2+w_r^2}{2}\ge w_iw_r\ge w_iw_j$,
where 
the last inequality holds because $w\in\Delta$.
The case where $i>r$ and $j\le r$ is symmetrical, and the case where $i>r$ and $j>r$ is clear.
Since $t'$ and $z'$ satisfy complementarity-slackness conditions, they are optimal solutions to the primal and the dual, respectively. 
Their common objective value is $\sum_{i=1}^r w_i^2$.
\hfill \Halmos
\end{proof}

By Lemma~\ref{lemma:diagpqr}, when $p=q=1$, $h(W)$ is the sum of $r$ largest diagonal entries of $W$.
While $h(W)$ is nonlinear, $h(U)$ is linear because it is the sum of the first $r$ diagonal entries.
On the other hand, the inequalities $h(U)\ge h(X)$ for $r\in \{1,\dots,n-1\}$ are equivalent to the inequality parts of the majorization $\diag(U)\ge_m \diag(X)$. The equality part of the majorization is equivalent to the existing constraint \eqref{basicvalid_a}.
Therefore, $\diag(U)\ge_m \diag(X)$ is a special case of the aforementioned modeling technique.

\begin{lemma}
\label{lemma:transportation}
When $W=ww^\intercal$ for $w\ge 0$ and $r=pq$,
the optimal value of \eqref{eq:transpqr} is
$\sum_{i=1}^p\sum_{j=1}^q w_{[i]}w_{[j]}$.
\end{lemma}
\begin{proof}{Proof.}
Without loss of generality, we assume that $w\in\Delta_+$.
First, define $t'$ as $t'_{ij}=1$ if $i\le p$ and $j\le q$ and 0 otherwise. 
It is clear that $t'$ is feasible and that its objective function value is $(\sum_{i=1}^p w_i)(\sum_{j=1}^q w_j)$.
Next, we consider its dual \eqref{eq:dualtransportationpqr}
and define $z':=(\alpha', \beta', \gamma', \delta')\in\Re^n\times\Re^n\times\Re\times \mathcal{M}^n$ as follows: 
\[
\begin{array}{rl}
\alpha'_i &=\max\{(w_i-w_p)w_q, 0\},\quad\: i=1,\dots,n\\
\beta'_j &=\max\{w_p(w_j-w_q), 0\}, \quad j=1,\dots,n\\
\gamma'&=w_pw_q\\
\delta'_{ij}&=\left\{
\begin{array}{ll}
(w_i-w_p)(w_j-w_q)& \quad \textup{if $i\le p$ and $j\le q$}\\
0 & \quad \textup{otherwise.}
\end{array}
\right.
\end{array}
\]
We first show that $z'$ is feasible for the dual.
The nonnegativity of the variables is clear from their definition. To prove that they satisfy \eqref{dualmainconstpqr}, we first consider the case where $i\le p$ and $j\le q$. 
Then, we compute that 
$\alpha'_i+\beta'_j+\gamma'+\delta'_{ij} = w_i w_j$, 
showing the result. 
Next, consider the case where $i>p$ or $j>q$. 
Without loss of generality, we assume that $i>p$. 
When $j>q$, it holds that $\alpha'_i+\beta'_j+\gamma'+\delta'_{ij}=\gamma'=w_pw_q\ge w_iw_j$ where the inequality holds because $i>p$, $j>q$, and $w\in\Delta_+$.
When $j\le q$, it holds that $\alpha'_i+\beta'_j+\gamma'+\delta'_{ij}=\beta'_j+\gamma'=w_pw_j\ge w_iw_j$ where the inequality holds because $i>p$ and $w\in\Delta_+$.
Since $t'$ and $z'$ satisfy complementarity-slackness conditions, they are optimal solutions to the primal and dual, respectively.
Their common objective value is $\sum_{i=1}^p\sum_{j=1}^q w_iw_j$.
\hfill \Halmos
\end{proof}

Lemma~\ref{lemma:transportation} and \eqref{eq:transpqr} can be trivially extended to the case where $W=\bar{w}w^\intercal$, where $\bar{w}\in \Re^m$, $w\in \Re^n$, $\bar{w},w\ge 0$, and $r=pq$, in which case the optimal value is $\sum_{i=1}^p\sum_{j=1}^q\bar{w}_{[i]}w_{[j]}$. 
However, this is not of direct relevance for our sparse PCA relaxations.
While $h(W)$ is nonlinear because the order of $w$ is unknown, $h(U)$ is linear because it is the sum of the entries of the $p$-by-$q$ upper-left submatrix of $U$.
In particular, for $q=n$ and $p\in\{1,\dots,n\}$, the inequalities $h(U)\ge h(X)$ are equivalent to the inequality parts of $R^U\ge_m R^X$ where $R^U$ and $R^X$ are the vectors of the row sums of $U$ and $X$, respectively. 
The equality part of the majorization is equivalent to \eqref{basicvalid_b}.
Therefore, $R^U\ge_m R^X$ is a special case of the aforementioned modeling technique.

In the remainder of the section, we introduce various strengthened semidefinite programming relaxations for sparse PCA.

\subsection{An SDP relaxation for sparse PCA}
\label{subsec:sdpspca}

Principal Component Analysis (PCA) is a well-known dimension reduction technique in statistical analysis.
A principal component is a linear combination of independent variables.
It also typically stands for the coefficient vector of the linear combination. 
The first principal component is a unit principal component for which variance is maximized; it is the eigenvector corresponding to the largest eigenvalue of the covariance matrix.
Even though the first principal component explains the most variance of the data, it is often hard to interpret because most of its coefficients are nonzero.
Sparse PCA is a variant of the approach introduced to resolve this issue by finding linear combinations with few explanatory variables.

Formally, let $\Sigma\in\mathcal{S}^n$ be the covariance matrix of the data set.
The following optimization problem, which we refer to as \emph{sparse PCA} and as defined in Section~\ref{sec:intro},
where $x\in\Re^n$ is the coefficient vector of the principal component,
finds a unit sparse vector with at most $K$ nonzero entries that explains most of the variance of the data:
\[\begin{array}{ll}
\textup{max}	& x\tr\Sigma x\\
\textup{s.t.}		& \|x\|\le 1,\\
				& \textup{card}(x)\le K,
\end{array}
\tag{Sparse PCA}
\]
where $K$ is a positive integer satisfying $1<K<n$.

We studied the feasible set of sparse PCA in Section~\ref{sec:sparsity} where we denoted it as $N^K_{\|\cdot\|}$ assuming $\|\cdot\|$ is the $\ell_2$-norm.
The feasible region of sparse PCA is non-convex because of the
sparsity constraint. 
We established in Section~\ref{sec:sparsity}  that 
\begin{equation}
\textup{conv}(N^K_{\|\cdot\|})
=\left\{x\middle\vert
\begin{array}{l}
\|u\|\le 1,\\
u_1\ge\dots\ge u_K\ge 0,\\
u_{K+1}=\dots=u_n=0,\\
u\ge_{wm} |x|
\end{array}
\right\}.
\label{convhull}
\end{equation}

Because sparse PCA maximizes a convex function, we can replace the feasible set with its convex hull, thereby obtaining a new problem formulation. 
This formulation, however, remains difficult to solve as it is a convex maximization problem. 
Next, we introduce new positive semidefinite relaxations for sparse PCA.
The most commonly used (and, to the best of our knowledge, only) SDP relaxation for sparse PCA was introduced in \cite{d2007direct} as follows
\begin{equation}
\label{eq:dassdp}
\begin{array}{rll}
\max	&\quad& \trc(\Sigma X)\\
\st 		&& \trc(X)\le 1,\\
		&& \mathbbm{1}^\intercal|X|\mathbbm{1} \le K,\\
		&& X\succeq 0.
\end{array}
\end{equation}
We refer to this as the \emph{$D$-relaxation}.

We next present a strengthened SDP relaxation based on the convex hull description \eqref{convhull}.
First, we introduce the matrix variable $X$ to model the relationship $X=xx^\intercal$. 
Then, we introduce variables $y$ and $Y$ to represent $|x|$ and $|X|$, respectively. 
Further, we add the auxiliary variables $v$, $w$, $V$, and $W$ to model the absolute values. 
The variables and the constraints are
\begin{equation}
\label{eq:absx}
\left\{
\begin{array}{l}
    x = v-w,\:\: y = v+w\\
    v,  w\in\Re^n_+,\:\: x, y\in\Re^n
\end{array}
\right.
\end{equation}
and
\begin{equation}
\label{eq:absX}
\left\{
\begin{array}{l}
    X = V-W,\:\:
    Y = V+W\\
    V, W\in\mathcal{M}^n_{\ge 0},\:\:
    X, Y\in\mathcal{S}^n,
\end{array}
\right.
\end{equation}
where $\mathcal{M}^n_{\ge 0}$ is the set of $n$-by-$n$ (entry-wise) nonnegative matrices and $\mathcal{S}^n$ is the set of $n$-by-$n$ symmetric matrices.
Next, we introduce the vector $u$ majorizing $y (=|x|)$ and the matrix $U$ to model $uu^\intercal$. 
The constraint $u\in \Delta_+$ and the constraint \eqref{basicdescending} in the cardinality setting can be written as
\begin{equation}
\label{eq:strfwdconst}
\left\{
\begin{array}{lll} 
    u_1\ge\dots\ge u_K\\
    u_{i}=0, && i\ge K+1\\
    U_{i,1}\ge\dots\ge U_{i,K}, && i=1,\dots,K\\
    U_{i,j}=0, &&\textup{$i\ge K+1$ or $j\ge K+1$}\\
      u\in\Re^n_+, U\in\mathcal{S}^n_{\ge 0},
\end{array}
\right.
\end{equation}
where $\mathcal{S}^n_{\ge 0}$ is the set of $n$-by-$n$ (entry-wise) nonnegative symmetric matrices.
In the following construction, we use the relationship between $Y$ and $U$ that $Y=PUP^\intercal$ for some $P\in\mathcal{P}_n$. (That is, the entries of $Y$ and $U$ equal up to row permutations and the corresponding column permutations.)
We impose the constraints  \eqref{basicvalid} as follows:
\begin{subequations}
\label{eq:strfwdconst2}
\begin{empheq}[left={\empheqlbrace\,}]{align}
    &\textup{trace}(U)\le 1\label{eq:trUle1}\\
    &\mathbbm{1}^\intercal U\mathbbm{1} = \mathbbm{1}^\intercal Y\mathbbm{1}.\label{eq:sumUeqsumY}
\end{empheq}
\end{subequations}
The nonconvex relationships $X=xx^\intercal, Y=yy^\intercal$, and $U=uu^\intercal$ can be relaxed using Schur complements as
\begin{equation}
\label{eq:schurcomp}
    \begin{bmatrix}
    X&x \\x\tr&1
    \end{bmatrix}\succeq 0, \quad
    \begin{bmatrix}
    Y&y \\y\tr&1
    \end{bmatrix}\succeq 0, \quad
    \begin{bmatrix}
    U&u \\u\tr&1
    \end{bmatrix}\succeq 0.
\end{equation} 
By constraint $X\succeq xx^\intercal$, it holds that $X\succeq 0$; hence, $\diag(X)\ge 0$. 
Therefore, the constraint $\trc(U)=\trc(Y)$ 
can be replaced with
    \begin{equation}
        \label{eq:equalTr}
        \textup{trace}(U)=\textup{trace}(X).
    \end{equation}
We next present modeling details for the majorization constraints using the arguments presented earlier.
First, the majorization relationship $u\ge_m y$ is represented as
\begin{equation}
\label{eq:umajy}
    \left\{\begin{array}{lll}
    \sum_{i=1}^j u_i\ge jr_j + \sum_{j=1}^n t_{ij},&& j=1, \dots, n-1\\
    y_i\le t_{ij}+r_j, &&i=1,\dots,n, \:\:j=1,\dots,n-1\\
    r\in\mathbb{R}^{n-1},\:\: t\in\mathbb{R}^{n\times (n-1)}_+
    \end{array}
    \right.
\end{equation}
as mentioned in \eqref{plargestdual}.  
Even though the modeling techniques using \eqref{eq:transpqr} and \eqref{eq:dualtransportationpqr} can potentially derive several inequalities, we only present certain representative inequalities in the proposed SDP relaxation for the sake of exposition.
First, the row sum majorization $R^U\ge_m R^Y$ and the diagonal majorization $\textup{diag}(U)\ge_m\textup{diag}(Y)$ are represented as
\begin{equation}
    \label{eq:RUmajRY}
        \left\{\begin{array}{lll}
        R^U_i = \sum_{j=1}^n U_{ij}, \:\:
        R^Y_i = \sum_{j=1}^n        Y_{ij}&&i=1,\dots,n\\
        \sum_{i=1}^j R^U_i\ge jr^R_j + \sum_{j=1}^n t^R_{ij},&& j=1, \dots, n-1\\
        R^Y_{i}\le t^R_{ij}+r^R_j, &&i=1,\dots,n, \:\:j=1,\dots,n-1\\
        r^R\in\Re^{n-1},\:\: t^R\in\Re^{n\times (n-1)}_+,\:\:
        R^U, R^Y\in\Re^n
        \end{array}
    \right.
\end{equation}
    and
\begin{equation}
        \label{eq:diagUmajdiagY}
        \left\{\begin{array}{lll}
        \sum_{i=1}^j U_{ii}\ge jr^D_j + \sum_{j=1}^n t^D_{ij},&& j=1, \dots, n-1\\
        X_{ii}\le t^D_{ij}+r^D_j, &&i=1,\dots,n, \:\:j=1,\dots,n-1\\
        r^D\in\Re^{n-1},\:\: t^D\in\Re^{n\times (n-1)}_+,
        \end{array}
        \right.
    \end{equation}
respectively.
Lastly, we use \eqref{eq:dualtransportationpqr} as follows to underestimate $\mathbbm{1}^\intercal (U^{pq})\mathbbm{1}$
for $p\in\{1,\dots,n\}$ and $q\in\{1,\dots,n\}$,
where for a matrix $A\in\mathbb{R}^{m\times n}$ 
and indices $p\le m$ and $q\le n$, $A^{pq}$ is defined as the $p$-by-$q$ submatrix whose sum of entries is maximum:
 \begin{equation}    
 \label{eq:ULtransportation}
 \left\{
 \begin{array}{ll}
        \multicolumn{2}{l}{\displaystyle\sum_{i=1}^p\sum_{j=1}^q U_{ij}\ge  \bar{q}\sum_{i=1}^n\alpha^{pq}_i+\bar{p}\sum_{j=1}^q \beta^{pq}_j+\bar{p}\bar{q}\gamma^{pq}+\sum_{i=1}^n\sum_{j=1}^n \delta^{pq}_{ij},}\\
        &\qquad p\in\{1,\dots,K\},q\in\{p,\dots,K\}\vspace{2mm}\\
        \alpha^{pq}_i+\beta^{pq}_j+\gamma^{pq}+\delta^{pq}_{ij}\ge Y_{ij},&\qquad 
        p\in\{1,\dots,K\}, q\in\{p, \dots,K\}, i,j\in\{1,\dots,n\}\\
        \multicolumn{2}{l}{\alpha\in\Re^{\frac{K(K-1)}{2}\times n}_+, \beta\in\Re^{\frac{K(K-1)}{2}\times n}_+, \gamma\in\Re^{\frac{K(K-1)}{2}}_+, \delta\in\Re^{\frac{K(K-1)}{2}\times n\times n}_+}
\end{array}
\right.
\end{equation}
with $\bar{p}=\left\{\begin{array}{ll}p& \textup{if $p\le K-1$}\\n&\textup{if $p=K$}\end{array}\right.$ and $\bar{q}=\left\{\begin{array}{ll}q& \textup{if $q\le K-1$}\\n&\textup{if $q=K$}\end{array}\right.$.

The proposed SDP relaxation, which we refer to as the \emph{submatrix relaxation}, is
\begin{equation}
\label{eq:majsdp}
\begin{array}{rl}
\max & \trc (\Sigma X)\\ \st & \eqref{eq:absx}-\eqref{eq:diagUmajdiagY}, \eqref{eq:ULtransportation}.
\end{array}
\end{equation}

\begin{theorem}
All constraints in \eqref{eq:dassdp} are implied by  \eqref{eq:absX}, \eqref{eq:trUle1}, \eqref{eq:sumUeqsumY}, \eqref{eq:schurcomp}, and  \eqref{eq:equalTr}.
\label{thm:dAsimplied}
\end{theorem}

\begin{proof}{Proof.}
First, $\trc(X)=\trc(U)\le 1$, where the equality and the inequality directly follow from \eqref{eq:equalTr} and  \eqref{eq:trUle1}, respectively.
We next show that $\one^\intercal|X|\one\le K$ is implied.
Let $U_{K}$ be the upper-left $K$-by-$K$ submatrix of $U$ and let $\mathbbm{1}_K$ be the $K$-dimensional vector of ones.
Define $f:\mathbb{R}^K\rightarrow\mathbb{R}$ as $f(x):=x^\intercal U_{K,K} x$.
Since $U\succeq 0$, we have $U_{K,K}\succeq 0$, showing that $f$ is convex.
Furthermore, $f(\alpha x)=\alpha^2f(x)$ for any scalar $\alpha$.
Therefore,
\begin{equation}\label{eq:sumUbound}\begin{array}{ll}
\one^\intercal U\one &=\one_K^\intercal U_{K,K}\one_K= f(\one_K)=f\left(\sum_{i=1}^K e_i\right)=f\left(K\sum_{i=1}^K \frac{1}{K}e_i\right)\\
&=K^2f\left(\sum_{i=1}^K \frac{1}{K}e_i\right)\le K^2\frac{1}{K}\sum_{i=1}^K f(e_i)=K\trc(U)\le K,
\end{array}
\end{equation}
where the inequalities follows from the convexity of $f$ and \eqref{eq:trUle1}.
Therefore,
\begin{equation}
    \label{eq:sumUleK}
\one^\intercal |X|\one
=\one^\intercal |V-W|\one
\le\one^\intercal (V+W)\one
=\one^\intercal Y\one
=\one^\intercal U\one\le K,
\end{equation}
where the first two equalities and the first inequality follow from \eqref{eq:absX}, the third equality from \eqref{eq:sumUeqsumY}, and the last inequality from \eqref{eq:sumUbound}.
The positive semidefiniteness of $X$ follows from the first Schur complement condition in \eqref{eq:schurcomp}.
\hfill \Halmos
\end{proof}

\begin{remark}
\label{rmk:2step}
We present another relaxation of \eqref{eq:ULtransportation} which improves computational efficiency compared to \eqref{eq:ULtransportation}, albeit this approach provides a weaker bound.
First, we introduce variables $(SR)_{iq}$ to define the sum of $q$-largest components of $i^{\textrm{\scriptsize th}}$ row of $Y$ as follows:
\[
\left\{ \begin{array}{ll}
\displaystyle(SR)_{iq} \ge qr^b_{qi}+\sum_{j=1}^n t^b_{qij} & i=1,\dots,n, q=1,\dots,K\\
Y_{ij}\le t^b_{qij}+r^b_{qi}&i=1,\dots,n, j=1,\dots,n, q=1,\dots,K\\
t^b\ge 0\\
t^b\in\Re^K\times\Re^n\times\Re^n, r^b\in\Re^K\times\Re^n.
\end{array}
\right.
\]
Using a specific feasible solution of \eqref{eq:transpqr}, we relax \eqref{eq:ULtransportation} as follows: \[
\left\{\begin{array}{ll}
\displaystyle\sum_{i=1}^p\sum_{j=1}^q U_{ij}\ge pr^B_{pq}+\sum_{i=1}^n t^B_{pqi} & p=1,\dots,K, q=1,\dots,K\\
(SR)_{iq}\le t^B_{pqi}+r^B_{pq}&p=1,\dots,K, q=1,\dots,K, i=1,\dots,n\\
t^B\ge 0\\
t^B\in\Re^K\times\Re^K\times\Re^n, r^B\in\Re^K\times\Re^K
\end{array}
\right.
\]
We refer to this relaxation as the \emph{2-Step relaxation}.
\end{remark}

\subsection{Computational experiments for sparse PCA}
\label{subsec:sparsePCA}

We next report our computational results on sparse PCA.
First, we summarize the following standard result. 
\begin{proposition}\label{prop:SPCAform}
The following is a correct formulation for sparse PCA:
\begin{subequations}\label{eq:sparsePCA}
\begin{alignat}{3}
    &\max &\quad& \trc(\Sigma X)\\
    &\st &&\trc(X)\le 1\\
    &            &&X\succeq 0\\
    &            && \diag(X)\le z\label{eq:SPCAzconstraint}\\
    &            && \one^\intercal z = K\label{eq:zSum}\\
    &            && z\in \{0,1\}^n.\label{eq:zbinary}
\end{alignat}
\end{subequations}
Let $(X^*,z^*)$ be an optimal solution to \eqref{eq:sparsePCA} and $\sum_{i=1}^n \lambda_i x^i{x^i}^\intercal$ be an eigenvalue decomposition of $X^*$ so that each $x^i$ is a unit eigenvector. 
Then, for any $i'$ such that $\lambda_{i'} > 0$, $x^{i'}$ is an optimal solution to sparse PCA.
\hfill \Halmos
\end{proposition}
\long\def\proofSPCAformulation{
When $\Sigma = 0$ or $K=0$, the validity is trivial since $(X^*,z^*)=(0,\sum_{i=1}^K e_i)$ is a feasible solution with an objective value of $0$, which is also optimal. 

Therefore, assume that $\Sigma\ne 0$ and $K> 0$. 
Let $(X^*,z^*)$ be an optimal solution to \eqref{eq:sparsePCA}. 
We show that $\trc(X^*)=1$. 
In this case, there is a diagonal element $\Sigma_{ii} > 0$. Then, $e_i^\intercal \Sigma e_i = \Sigma_{ii} > 0$. 
Therefore, the optimal value is strictly positive and $\trc(X^*) > 0$. 
Assume by contradiction that $\trc(X^*) < 1$.
Then, consider $X' = \frac{1}{\trc(X^*)}X^*$.
Clearly, $X' \succeq 0$. 
If $z^*_i = 1$, then $X'_{ii}\le \trc(X') = z^*_i$. 
If $z^*_i = 0$, then $X'_{ii}=0$ since $X_{ii}=0$. 
Now, $\Sigma\circ X^* < \Sigma \circ X'$, a contradiction. 
Therefore, $\trc(X^*) = 1$. 

Now we use the eigenvalue decomposition to express $X^* = \sum_{i} \lambda_i {x^i}{x^i}^\intercal$ so that $\sum_i \lambda_i = 1$, $\lambda_i > 0$, and, for all $i$, $\|x^i\| = 1$. 
We show by contradiction that, for all $i$ such that $\lambda_i > 0$, we have $\Sigma\circ x^i{x^i}^\intercal = \Sigma\circ X^*$. 
Let $i$ be such that $\lambda_{i} > 0$. 
We first show that $(x^{i}{x^{i}}^\intercal,z^*)$ is feasible to \eqref{eq:sparsePCA}. 
We only need to check the feasibility to \eqref{eq:SPCAzconstraint} since the remaining constraints are simple to verify. 
In particular, if $z^*_{j} = 0$, we have $0\le \lambda_{i}{x^{i}_{j}}^2 = -\sum_{\bar{i}\ne i}\lambda_{\bar{i}} {x^{\bar{i}}_j}^2 \le 0$ since $X^*_{jj}=( \sum_{i} \lambda_i x^i {x^i}\tra)_{jj} = 0$, so that equality holds throughout and $x^{i}_{j}=0$. 
If $z^*_j = 1$, we have ${x^{i}_j}^2 = 1 - \sum_{j'\ne j} {x^{i}_{j'}}^2\le 1 = z^*_j$. 
Therefore, $(x^{i}{x^{i}}^\intercal, z^*)$ is feasible to \eqref{eq:sparsePCA}. 
Now, assume that $\Sigma\circ x^i{x^i}^\intercal < \Sigma\circ X^*$. 
Then, $\Sigma \circ X^* = \Sigma \circ (\sum_{\bar{i}}\lambda_{\bar{i}} {x^{\bar{i}}}{x^{\bar{i}}}^\intercal) 
=  \sum_{\bar{i}: \lambda_{\bar{i}} > 0}\lambda_{\bar{i}} \Sigma\circ {x^{\bar{i}}}{x^{\bar{i}}}^\intercal < \Sigma\circ X^*$, where the strict inequality is because 
$\sum_{\bar{i}:\lambda_{\bar{i}}>0}\lambda_{\bar{i}} = 1$, $\Sigma\circ {x^{\bar{i}}}{x^{\bar{i}}}^\intercal \le \Sigma \circ X^*$ and $\Sigma\circ {x^i}{x^i}^\intercal < \Sigma \circ X^*$. 
This yields a contradiction. 
Therefore, whenever $\lambda_i > 0$, $\card(x^i) \le K$, $\|x^i\|=1$, and $\Sigma\circ X^* = {x^i}^\intercal \Sigma x^i$,   proving that $x^i$ is an optimal solution to sparse PCA.
}

\begin{theorem}\label{thm:ipconstraint2}
The following constraints are valid for sparse PCA:
\begin{subequations}\label{eq:Tformulation}
\begin{alignat}{2}
    &Y_{ij}^2\le T_{ij}T_{ji}&\quad& i=1,\ldots,n, j=i+1,\ldots,n\label{eq:toneprime}\\
    &T_{ii} = Y_{ii} &&i=1,\ldots,n\label{eq:ttwo}\\
    &\sum_{j=1}^n T_{ij}= z_i && i=1,\ldots,n\label{eq:tthree}\\
    &\sum_{i=1}^n T_{ij} = KY_{jj} && j=1,\ldots,n\label{eq:tfour}\\
    & 0\le T_{ij}\le Y_{jj} &&i=1,\ldots,n, j=1,\ldots,n.\label{eq:tfive}
\end{alignat}
\end{subequations}
In particular, $T_{ij}$ may be regarded as a linearization of $z_iy_j^2$. Using that $Y_{ij}=Y_{ji}$ for all $(i,j)$, the above constraints imply 
\begin{enumerate}
    \item\label{itm:showtone} for all $(i,j)$, $T_{ij} + T_{ji}\ge 2Y_{ij}$;
    \item\label{itm:showYij} for all $i$, $Y_{ii}\le z_i$ and, for all $(i,j)$, with $i\ne j$, $2Y_{ij}\le z_i$;
    \item\label{itm:showipconstraint} for all $S\subseteq \{1,\ldots,n\}$, $\sum_{i,j\in S}Y_{ij}\le \sum_{i\in S}z_i$; 
    \item\label{itm:showBert} for all $i$, $\sum_{j=1}^n Y_{ij}^2 \le z_iY_{ii}$;
    \item \label{itm:showBertlin} for all $i$, $\sum_{j=1}^n Y_{ij}^2 \le T_{ii}$ as long as, for all $j'$, $Y_{ij'}^2\le Y_{ii}Y_{j'j'}$ or, more specifically, $Y\succeq 0$.
    \item \label{itm:showsumK} $\one^\intercal |X|\one\le K$, as long as \eqref{eq:absX} holds. 
    \item \label{itm:showsumx2x} $\sum_{j=1}^n X_{ij}^2\le z_iX_{ii}$ as long as \eqref{eq:absX}, and $\trc(Y)=\trc(X)$ hold. 
    \item \label{itm:showsumx2t} $\sum_{j=1}^n X_{ij}^2\le T_{ii}$ as long as \eqref{eq:absX} holds, $\trc(Y)=\trc(X)$, and $X\succeq 0$. 
\end{enumerate}
\end{theorem}
\begin{proof}{Proof.}
To show that the constraints are valid, we only need to show that $T_{ij}$ can be chosen to be $z_iy_j^2$. 
We may assume that for all $(i,j)$, $Y_{ij}=y_iy_j$.
Constraint \eqref{eq:toneprime} follows since for $i'\in \{i,j\}$, $y_{i'} = z_{i'}y_{i'}$ implies that $(y_iy_j)^2 \le z_iy_i^2z_jy_j^2$.  
Constraint \eqref{eq:ttwo} follows since $z_iY_{ii} = Y_{ii}$. 
Constraint \eqref{eq:tthree} follows since $\trc(Y)=1$ implies $\sum_{j=1}^n z_iY_{jj} = z_i$. Constraint \eqref{eq:tfour} is valid because \eqref{eq:zSum} implies that $\sum_{i=1}^n z_iY_{jj} = KY_{jj}$. 
Finally, \eqref{eq:tfive} follows because it relaxes $y_i^2=z_iy_i^2 = Y_{ii}$.

We now show the implications claimed.
First, we show Statement~\ref{itm:showtone}. 
To see this, observe that we can write \eqref{eq:toneprime} as $(2Y_{ij})^2 + (T_{ij}-T_{ji})^2 \le (T_{ij}+T_{ji})^2$. 
Then,
\begin{equation}\label{eq:showtone}
    2Y_{ij}\le \sqrt{(2Y_{ij})^2 + (T_{ij}-T_{ji})^2} \le \sqrt{T_{ij}+T_{ji}} = T_{ij}+T_{ji},
\end{equation}
where the first inequality is because $(T_{ij}-T_{ji})^2\ge 0$, the second inequality is because of \eqref{eq:toneprime}, and the equality is because, by \eqref{eq:tfive}, $T_{ij}$ and $T_{ji}$ are non-negative.
Second, we show Statement~\ref{itm:showYij}. Clearly, $Y_{ii} = T_{ii}\le z_i$, where the first equality is by \eqref{eq:ttwo} and the inequality follows from \eqref{eq:tthree}.
If $i\ne j$, we have:
\begin{equation}\label{eq:Yandz}
 2Y_{ij}\le T_{ij} + T_{ji}\le T_{ij} + Y_{ii} = T_{ij} + T_{ii} \le z_i,  
\end{equation}
where the first inequality is by Statement~\ref{itm:showtone}, the second inequality is because of \eqref{eq:tfive}, the equality is because of \eqref{eq:ttwo}, and the inequality because $i\ne j$ and \eqref{eq:tthree}.
Third, we show Statement~\ref{itm:showipconstraint}. 
This is because
\begin{equation}\label{eq:ipconstraintimplied}
0\le \sum_{i,j\in S} (T_{ij} - Y_{ij}) \le \sum_{i\in S}\sum_{j=1}^n T_{ij} - \sum_{i,j \in S}Y_{ij} = \sum_{i\in S} z_i - \sum_{i,j\in S}Y_{ij},
\end{equation}
where the first inequality is because of Statement~\ref{itm:showtone} and $Y_{ij}=Y_{ji}$, the second inequality is because of \eqref{eq:tfive}, and the equality is by \eqref{eq:tthree}.
Now, we show Statement~\ref{itm:showBert} as follows:
\begin{equation}
    \sum_{j=1}^n Y_{ij}^2\le \sum_{j=1}^n T_{ji}T_{ij}\le Y_{ii}\sum_{j=1}^n T_{ij}= z_iY_{ii},
\end{equation}
where the first inequality follows from \eqref{eq:toneprime}, the second inequality from \eqref{eq:tfive}, and the equality from \eqref{eq:tthree}.
To see Statement~\ref{itm:showBertlin}, observe that 
\begin{equation}\label{eq:Bertlin}
\sum_{j=1}^n Y_{ij}^2\le Y_{ii}\sum_{j=1}^n Y_{jj} = Y_{ii} = T_{ii},
\end{equation} where the first inequality is because, for all $j'$ we have $Y_{ij'}^2\le Y_{ii}Y_{j'j'}$, the first equality is because $\trc(Y)=1$ and the third equality is by \eqref{eq:ttwo}.
Next, we show Statement~\ref{itm:showsumK}. 
We write
\begin{equation*}
    \one^\intercal |X|\one =  \one^\intercal |V-W|\one \le \one^\intercal (|V|+|W|)\one = \one^\intercal Y \one \le \one^\intercal z = K,
\end{equation*}
where the second equality is because $Y=V+W$, $V\ge 0$, and $W\ge 0$, and the first inequality is because of \eqref{eq:ipconstraintimplied} with $S=\{1,\ldots,n\}$, and the last equality is because of \eqref{eq:zSum}.
To show Statement~\ref{itm:showsumx2x}, we write
\begin{equation}
    \sum_{j=1}^n X_{ij}^2 = \sum_{j=1}^n (V_{ij} - W_{ij})^2 \le \sum_{j=1}^n (V_{ij}+W_{ij})^2 = \sum_{j=1}^n Y_{ij}^2 \le z_{i}Y_{ii} = z_iX_{ii},
\end{equation}
where the first two equalities and first inequality are by \eqref{eq:absX}, the second equality is because of Statement~\ref{itm:showBert}. 
The last equality is because \eqref{eq:absX} implies $\diag(Y-X)\ge 0$. 
Together with $\trc(Y-X) = 0$, this  implies that, for all $i$, $Y_{ii} = X_{ii}$. 
Finally, the proof of Statement~\ref{itm:showsumx2t}
is similar to that for Statement~\ref{itm:showBertlin} where $Y$ is replaced with $X$, where we also utilize that $\trc(Y-X)=0$, $\diag(Y-X)\ge 0$, and \eqref{eq:ttwo} imply that $X_{ii}=T_{ii}$ for all $i$. 
\hfill\Halmos
\end{proof}

We use the following formulation, which we refer to as the \emph{$T$-formulation}, to find the optimal solution for sparse PCA.
We chose to develop this formulation around the diagonal relaxation obtained using constraints \eqref{eq:diagUmajdiagY} as it is simple to implement but, as we will show later, it is also strong.
\begin{alignat*}{3}
(T): &\max &\quad&\text{trace}(\Sigma X)\\
&\st&&
\eqref{eq:absX}, 
\eqref{eq:strfwdconst},\eqref{eq:strfwdconst2}, 
\eqref{eq:diagUmajdiagY},\eqref{eq:zSum}, \eqref{eq:zbinary},\eqref{eq:Tformulation}\\
&&&\trc(U) = \trc(Y) = \trc(X) = 1\\
&&&Y_{ij}^2\le Y_{ii}Y_{jj}& i=1,\ldots,n, j=i+1,\ldots,n\\
&&&X\succeq 0, U \succeq 0.
\end{alignat*}
To prove that this formulation is correct, we only need to show that \eqref{eq:SPCAzconstraint} is satisfied. 
To see this observe that $X_{ii} \le Y_{ii} \le z_i$, where the first inequality follows from \eqref{eq:absX}, and the second inequality from \eqref{eq:Yandz}. Instead of requiring that $Y\succeq 0$, we relax it so that $Y_{ij}^2\le Y_{ii}Y_{jj}$ for all $(i,j)$.
Notice that the constraints $Y_{ij}^2\le T_{ij}T_{ji}$ and $Y_{ij}^2\le Y_{ii}Y_{jj}$ can be written as SOCP constraints $\|(2Y_{ij}, T_{ji}-T_{ij})\|_2\le T_{ji}+T_{ij}$ and $\|(2Y_{ij}, Y_{ii}-Y_{jj})\|_2\le Y_{ii}+Y_{jj}$, respectively.
We refer to the relaxation obtained by dropping \eqref{eq:zbinary} as the \emph{$T$-relaxation}.

After the initial draft of this paper \cite{kim2019convexification},  the following relaxation for sparse PCA was proposed in \cite{bertsimas2020solving}:
\begin{subequations}
\begin{alignat}{3}
    &\max&\quad &\text{trace}(\Sigma X)\nonumber\\
    &\mbox{s.t.}&&\eqref{eq:absX},\eqref{eq:zSum}, \eqref{eq:zbinary}\nonumber\\
    &&&\trc(X) = 1\nonumber\\
    &&&Y_{ii} \le z_i&&i=1,\ldots,n\label{eq:BertYii}\\
    &&&2Y_{ij}\le z_i &\quad& i=1,\ldots,n,\:\: j=i+1,\ldots,n\label{eq:BertYij}\\
    &&&\sum_{j=1}^nX_{ij}^2\le z_iX_{ii}  &&i=1,\ldots,n\label{eq:BertCone}\\ \qquad 
    &&&\one^\intercal Y \one = K\nonumber\\
    &&&X\succeq 0.\nonumber
\end{alignat}
\end{subequations}
We refer to the relaxation as the \emph{$B$-relaxation}. 
Theorem~\ref{thm:ipconstraint2} shows that $B$-relaxation is dominated by $T$-relaxation since \eqref{eq:BertYii}-\eqref{eq:BertCone} are implied in the $T$-relaxation.

We remark that the relaxations we develop here do not intrinsically depend on $\trc(X) = 1$. 
In particular, when the constraint \eqref{eq:tthree} is replaced with $\sum_{j=1}^n T_{ij}\le z_i\trc(T)$ and the constraint $\trc(U)=\trc(Y)=\trc(X)=1$ is replaced with $\trc(U)=\trc(Y)=\trc(X)$ our relaxations can be used more generally. More specifically, they are valid whenever we seek a rank-one matrix $X$, that is symmetric and positive-definite, and is such that only $K$ rows and columns have non-zero values.

We next define the notations used in the computational experiments.
We refer to the relaxation obtained by dropping constraints  \eqref{eq:diagUmajdiagY} and \eqref{eq:ULtransportation} from \eqref{eq:majsdp} as the \emph{rowsum relaxation}. 
Also, we refer to the relaxation obtained by dropping constraints  \eqref{eq:RUmajRY} and \eqref{eq:ULtransportation} from \eqref{eq:majsdp} as the \emph{diagonal relaxation}.
We denote the optimal value of the $D$-, $B$-, rowsum, diagonal, 2-step, submatrix, and $T$-relaxation by $z^*_D, z^*_B, z^*_{rowsum}, z^*_{diag}, z^*_{2step}, z^*_{submat}$, and $z^*_T$, respectively.
We report test results for these relaxations in Tables~\ref{table:pitpropstext} and \ref{table:random}.

We use \texttt{CVX} \cite{cvx, gb08} version 2.2 or \texttt{YALMIP} \cite{Lofberg2004} version R20210331 to solve SDPs in the experiments.
\texttt{MOSEK} \cite{aps2017mosek} version 9.2.47 was selected as the SDP solver in both cases.
To measure the relative tightness of a relaxation when compared to \eqref{eq:dassdp}, we calculate \emph{gap closed} as
\[
\left(\frac{z^*_{D}-z^*_{SDP}}{z^*_D-z^*}\right)\times 100.
\]
where $z^*_{SDP}$ is the target SDP relaxation on which we calculate the gap closed. 
Here, $z^*$ denotes the optimal value of sparse PCA and we obtain this value by solving the T-formulation of sparse PCA, $(T)$, to optimality using the \texttt{bnb} solver of \texttt{YALMIP} with \texttt{MOSEK} version 9.1.9. 

In addition, we conducted experiments with an alternate formulation. This formulation is based on a compact extended formulation of the permutahedron proposed by Goemans \cite{goemans2015smallest}, where the construction and the size of the reformulation depend on the choice of a sorting network. While the optimal Ajtai–Koml\'{o}s–Szemer\'{e}di sorting network gives an extended formulation for the permutahedron with $\Theta(n\log n)$ variables and inequalities, it has limited practical value because the constant hidden in the $\Theta(\cdot)$ notation is very large. Instead, we used Batcher's bitonic sorting network that gives an extended formulation with $\Theta(n\log^2n)$ variables and inequalities. 
When it comes to the time comparison between these two reformulations, we use the diagonal relaxation among the aforementioned SDP relaxations because, as we show later, this relaxation produces competitive bounds and is simple to implement.
Since the formulation based on a sorting network is equivalent to the duality-based formulation based on \eqref{plargestdual}, the quality of bounds is the same, and we only compare their computation times. The result is summarized in Tables~\ref{table:pitpropsgoemans} and \ref{table:randomsgoemans}.

The experiments are performed with an Intel Core i5-10400 machine containing a 2.90GHz CPU, 32GB RAM, running Windows 10.


\subsubsection{\texttt{pitprops} problem}
The \texttt{pitprops} problem \cite{jeffers1967two} is one of the most commonly used problems for sparse PCA algorithms. The instance has 13 variables and 180 observations. Table~\ref{table:pitpropstext} shows the test results for cardinality $K=3,\dots,10$.

\begin{table}[H]
\small{}
\centering
\setlength\tabcolsep{0.25em}
\begin{tabular}{|r|c|c|c|r|r|c|r|r|c|r|r|}
\hline
\multirow{3}{*}{$K$} &
\multicolumn{1}{c|}{$T$-formulation}
 &
\multicolumn{1}{c|}{$D$-relaxation} & 
\multicolumn{3}{c|}{$B$-relaxation} & 
\multicolumn{3}{c|}{Rowsum relaxation} & 
\multicolumn{3}{c|}{Diagonal relaxation}\\
\cline{2-12}
& \multirow{2}{*}{$z^*$}& 
\multirow{2}{*}{$z^*_D$}& 
\multirow{2}{*}{$z^*_{B}$} & \multicolumn{1}{c|}{Time}& \multicolumn{1}{c|}{Gap}&
\multirow{2}{*}{$z^*_{rowsum}$} & \multicolumn{1}{c|}{Time}& \multicolumn{1}{c|}{Gap}&\multirow{2}{*}{$z^*_{diag}$}& \multicolumn{1}{c|}{Time}& \multicolumn{1}{c|}{Gap}\\[-0.2em]
&
&
&& \multicolumn{1}{c|}{(sec)}&\multicolumn{1}{c|}{clsd}
&& \multicolumn{1}{c|}{(sec)}&\multicolumn{1}{c|}{clsd}
&& \multicolumn{1}{c|}{(sec)}&\multicolumn{1}{c|}{clsd}
\\
\hline
3&	2.4753
&	2.5218
&	2.4987&	0.20&	49.60&	2.5033&	0.26&	39.78&	2.4949&	0.24&	57.86\\
4&	2.9375
&	3.0172
&	2.9917&	0.22&	31.98&	2.9766&	0.27&	50.89&	2.9671&	0.25&	62.83\\
5&	3.4062
&	3.4581
&	3.4303&	0.22&	53.60&	3.4103&	0.27&	91.92&	3.4072&	0.26&	97.96\\
6&	3.7710
&	3.8137
&	3.7886&	0.22&	58.80&	3.7710&	0.26&	100.00&	3.7710&	0.27&	100.00\\
7&	3.9962
&	4.0316
&	3.9967&	0.22&	98.69&	3.9962&	0.24&	100.00&	3.9962&	0.26&	100.00\\
8&	4.0686
&	4.1448
&	4.0839&	0.22&	79.93&	4.0770&	0.24&	88.99&	4.0721&	0.29&	95.48\\
9&	4.1386
&	4.2063
&	4.1398&	0.20&	98.28&	4.1399&	0.26&	98.20&	4.1386&	0.28&	100.00\\
10&	4.1726
&	4.2186
&	4.1778&	0.23&	88.87&	4.1807&	0.26&	82.42&	4.1766&	0.26&	91.32\\
\hline
\multicolumn{3}{c|}{}
&Average&0.21
&69.97
&Average &0.26
& 81.53
&Average&0.26
&88.18
\\
\cline{4-12}
\multicolumn{12}{c}{}\\[-0.5em]
\cline{4-12}
\multicolumn{3}{c|}{} & 
\multicolumn{3}{c|}{2-step relaxation} & 
\multicolumn{3}{c|}{Submatrix relaxation} & 
\multicolumn{3}{c|}{$T$-relaxation}\\
\cline{4-12}
\multicolumn{3}{c|}{}&
\multirow{2}{*}{$z^*_{2step}$} & \multicolumn{1}{c|}{Time}& \multicolumn{1}{c|}{Gap}&
\multirow{2}{*}{$z^*_{submat}$} & \multicolumn{1}{c|}{Time}& \multicolumn{1}{c|}{Gap}&\multirow{2}{*}{$z^*_{T}$}& \multicolumn{1}{c|}{Time}& \multicolumn{1}{c|}{Gap}\\[-0.2em]
\multicolumn{3}{c|}{}
&& \multicolumn{1}{c|}{(sec)}&\multicolumn{1}{c|}{clsd}
&& \multicolumn{1}{c|}{(sec)}&\multicolumn{1}{c|}{clsd}
&& \multicolumn{1}{c|}{(sec)}&\multicolumn{1}{c|}{clsd}
\\
\cline{4-12}
\multicolumn{3}{c|}{}&	2.4753&	0.68&	100.00&	2.4753&	0.80&	100.00&	2.4753&	0.22&	100.00\\
\multicolumn{3}{c|}{}&	2.9477&	0.82&	87.15&	2.9477&	1.85&	87.15&	2.9375&	0.22&	100.00\\
\multicolumn{3}{c|}{}&	3.4062&	1.23&	100.00&	3.4062&	4.26&	100.00&	3.4062&	0.23&	100.00\\
\multicolumn{3}{c|}{}&	3.7710&	1.39&	100.00&	3.7710&	8.81&	100.00&	3.7710&	0.22&	100.00\\
\multicolumn{3}{c|}{}&	3.9962&	1.52&	100.00&	3.9962&	15.94&	100.00&	3.9962&	0.24&	100.00\\
\multicolumn{3}{c|}{}&	4.0721&	2.53&	95.48&	4.0721&	36.79&	95.47&	4.0686&	0.23&	99.99\\
\multicolumn{3}{c|}{}&	4.1386&	3.32&	100.00&	4.1386&	70.24&	100.00&	4.1386&	0.23&	100.00\\
\multicolumn{3}{c|}{}&	4.1766&	3.75&	91.32&	4.1766&	117.66&	91.32&	4.1733&	0.23&	98.49\\
\cline{4-12}
\multicolumn{3}{c|}{}&Average&1.91
&96.74
&Average &32.04
&96.74
&Average&0.23
&99.81
\\
\cline{4-12}
\end{tabular}
\caption{Optimal values and gaps closed for the test problem \texttt{pitprops}}
\label{table:pitpropstext}
\end{table}

Observe that the diagonal (resp. submatrix) relaxation reduces the gaps of \eqref{eq:dassdp} by more than 88\% (resp. 96\%), returning global optimal solutions for three (resp. five) problems. The $T$-relaxation attains the optima except for the instance $K=10$. On average, it reduces the gaps of \eqref{eq:dassdp} by 99.81\%.

For all computational times reported in Table~\ref{table:pitpropstext}, we used \eqref{eq:diagUmajdiagY} to model the majorization constraints.
However, as we had discussed earlier, these constraints can also be formulated using Batcher's bitonic sorting network, which only requires $\Theta(n\log^2 n)$ variables. 
Our primary intention here is to compare the formulation of the permutahedron based on  \eqref{eq:diagUmajdiagY} with that obtained via Batcher's bitonic sorting network. 
For this comparison, we use the diagonal relaxation and simplify it by dropping the requirement that $Y\succeq 0$. We refer to this simplified relaxation as \emph{Diagonal$'$-relaxation}. 
To differentiate the relaxation based on sorting-network, we will refer to it as \emph{Diagonal$'$-relaxation-sort}.
Clearly, both relaxations yield the same bound. 
Therefore, we only report on the solution times when comparing these relaxations. 
As $n$ and $K$ are small, no significant difference between the reformulations is observed. 

\begin{table}[H]
\small{}
\centering
\begin{tabular}{|c|c|c|c|c|c|c|c|c|c|}
\hline
$K$ & 3&4&5&6&7&8&9&10&Average\\
\hline
Diagonal$'$-relaxation
&0.23&0.22&0.23&0.22&0.23&0.23&0.23&0.22&0.23\\
\hline
Diagonal$'$-relaxation-sort
&0.20&0.23&0.23&0.22&0.22&0.21&0.22&0.23&0.22\\
\hline
\end{tabular}
\caption{Computation time (in seconds) comparison with the diagonal relaxation for pitprops problems}
\label{table:pitpropsgoemans}
\end{table}

\subsubsection{Experiments with randomly generated matrices}
\label{subsec:experiment_random}

We next report test results for randomly generated covariance matrices.
Random matrices are generated using the following procedure.
First, we choose a random integer $m\in\{1,\dots,n\}$ for the number of nonzero eigenvalues of the matrix by setting $m = \lceil nU\rceil$ where $U$ is randomly drawn from the uniform distribution $\mathcal{U}(0,1)$.
Second, we generate $m$ random vectors $v_i\in\mathbb{R}^n\sim \mathcal{N}(0,I_n)$, for $i=1,\dots,m$ for rank-1 matrices.
Third, we generate $m$ positive random eigenvalues $\lambda_i\sim \mathcal{U}(0,1)$, for $i=1,\dots,m$.
Finally, we construct the desired random covariance matrix as $\Sigma = \sum_{i=1}^m\lambda_i v_i v_i^\intercal$.


The tests are performed for problems with size $n\in\{5+5i\mid i=1,\dots,9\}$, and the  cardinality $K=\text{round}(n/6)$ is chosen to reflect the motivation of sparse principal components analysis to produce sparse vectors, where $\text{round}(x)$ represents the integer closest to $x$ ({\it i.e.}, $\text{round}(x) = \lfloor x \rfloor$ as long as the fractional part of $x$ is strictly less than $0.5$ and $\lceil x\rceil$ otherwise). 
For each $n$ and the associated $K$, 30 instances generated from random covariance matrices are tested. The reported computation times and gaps closed are the averages for the 30 instances.
In Table~\ref{table:random}, we present average gap closed for the computation times in seconds and gap closed of the relaxations. 
The computational results show that our SDP relaxations improve the gaps of the SDP relaxation \eqref{eq:dassdp} significantly.

\begin{table}[H]
\centering
\small{}
\begin{tabular}{|c|c|r|r|r|r|r|r|r|r|r|r|r|r|}
\hline
\multirow{3}{*}{$n$}&\multirow{3}{*}{$K$}
& \multicolumn{2}{c|}{$B$-relaxation}
& \multicolumn{2}{c|}{Rowsum rlxn.}
& \multicolumn{2}{c|}{Diag. rlxn.} 
& \multicolumn{2}{c|}{2-step rlxn.}
& \multicolumn{2}{c|}{Submat. rlxn.}
& \multicolumn{2}{c|}{$T$-relaxation} 
\\
\cline{3-14}
&  
& \multicolumn{1}{c|}{Gap}&\multicolumn{1}{c|}{Time}
& \multicolumn{1}{c|}{Gap}&\multicolumn{1}{c|}{Time}
& \multicolumn{1}{c|}{Gap}&\multicolumn{1}{c|}{Time}
& \multicolumn{1}{c|}{Gap}&\multicolumn{1}{c|}{Time}
& \multicolumn{1}{c|}{Gap}&\multicolumn{1}{c|}{Time}
& \multicolumn{1}{c|}{Gap}&\multicolumn{1}{c|}{Time}\\[-.2em]
&
&\multicolumn{1}{c|}{Clsd}&\multicolumn{1}{c|}{(sec)}
&\multicolumn{1}{c|}{Clsd}&\multicolumn{1}{c|}{(sec)}
&\multicolumn{1}{c|}{Clsd}&\multicolumn{1}{c|}{(sec)}
&\multicolumn{1}{c|}{Clsd}&\multicolumn{1}{c|}{(sec)}
&\multicolumn{1}{c|}{Clsd}&\multicolumn{1}{c|}{(sec)} 
&\multicolumn{1}{c|}{Clsd}&\multicolumn{1}{c|}{(sec)}\\
\hline
10&	2&	66.18&	0.37&	81.28&	0.34&	88.02&	0.35&	93.65&	0.35&	93.65&	0.34&	100.00&	0.36\\
15&	3&	58.27&	0.37&	71.98&	0.36&	81.98&	0.36&	91.15&	0.39&	91.15&	0.40&	99.80&	0.37\\
20&	3&	58.48&	0.36&	68.63&	0.44&	82.83&	0.41&	90.61&	0.53&	90.61&	0.55&	99.90&	0.39\\
25&	4&	51.28&	0.37&	68.01&	0.63&	79.40&	0.55&	86.43&	0.94&	86.44&	1.13&	98.95&	0.50\\
30&	5&	43.76&	0.39&	54.83&	1.11&	72.26&	0.85&	84.01&	1.83&	84.01&	3.02&	97.57&	0.75\\
35&	6&	44.18&	0.47&	55.36&	2.04&	69.19&	1.59&	81.31&	3.42&	81.30&	6.23&	97.98&	1.22\\
40&	7&  39.74&	0.52&	53.26&	3.58&	65.70&	2.62&	81.09&	6.79&	81.09&12.63&	98.05&	1.33\\
45&	8&  32.52&  0.65&	51.06&	5.63&	65.35&	4.44&	77.60&	11.23&	77.60&22.79&	94.81&	2.09\\
50&	8 & 25.92&	0.94&	40.61&	9.47&	54.94&	7.54&	70.93&	18.62&	70.93&	34.73&	94.29&	3.43\\
\hline
\multicolumn{2}{|c|}{Average}&46.70&	0.49&	60.56&	2.62&	73.30&	2.08&	84.09&	4.90&	84.09&	9.09&	97.93&	1.16\\
\hline
\end{tabular}
\caption{Test results for sparse PCA with $n\in\{10,15,\dots,40, 45, 50\}$ and $K=\text{round}(n/6)$}
\label{table:random}
\end{table}

Among the relaxations in Table~\ref{table:random}, the $T$-relaxation yields the tightest bound on our instances. 
We remark that the $T$-relaxation bound improves when the constraints for the 2-step relaxation and/or the submatrix relaxation are also imposed. 
However, we do not report on the performance of these relaxations since they are more complex and more computationally expensive to solve.
For larger dimensional instances, we choose the $B$-relaxation and the $T$-relaxation to compare. 
For $n\in\{5+5i\mid i=10,\dots,19\}$, we generate 30 random covariance matrices for each $n$ and summarize the results of the comparison in Table~\ref{table:B_vs_T} and Figure~\ref{fig:B_vs_T}, where the computation times are in minutes.

\begin{table}[H]
\centering
\begin{tabular}{|c|c|r|r|r|r|r|r|r|r|r|r|}
\hline
&$n$ & \multicolumn{1}{c|}{55}&\multicolumn{1}{c|}{60}&\multicolumn{1}{c|}{65}&\multicolumn{1}{c|}{70}&\multicolumn{1}{c|}{75}&\multicolumn{1}{c|}{80}&\multicolumn{1}{c|}{85}&\multicolumn{1}{c|}{90}&\multicolumn{1}{c|}{95}&\multicolumn{1}{c|}{100}\\
\cline{2-12}
&$K$ & \multicolumn{1}{c|}{9}&\multicolumn{1}{c|}{10}&\multicolumn{1}{c|}{11}&\multicolumn{1}{c|}{12}&\multicolumn{1}{c|}{13}&\multicolumn{1}{c|}{13}&\multicolumn{1}{c|}{14}&\multicolumn{1}{c|}{15}&\multicolumn{1}{c|}{16}&\multicolumn{1}{c|}{17}\\
\hline
\multirow{2}{*}{$B$-relaxation} 
& Gap Clsd&32.88&29.28&17.18&19.07&18.32&17.01&11.82&13.56&10.68&11.61\\
\cline{2-12}
& Time (min)&0.02&0.03&0.05&0.07&0.09&0.13&0.17&0.23&0.30&0.40\\
\hline
\multirow{2}{*}{$T$-relaxation} 
& Gap Clsd&96.19&91.93&87.11&90.97&84.79&87.12&85.44&79.70&79.97&76.65\\
\cline{2-12}
& Time (min)&0.09&0.14&0.21&0.31&0.45&0.66&0.90&1.17&1.60&2.14\\
\hline
\end{tabular}
\caption{Comparison between the $B$-relaxation and the $T$-relaxation with $n\in\{55,60,\dots,95,100\}$ and $K=\text{round}(n/6)$}
\label{table:B_vs_T}
\end{table}

\begin{figure}[H]
    \centering
    \begin{tikzpicture}[scale=0.7]
	\begin{axis}[
		height=8cm,
		width=15cm,
		xlabel=$n$,
		ylabel=Gap Reduced (\%),
		legend style={at={(0.08,0.25)},anchor=west,legend cell align=left}, 
	]
	\addplot[gray, thick, dashed, forget plot] coordinates {
	(10,0)
	(100,0)
	};
	
	\addplot[blue, thick, mark=*, line legend] coordinates {
(10, 100.00)
(15, 99.80)
(20, 99.90)
(25, 98.95)
(30, 97.57)
(35, 97.98)
(40, 98.05)
(45, 94.81)
(50, 94.29)
(55, 96.19)
(60, 91.93)
(65, 87.11)
(70, 90.97)
(75, 84.79)
(80, 87.12)
(85, 85.44)
(90, 79.70)
(95, 79.97)
(100, 76.65)
	};

	\addplot[red, thick, mark=square*, line legend] coordinates {
(10, 66.18)
(15, 58.27)
(20, 58.48)
(25, 51.28)
(30, 43.76)
(35, 44.18)
(40, 39.74)
(45, 32.52)
(50, 25.92)
(55, 32.88)
(60, 29.28)
(65, 17.18)
(70, 19.07)
(75, 18.32)
(80, 17.01)
(85, 11.82)
(90, 13.56)
(95, 10.68)
(100, 11.61)
	};		

	\legend{$T$-relaxation,$B$-relaxation};
	\end{axis}
\end{tikzpicture}
    \caption{Comparison between the $B$-relaxation and the $T$-relaxation with $n\in\{10,15,\dots,95,100\}$ and $K=\text{round}(n/6)$}
    \label{fig:B_vs_T}
\end{figure}
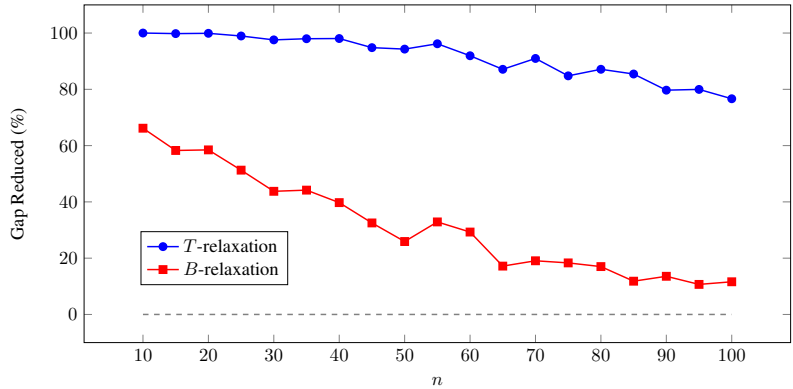

We next show that replacing \eqref{eq:diagUmajdiagY} with a sorting network helps reduce the solution time for the Diagonal relaxation by comparing the Diag$'$-relaxation and Diag$'$-relaxation-sort on the synthetic data sets.
For this demonstration, we consider dimensions $n\in\{10i\mid i=1,\dots,15\}$ and choose $K=\text{round}(n/6)$.
For each choice of $n$ and the associated $K$, $30$ randomly generated instances are tested;
see Table~\ref{table:randomsgoemans} and Figure~\ref{fig:dual_vs_sorting} for a comparison of the computation times (in minutes). Our results show that the sorting network based formulation reduces the computational time for the relaxation.

\begin{table}[ht]
\centering
\begin{tabular}{|c|c|r|r|r|r|r|r|r|r|}
\hline
&$n$ & \multicolumn{1}{c|}{10} & \multicolumn{1}{c|}{20} & \multicolumn{1}{c|}{30} & \multicolumn{1}{c|}{40} & \multicolumn{1}{c|}{50} & \multicolumn{1}{c|}{60} & \multicolumn{1}{c|}{70} & \multicolumn{1}{c|}{80}\\
\hline
\multirow{3}{*}{\begin{tabular}{c}Diagonal$'$-\\[-.2em]relaxation\end{tabular}  }
& \# vars &832&3261&7289&12917&20146&28974&39402&51431\\
\cline{2-10}
& \# cons &384&1469&3254&5739&8924&12809&17394&22679\\
\cline{2-10}
& Time (min) & 0.00&0.00&0.01&0.01 &0.04&0.10&0.22&0.45\\
\hline
\multirow{3}{*}{\begin{tabular}{c}Diagonal$'$-\\[-.2em]relaxation-\\[-.2em]sort\end{tabular}  }  
& \# vars &848&3087&6125&11451&16890&23528&34214&43253\\
\cline{2-10}
& \# cons &411&1452&2737&5226&7511&10296&15437&19222\\
\cline{2-10}
& Time (min) &0.00&0.00&0.01&0.01&0.03&0.07&0.17&0.36\\
\hline
\multicolumn{10}{c}{}\\
\cline{1-9}
&$n$ & \multicolumn{1}{c|}{90} & \multicolumn{1}{c|}{100} & \multicolumn{1}{c|}{110} & \multicolumn{1}{c|}{120} & \multicolumn{1}{c|}{130} & \multicolumn{1}{c|}{140} & \multicolumn{1}{c|}{150} & \multicolumn{1}{c}{}\\
\cline{1-9}
\multirow{3}{*}{\begin{tabular}{c}Diagonal$'$-\\[-.2em]relaxation\end{tabular} }  
& \# vars &65059&80287&97116&115544&135572&157201&180429& \multicolumn{1}{c}{}\\
\cline{2-9}
& \# cons &28664&35349&42734&50819&59604&69089&79274& \multicolumn{1}{c}{}\\
\cline{2-9}
& Time (min) &0.90&1.62&2.6184&4.30&6.48&11.33&13.93& \multicolumn{1}{c}{}\\
\cline{1-9}
\multirow{3}{*}{\begin{tabular}{c}Diagonal$'$-\\[-.2em]relaxation-\\[-.2em]sort\end{tabular}  }  
& \# vars &53491&64929&77568&91406&113676&129915&147353& \multicolumn{1}{c}{}\\
\cline{2-9}
& \# cons &23507&28292&33577&39362&50383&57168&64453& \multicolumn{1}{c}{}\\
\cline{2-9}
& Time (min) &0.69&1.24&2.06&3.31&5.12&8.21&11.06& \multicolumn{1}{c}{}\\
\cline{1-9}
\end{tabular}
\caption{Computation time comparison between the Diagonal$'$-relaxation and Diagonal$'$-relaxation-sort for sparse PCA with random covariance matrix, $n=10,20,\dots,150$ and $K= \text{round}(n/6)$}
\label{table:randomsgoemans}
\end{table}

\begin{figure}[ht]
    \centering
    \begin{tikzpicture}[scale=0.7]
	\begin{axis}[
		height=8cm,
		width=15cm,
		xlabel=$n$,
		ylabel=Time (min),
		legend style={at={(0.05,0.85)},anchor=west,legend cell align=left} 
	]
	\addplot[red, thick, mark=square*] coordinates {
(10, 0.0035)
(20, 0.0040)
(30, 0.0066)
(40, 0.0148)
(50, 0.0372)
(60, 0.0956)
(70, 0.2164)
(80, 0.4500)
(90, 0.8973)
(100, 1.6211)
(110, 2.6184)
(120, 4.0323)
(130, 6.4785)
(140, 11.3198)
(150, 13.9251)
	};		
	\addlegendentry{Diagonal$'$-relaxation}

	\addplot[blue, thick, mark=*] coordinates {
(10, 0.0035)
(20, 0.0039)
(30, 0.0055)
(40, 0.0123)
(50, 0.0275)
(60, 0.0653)
(70, 0.1681)
(80, 0.3555)
(90, 0.6938)
(100, 1.2431)
(110, 2.0609)
(120, 3.3095)
(130, 5.1209)
(140, 8.2164)
(150, 11.0626)
	};
	\addlegendentry{Diagonal$'$-relaxation-sort}
	\end{axis}
\end{tikzpicture}
\caption{Computation time comparison between the Diagonal$'$-relaxation and Diagonal$'$-relaxation-sort for sparse PCA with random covariance matrix, $n=10,20,\dots,150$ and $K= \text{round}(n/6)$}
\label{fig:dual_vs_sorting}
\end{figure}
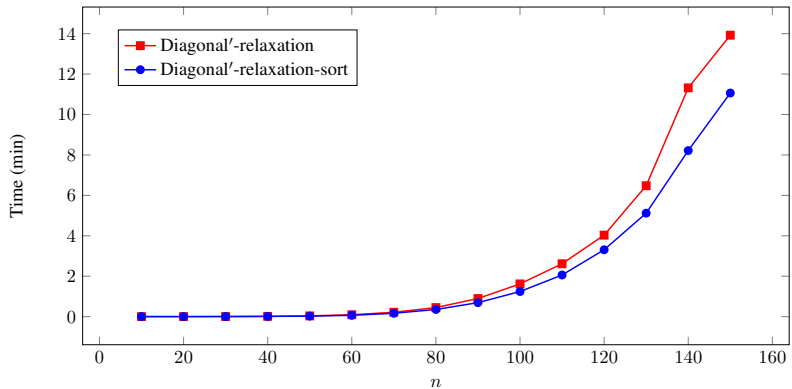


\section{Conclusion}
In this paper, we present an explicit convex hull description of permutation-invariant sets and applications of the results to various important sets/functions in optimization.
The construction of the convex hull is based on the fact that a permutation-invariant set is a union of permutahedra with generating vectors in $\Delta=\{x\in\mathbb{R}^n\mid x_1\ge\dots\ge x_n\}$ and the convex hull of this union can be described in closed-form.
We then applied this result to derive various convexification results.
First, we presented an extended formulation for the convex hull of permutation-invariant norm balls constrained by a cardinality requirement.
Second, we convexified sets of matrices that are characterized using functions of their singular values.
Third, we derived convex/concave envelopes of various nonlinear functions and convex hulls of sets defined using nonlinear functions when bounds for variables are congruent. 
Fourth, we studied sets of rank-one matrices whose generating vectors lie in a permutation-invariant set.
We use majorization inequalities in the space of generating vectors to construct valid inequalities for the convex hull in the matrix space.
As a motivating example, we construct tight semidefinite programming relaxations for sparse principal component analysis and report computational results that show that
our tightest relaxation reduces more than 99.8\% (resp. 97.9\%) of the gaps for the pitprops data set (resp. synthetic data sets with the dimensions up to $n=50$).
The concept of permutation-invariance can be used to study a variety of other sets, including those arising from logical requirements in 0-1 mixed integer programming; see \cite{kim2019convexification} for additional descriptions.

\section*{Acknowledgments.}
The second and third authors wish to acknowledge the support from NSF CMMI Award \#1727989 and \#1917323, respectively.

\bibliographystyle{plain}
\bibliography{majorization}

\newpage
\appendix

\section{Table of Notations}
The following table lists some notations used multiple times throughout the paper.
\begin{table}[H]
\centering
\small{}
\begin{tabular}{|c|l|l|}
\hline
Notation & \multicolumn{1}{c|}{Definition} & \multicolumn{1}{c|}{Defined in}\\
\hline
$\card(x)$&the number of nonzero components of $x$ & Section~\ref{sec:intro}\\
$\|\cdot\|_1$ & $l_1$-norm of vectors & Section~\ref{sec:intro}\\
$\mathcal{M}^{m,n}(\mathbb{R})$&the set of $m$-by-$n$ real matrices&Section~\ref{sec:intro}, Section~\ref{subsec:convexifysingular}\\
$\|\cdot\|_{sp}$&spectral norm of matrices&Section~\ref{sec:intro}, Section~\ref{subsec:convexifysingular}\\
$\Delta_\pi$&$\{x\in\mathbb{R}^n\mid x_{\pi(1)}\ge\dots\ge x_{\pi(n)}\}$&Section~\ref{sec:intro}\\
$\Delta^n (=\Delta)$&$\{x\in\mathbb{R}^n\mid x_1\ge\dots\ge x_n\}$&Section~\ref{sec:intro} \eqref{eq:simplex_n}\\
 $\conv(X)$ & the convex hull of the set $X$ & Section~\ref{sec:intro}\\
 $\mathcal{P}_k$ & $k$-by-$k$ permutation matrices & Section~\ref{sec:mainresults}\\
 $x_{[i]}$ & $i$th largest element of $x$ & Section~\ref{sec:mainresults} Definition~\ref{def:maj}\\
$x\ge_m y$ & $x$ majorizes $y$ & Section~\ref{sec:mainresults} Definition~\ref{def:maj}\\
$x\ge_{wm} y$ & $x$ weakly majorizes $y$ from below & Section~\ref{sec:mainresults} Definition~\ref{def:maj}\\
$\|\cdot\|_s$ & an arbitrary sign- and permutation-invariant norm of vectors & Section~\ref{sec:sparsity} \eqref{eq:sparsity}\\
$N^K_{\|\cdot\|_s}$ & $\{x\in\mathbb{R}^n\mid \|x\|_s\le 1,\card(x)\le K\}$ &Section~\ref{sec:sparsity} \eqref{eq:sparsity}\\
$B_s(r)$ & $\{x\in\mathbb{R}^n\mid \|x\|_s\le r\}$ & Section~\ref{sec:sparsity} \\
$\Delta^n_+(=\Delta_+)$ & $\Delta^n\cap\mathbb{R}^n_+$ & Section~\ref{sec:sparsity} \eqref{eq:delta_n_+}\\
$x_{\Delta^n} (=x_{\Delta})$ & $(x_{\Delta^n})_i:=x_{[i]}$ & Section~\ref{sec:sparsity} \eqref{eq:delta_n_+}\\
$x_{\Delta^n_+} (=x_{\Delta_+})$ & $(x_{\Delta^n_+})_i:=|x|_{[i]}$ & Section~\ref{sec:sparsity} \eqref{eq:delta_n_+}\\
$\text{vert}(P)$ & the set of vertices of the polyhedron $P$ & Section~\ref{sec:sparsity} Proposition~\ref{prop:cnorm_compact}\\
$\text{int}(X)$ & the set of interior points of $X$ & Section~\ref{sec:sparsity} Proposition~\ref{prop:cnorm_compact}\\
 $\|\cdot\|_c$ & the norm corresponding to the convex body $N^K_{\|\cdot\|_s}$& Section~\ref{sec:sparsity} \\
$s(x)$ & $s(x)_i=\frac{\sum_{j=i}^n|x|_{[j]}}{K-i+1},\:\: i=1,\dots,K, \:\: s(x)_0=s(x)_{K+1}=\infty$ & Section~\ref{sec:sparsity} \\
 $i_x$ & $\arg\min\{s(x)_i\mid i=1,\dots,K\}$ & Section~\ref{sec:sparsity}\\
$\delta(x)$ & $s(x)_{i_x}$ & Section~\ref{sec:sparsity} \\
$u(x)$ & $u(x)_i:=\left\{\begin{array}{ll}
|x|_{[i]}, & i\in\{1,\dots,i_x-1\}\\
\delta(x), & i\in\{i_x, \dots,K\}\\
0, & \textup{otherwise}.
\end{array}\right.$ & Section~\ref{sec:sparsity}\\
$\|\cdot\|_K^{sp}$ & $K$-support norm ($K$-overlap norm) & Section~\ref{sec:sparsity}  \\
$\sigma(M)$ & the vector of singular value of the matrix $M$ & Section~\ref{subsec:convexifysingular}  \\
$\|\cdot\|_*$ & nuclear norm of matrices& Section~\ref{subsec:convexifysingular} \\
$\text{diag}(v)$ & the diagonal matrix whose diagonal is $v$& Section~\ref{subsec:convexifysingular} \\
$\mathcal{S}^p$ & the set of $p$-by-$p$ symmetric matrices& Section~\ref{subsec:convexifysingular} \\
$\mathcal{S}^p_+$ & the set of $p$-by-$p$ positive semidefinite matrices& Section~\ref{subsec:convexifysingular} \\
$\lambda(M)$ & the vector of eigenvalues of the matrix $M$& Section~\ref{subsec:convexifysingular} \\
$\conv_C(\phi)$ & the convex envelope of $\phi$ over $C$ & Section~\ref{sec:nonlinear} \\
$\phi|_X$ & $\phi|_X(x)=\phi(x)$ for any $x\in X$ and $+\infty$ otherwise & Section~\ref{sec:nonlinear} \\
$S(Z,a,b)$ & $\{(x,z)\in[a,b]^n\times\mathbb{R}^m\mid (x,z)\in Z\}$& Section~\ref{sec:nonlinear} \\
$X(Z,a,b,\{F_i\}_{i=1}^r)$ & $\left\{(x,z)\in [a,b]^n\times\R^m \mid (x,z)\in Z, x\in \bigcup_{i=1}^r F_i\right\}$& Section~\ref{sec:nonlinear}  \\
$\mathcal{M}^n$ & $\mathcal{M}^{n,n}(\mathbb{R})$ & Section~\ref{sec:rankone} \\
$M_S$ & $\left\{ (x,X)\in\mathbb{R}^n\times\mathcal{M}^n\mid X=xx^\intercal,  x\in S\right\}$ & Section~\ref{sec:rankone} \\
$\mathbbm{1}_n (=\mathbbm{1})$ & $n$-dimensional vector of ones& Section~\ref{sec:rankone}\\
$\mathbbm{1}_{n\times n}$ & $k$-by-$k$ matrix of  ones& Section~\ref{sec:rankone}\\
$\text{trace}(M)$ & the trace of the matrix $M$ & Section~\ref{sec:rankone} \eqref{basicvalid_a}\\
$\text{diag}(M)$ & the diagonal vector of the matrix $M$ & Section~\ref{sec:rankone}\\
$R^M$ & the vector of row sums of the matrix $M$ & Section~\ref{sec:rankone} \\
$\mathcal{M}^n_{\ge 0}$ & the set of $n$-by-$n$ non-negative matrices & Section~\ref{subsec:sdpspca} \\
$\mathcal{S}^n_{\ge 0}$ & the set of $n$-by-$n$ non-negative symmetric matrices & Section~\ref{subsec:sdpspca} \\
\hline
\end{tabular}
\caption{Table of notations}
\label{tab:notation}
\end{table}


\end{document}